\documentclass[twoside,11pt]{article}

\usepackage{blindtext}

\usepackage{jmlr2e}

\usepackage{lastpage}
\jmlrheading{23}{2022}{1-\pageref{LastPage}}{1/21; Revised 5/22}{9/22}{21-0000}{Author One and Author Two}

\usepackage{epsfig,epsf,fancybox}
\usepackage{amsmath,amssymb,amsfonts}
\usepackage{multirow}

\usepackage{natbib}

\usepackage{url}
\usepackage{lineno}
\usepackage{tikz}

\usepackage{array}
\usepackage{tabularx}

\usepackage{endnotes}

\usepackage[small, margin=1cm]{caption}
\usepackage{subcaption}
\usepackage{appendix}
\usepackage{color}
\definecolor{strcolor}{rgb}{0.6, 0.2, 0.6}
\definecolor{commentcolor}{rgb}{0.3125, 0.5, 0.3125}
\definecolor{keycol}{rgb}{0, 0, 1}

\usepackage{rotating}
\usepackage{fancyvrb}
\usepackage{tcolorbox}
\usepackage{bbm}
\usepackage{algorithm}
\usepackage{algpseudocode}
\usepackage{booktabs}
\usepackage{listings}

\usepackage{hyperref}
\usepackage[nameinlink,capitalize,noabbrev]{cleveref}

\usepackage{mathtools}

\usepackage[normalem]{ulem} 

\newcommand{\bm}[1]{\boldsymbol{#1}}

\newcommand{\argmax}{\arg\max}
\newcommand{\lb}{\lambda}
\newcommand{\va}{{\mathbf{a}}}
\newcommand{\vb}{{\mathbf{b}}}

\newcommand{\vd}{{\mathbf{d}}}

\newcommand{\vr}{{\mathbf{r}}} 

\newcommand{\vu}{{\mathbf{u}}}
\newcommand{\vv}{{\mathbf{v}}}

\newcommand{\vx}{{\mathbf{x}}}
\newcommand{\vy}{{\mathbf{y}}}
\newcommand{\vz}{{\mathbf{z}}}

\newcommand{\vtheta}{{\bm{\theta}}}

\newcommand{\cG}{{\mathcal{G}}}

\newcommand{\cS}{{\mathcal{S}}}

\newcommand{\cU}{{\mathcal{U}}}

\newcommand{\cX}{{\mathcal{X}}}
\newcommand{\cY}{{\mathcal{Y}}}
\newcommand{\cZ}{{\mathcal{Z}}}

\newcommand{\EE}{\mathbb{E}} 
\newcommand{\RR}{\mathbb{R}} 
\newcommand{\zz}{^{\top}} 
\newcommand{\vzero}{\mathbf{0}} 

\newcommand{\dist}{\mathrm{dist}}    
\newcommand{\dom}{{\mathrm{dom}}} 

\DeclareMathOperator*{\Argmax}{Arg\,max}

\newcommand{\bc}{\begin{center}}
\newcommand{\ec}{\end{center}}

\newcommand{\bdm}{\begin{displaymath}}
\newcommand{\edm}{\end{displaymath}}

\newcommand{\beq}{\begin{equation}}
\newcommand{\eeq}{\end{equation}}

\newcommand{\bfl}{\begin{flushleft}}
\newcommand{\efl}{\end{flushleft}}

\newcommand{\bt}{\begin{tabbing}}
\newcommand{\et}{\end{tabbing}}

\newcommand{\beqn}{\begin{eqnarray}}
\newcommand{\eeqn}{\end{eqnarray}}

\newcommand{\beqs}{\begin{align*}} 
\newcommand{\eeqs}{\end{align*}}  

\newtheorem{assumption}{Assumption}

\newcommand{\vgamma}{{\boldsymbol{\gamma}}}
\newcommand{\vxi}{{\boldsymbol{\xi}}}

\ShortHeadings{A stochastic smoothing framework for   {minEmax} problems}{Liu, Khan, Mancino-Ball and Xu}
\firstpageno{1}

\allowdisplaybreaks[4]

\begin{document}

\title{A stochastic smoothing framework for nonconvex  {minEmax} problems with applications to Wasserstein distributionally robust optimization}

\author{\name Wei Liu \email lwdsdqqb@gmail.com \\
       \addr 
       Department of Applied Mathematics, Hong Kong Polytechnic University, Hong Kong, China\\ 
       Institute for Math \& AI, Wuhan (IMAI), Wuhan university, China
       \AND
       \name Muhammad Khan \email khanm7@rpi.edu \\
       \addr Department of Mathematical Sciences\\ Rensselaer Polytechnic Institute, Troy, NY, USA
       \AND
       \name Gabriel Mancino-Ball \email gabriel.mancino.ball@gmail.com \\
       \addr Department of Mathematical Sciences\\ Rensselaer Polytechnic Institute, Troy, NY, USA
       \AND
       \name Yangyang Xu \email xuy21@rpi.edu \\
       \addr Department of Mathematical Sciences\\ Rensselaer Polytechnic Institute, Troy, NY, USA}

\editor{My editor}

\maketitle

\begin{abstract}
	We study a class of stochastic nonsmooth optimization problems in which an outer variable minimizes the expectation of a pointwise maximum. This minimization--expectation--maximization (minEmax) problem arises in Wasserstein distributionally robust optimization and adversarially robust training, and it cannot in general be reformulated as a finite-dimensional minimax problem when the underlying distribution is not empirical. We propose a stochastic smoothing proximal gradient method based on log-mean-exp smoothing of the value function. Under compactness and Lipschitz-type assumptions, we present nonasymptotic analysis in terms of Goldstein stationarity and show that every almost-sure cluster point generated by our method is a Clarke stationary point; by Clarke regularity, such a point is also directional stationary for the original problem. Numerical experiments on newsvendor, robust regression, and adversarially robust learning problems show that the proposed method is competitive with existing baselines.
\end{abstract}


	\section{Introduction}\label{sec:intro}
	This paper studies the following minimization--expectation--maximization (minEmax) problem: 
\begin{equation}\label{eq:minmax}
	\min_{\vy\in\RR^{m_1}}\left\{g(\vy):=\varphi(\vy)+\EE_{\vx\sim\mathbb{P}}\left[\max_{\vz\in\mathcal{Z}} \Psi(\vy,\vz;\vx)\right]\right\}.\tag{P}
\end{equation}
Here  $\mathcal{Z}\subseteq\RR^{m_2}$ is a nonempty compact set, $\mathbb{P}$ is a probability distribution on $\mathcal{X}\subseteq\RR^{m_3}$, and independent samples can be drawn from $\mathbb{P}$. 
For each $\vx\in\mathcal{X}$, define
\begin{equation}\label{eq:def-Phi-i}
	\Phi(\vy;\vx):=\max_{\vz\in\mathcal{Z}}\Psi(\vy,\vz;\vx).
\end{equation}
Throughout the paper, we make the following assumptions on problem~\eqref{eq:minmax}. 
\begin{assumption}\label{ass:problemsetup}
	The following statements hold. 
	
	$\mathrm{(i)}$ For every $\vx\in\mathcal{X}$,  
	$  \Psi(\cdot,\cdot;\vx)$ is continuous,  
	$ \Psi(\cdot,\vz;\vx)$ is continuously differentiable, and
	$ \nabla_{\vy}\Psi(\cdot,\cdot;\vx)$ is continuous.
	
	$\mathrm{(ii)}$ The function $\varphi:\RR^{m_1}\mapsto\RR$ is 
	proper closed convex,  and its proximal mapping can be easily evaluated.  
	
	$\mathrm{(iii)}$ The sets $\dom(\varphi)$ and $\mathcal{Z}$ are nonempty and compact.
\end{assumption}

\begin{assumption} 
	\label{ass:compact1}
	There exists $l_{\Psi}>0$ such that, for all $\vx\in\cX$, 
    and 
	$(\vy,\vz)\in\operatorname{dom}(\varphi)\times\mathcal{Z}$,
	$
	\|\nabla_{\vy}\Psi(\vy,\vz;\vx)\|\le l_{\Psi}.
	$
	In addition,  
	the map $\vz\mapsto\Psi(\vy,\vz;\vx)$ is $l_{\Psi}$-Lipschitz continuous on
	$\mathcal{Z}$. 
\end{assumption}
Assumption~\ref{ass:problemsetup} allows nonsmooth 
convex regularizers, such as indicator functions of compact convex sets. The feasible set $\mathcal{Z}$ may be a connected compact set or a finite discrete set; in either case it is compact. By Berge's maximum theorem~\citep{berge1959espaces}, Assumption~\ref{ass:problemsetup} implies that $\Phi(\cdot;\vx)$ is well-defined and continuous for every $\vx\in\mathcal{X}$.
 In addition,
Assumption~\ref{ass:problemsetup}(i) imposes differentiability only with respect to the outer variable $\vy$; no differentiability with respect to $\vz$ is required, while Assumption~\ref{ass:compact1} gives   boundedness of the $\vy$-gradient and Lipschitz continuity in the inner variable.
 
Problem~\eqref{eq:minmax} encompasses a wide variety of modern machine learning applications, particularly those involving challenging inner maximization problems that may be nonsmooth in $\vz$ or may involve a finite feasible set.
Prominent examples include
adversarially robust training (e.g., \cite{goodfellow2014explaining, huang2015learning, madry2018towards, liang2023optimization})
and distributionally robust optimization (DRO)~(e.g., \cite{kuhn2024distributionallyrobustoptimization,rahimian2019distributionally}); see Section~\ref{sec:appli} for more details.

When $\mathbb{P}$ is the uniform distribution on a finite dataset $\{\vx_1,\ldots,\vx_n\}$, problem~\eqref{eq:minmax} can be written as
\[
\min_{\vy}\;\max_{\vz_1,\ldots,\vz_n\in\mathcal{Z}}
\left\{\varphi(\vy)+\frac{1}{n}\sum_{i=1}^n\Psi(\vy,\vz_i;\vx_i)\right\}.
\]
Minimax methods (see, e.g.,~\citep{xu2023unified,nouiehed2019solving,jiang2022optimality,liu2021first,li2022nonsmooth,lin2025twotimescale,thekumparampil2019efficient})  can then be applied in principle, although the reformulation may be expensive for a large $n$; see, e.g., \citep{liang2023implications}. 
For a general distribution $\mathbb{P}$, an analogous reformulation would require an infinite-dimensional measurable decision rule $\vx\mapsto \vz(\vx)$, rather than finitely many variables $\vz_1,\ldots,\vz_n$.  This fundamental distinction underscores the necessity for developing 
a new optimization framework that can efficiently handle the inherent complexity of expectation-over-maximization problems with general probability distributions. 

	We therefore view problem~\eqref{eq:minmax} as nonsmooth minimization of an expected-value objective. 
	Though a subgradient method is standard for nonsmooth stochastic minimization, 
	its convergence analysis typically requires access to high-quality (stochastic) subgradient evaluations. 
	For problem~\eqref{eq:minmax}, accurate inner objective values are not sufficient for this purpose: Example~\ref{exam:2} in Appendix~\ref{sec:appenE} shows that a nearly optimal inner maximizer can yield a vector far from the Clarke subdifferential of the outer objective. We avoid this issue by applying the log-mean-exp (LME) smoothing~\citep{shapiro2021lectures} to $\Phi(\cdot;\vx)$.
	Specifically, for every $\vx\in\mathcal{X}$ and $\mu>0$, define the LME smoothing function $\widetilde{\Phi}$ as follows:
	\begin{equation}\label{eq:smoothphi}
		\widetilde{\Phi}(\vy,\mu;\vx): =\mu \log \mathbb{E}_{\vz\sim \zeta} [e^{\Psi(\vy,\vz;\vx)/\mu}] =\mu\log\int_{\mathcal{Z}}\exp(\Psi(\vy,\vz;\vx)/\mu)\,\zeta(d\vz),
	\end{equation}
	where $\zeta$ denotes a fixed probability measure on $\mathcal{Z}$. In particular, if $\mathcal{Z}$ is finite, $\zeta$ corresponds to the uniform measure; if $\mathcal{Z}$ is a connected compact set (with positive Lebesgue measure), $\zeta$ is the normalized Lebesgue measure on $\mathcal{Z}$ (i.e., $\zeta(\cZ)=1$).
	 
	 By leveraging the properties of the LME smoothing function, we effectively mitigate the computational challenges discussed above and develop a stochastic smoothing proximal gradient method (SSPG). Our analysis in Section~\ref{sec:algorithm} establishes that every almost sure accumulation point of the sequence generated by SSPG is a Clarke stationary point of problem~\eqref{eq:minmax}. Given that the primal objective function is Clarke regular under our assumptions, such accumulation points are additionally directional stationary~\citep{pang2017computing, cui2018composite, cui2021modern}. {We further establish the iteration complexity results of SSPG to find an approximate stationary point under two different notions.}
	 
 The remainder of this section proceeds as follows: we first describe several relevant applications structured as~\eqref{eq:minmax}, then summarize our main technical contributions, and finally introduce necessary notation and definitions.

\subsection{Applications}\label{sec:appli}
Two representative applications with the structure of~\eqref{eq:minmax} are Wasserstein distributionally robust optimization (WDRO) and adversarially robust training. 

\subsubsection{WDRO}
The \textbf{WDRO problem}~\citep{kuhn2024distributionallyrobustoptimization,shafieezadeh2019regularization} can be formulated as
\begin{equation}
	\label{eq:model}
	\min_{\vtheta\in \Theta} \max _{\widehat{\mathbb{P}} \in \mathcal{B}_\delta(\mathbb{P})} \mathbb{E}_{\vx \sim \widehat{\mathbb{P}}}\left[\ell(\vtheta,\vx)\right],
\end{equation}
where $\mathbb{P}$ is the nominal distribution on the uncertainty set $\mathcal{Z}$, $\widehat{\mathbb{P}}$ ranges over alternative distributions in the Wasserstein ambiguity set, $\ell:\RR^{m_1}\times \RR^{m_2}\mapsto \RR$ is a given {(possibly nonconvex)} loss function, $\Theta$ is a closed convex set, 
and the ambiguity set 
$\mathcal{B}_\delta(\mathbb{P})$ is defined as the \mbox{$\delta$-ball} in the $p$-th Wasserstein distance centered at $\mathbb{P}$, i.e., 
\mbox{$\mathcal{B}_\delta(\mathbb{P})=\{\widehat{\mathbb{P}} \in \mathcal{P}(\mathcal{Z})\mid d_{\mathcal{W}_p}(\widehat{\mathbb{P}}, \mathbb{P}) \leq \delta\}.$} Here, $\mathcal{P}(\mathcal{Z})$ is the space of probability distributions $\widehat{\mathbb{P}}$ supported on $\mathcal{Z}$ with $\mathbb{E}_{\vx \sim \widehat{\mathbb{P}}} [\|\vx\|] <\infty$,  and the $p$-th Wasserstein distance \citep{kantorovich1958space} between distributions $\mathbb{Q}_1, \mathbb{Q}_2 \in \mathcal{P}(\mathcal{Z})$  is defined by 
\begin{align*}
&d_{\mathcal{W}_p}\left(\mathbb{Q}_1, \mathbb{Q}_2\right):= \\ &\inf \left\{\left(\int_{\mathcal{Z}\times \mathcal{Z}}\left\|\vz_1-\vz_2\right\|_p^p \mathrm{\Pi}\left(\mathrm{d} \vz_1, \mathrm{d} \vz_2\right)\right)^{1/p} \bigg| \begin{array}{l}
	\mathrm{\Pi} \text { is a joint distribution of } \vz_1 \text { and } \vz_2 \text{ with}\\
	\text {marginal distributions } \mathbb{Q}_1 \text { and } \mathbb{Q}_2 \text {, respectively}
\end{array}\right\}.
\end{align*}
For a fixed decision $\vtheta$, the inner problem in~\eqref{eq:model} is the worst-case risk:
\begin{equation}
	\label{eq:worstrisk}
	\max_{\widehat{\mathbb{P}} \in \mathcal{B}_\delta(\mathbb{P})} \mathbb{E}_{\vx \sim \widehat{\mathbb{P}}}\left[\ell(\vtheta,\vx)\right].
\end{equation}
Under some mild conditions~\citep{yue2022linear, gao2023distributionally}, the worst-case risk in~\eqref{eq:worstrisk}  is finite and attainable, and furthermore,
strong duality holds. 
The dual problem of~\eqref{eq:worstrisk} is given by~\citep{rahimian2019distributionally,yue2022linear,gao2023distributionally}
\begin{equation}
	\label{eq:worstriskdual}
	\min_{\lb \geq 0} \lambda \delta^p+\mathbb{E}_{\vx \sim \mathbb{P}}\left[\max _{\vz \in \mathcal{Z}}\{\ell(\vtheta,\vz)-\lambda d(\vx,\vz)\}\right]. 
\end{equation}
Here, $d: \RR^{m_2}\times \RR^{m_2}\mapsto \RR$ denotes the transport cost of the Wasserstein metric of order $p\in \mathbb{N}_+$ defined as $d(\vz_1,\vz_2):=\|\vz_1-\vz_2\|_p^p$ and $\lb\in\RR_+$ is the Lagrangian multiplier with respect to the inequality constraint $d_{\mathcal{W}_p}^p(\widehat{\mathbb{P}}, \mathbb{P}) \leq \delta^p$. By strong duality, the optimal values of the aforementioned two models are the same. 
The following lemma gives the compact interval for the dual multiplier; its proof is deferred to Appendix~\ref{appen:A}.
		\begin{lemma} 
		Consider problem \eqref{eq:worstriskdual} at a given parameter $\vtheta$, with $d(\vz_1,\vz_2)=\|\vz_1-\vz_2\|_p^p$, $p\ge 1$ and $\delta>0$. Assume the loss function $\ell(\vtheta,\cdot)$ is $L$-Lipschitz continuous for all $\vtheta\in\Theta$. Then, any optimal solution $\lambda^\star$ of problem \eqref{eq:worstriskdual} has an upper bound 
		\[
		\lambda^\star \le L\,C_{p,m_2}\,\delta^{-(p-1)},\quad \text{where}\quad
		C_{p,m_2}=\sup_{\vv \in \RR^{m_2}, \vv\neq \vzero}\frac{\|\vv\|_2}{\|\vv\|_p}=\begin{cases}
			1, & 1\le p\le 2,\\[0.6ex]
			m_2^{\frac12-\frac1p}, & p>2.
		\end{cases}
		\]
	\end{lemma} 
	Hence, problem~\eqref{eq:worstriskdual} can be restricted to $\lambda\in[0,B_\lambda]$ without loss of optimality, where \mbox{$B_{\lb}= L\,C_{p,m_2}\,\delta^{-(p-1)}$.} Substituting this dual representation into~\eqref{eq:model},  and explicitly imposing the constraint $\lb\in[0,B_{\lb}]$, we arrive at the following reformulation:
\begin{equation}
	\label{eq:model2}
	\min_{\vtheta\in \Theta, \lb\in[0,B_{\lb}]} g(\vtheta, \lb):= \lambda \delta^p+\mathbb{E}_{\vx \sim \mathbb{P}}\left[\max _{\vz \in \mathcal{Z}}\{\ell(\vtheta,\vz)-\lambda d(\vx,\vz)\}\right].
\end{equation}
This problem has the form~\eqref{eq:minmax}, with $\vy=(\vtheta,\lambda)$ and
$
\Psi((\vtheta,\lambda),\vz;\vx)=\ell(\vtheta,\vz)-\lambda d(\vx,\vz).
$
It satisfies Assumption~\ref{ass:problemsetup} when $\Theta$, $[0,B_\lambda]$, and $\mathcal{Z}$ are compact. 


\subsubsection{Adversarially robust training}

In a classification setting, the \textbf{adversarially robust training} \citep{goodfellow2014explaining, huang2015learning, madry2018towards, liang2023optimization} problem takes the form of 
\begin{equation}\label{eq:adt}\min _{{\vtheta\in \Theta}}\frac{1}{n}\sum_{i=1}^n \max _{\vz \in \Delta({\vx}_i)} \ell\left({\bar{\vx}_i}, f(\vtheta, \vz)\right).\end{equation}
Here, $\{(\vx_i,{\bar{\vx}}_i)\}_{i=1}^n
\subset \RR^{m_2}\times \RR^{m_3}$ are sampled ordered pairs of data points and their corresponding labels (respectively),
$\Delta({\vx})=\left\{\vz\in\RR^{m_2}\big| d\left({\vx}, \vz\right) \leq \epsilon\right\}$ represents a set of feasible perturbations applied to a given data point $\vx$ based on some distance function $d$, 
$f:\RR^{m_1}\times \RR^{m_2}\mapsto \RR^{m_3}$ is a  prediction function, and $\ell:~\RR^{m_3}\times \RR^{m_3}\mapsto \RR$ is a {(possibly nonconvex)} loss function. 
In the context of computer vision, this model seeks to identify the worst-case perturbations for images within a prescribed radius \citep{huang2015learning, madry2018towards}.
Let $\mathbb{P}_n$ be the empirical distribution on the datasets.  Then $$\frac{1}{n}\sum_{i=1}^n \max _{\vz \in \Delta({\vx}_i)} \ell\left({\bar{\vx}_i}, f(\vtheta, \vz)\right) = \EE_{(\vx,\bar\vx)\sim \mathbb{P}}\max _{\vz \in \Delta({\vx})} \ell\left({\bar{\vx}}, f(\vtheta, \vz)\right),$$ and thus~\eqref{eq:adt} can be formulated into~\eqref{eq:minmax}.  {This formulation satisfies Assumption~\ref{ass:problemsetup} whenever $\Theta$ is compact, $\ell(\cdot,\cdot)$ is continuous, and $\ell(\bar\vx,\cdot)$ and
	$f(\cdot,\vz)$ are smooth for all $\bar\vx$ and $\vz$.}
{Indeed, adversarially robust training can be viewed as a dual problem of the $\infty$-WDRO problem 
	whose ambiguity set is characterized by the Wasserstein distance with $p=+\infty$; see~\citep[Equation (D)]{gao2024wasserstein}. 
	Notice that the dual formulation in problem \eqref{eq:model2} applies for $p\in [1,\infty)$ but not for $p=\infty$. 
}

\subsection{Contributions}
\label{sec:contri}
This paper develops a stochastic smoothing framework for the minEmax problem~\eqref{eq:minmax}, a nonsmooth expected-value problem in which the outer variable is to minimize the expectation of a pointwise maximum. This formulation covers, among other examples, Wasserstein distributionally robust optimization (WDRO) after dualization of the worst-case risk and adversarially robust training. Our main contributions are summarized below.

\begin{enumerate}
 
	\item 
	We introduce the log-mean-exp (LME) smoothing in~\eqref{eq:smoothphi}
	for each 
    $\vx\in\cX$. We prove that $\widetilde\Phi(\cdot,\mu;\vx)$ is a valid smoothing function of $\Phi(\cdot;\vx)$ defined in \eqref{eq:def-Phi-i}, preserves the basic Lipschitz bound in~$\vy$, and has a $\vy$-gradient Lipschitz constant of order $L_\Psi+l_\Psi^2/\mu$. We also derive approximation-gap bounds for both finite~$\cZ$ and connected compact~$\cZ$. These bounds make it possible to translate stationarity of the smoothed objectives into Goldstein stationarity of problem~\eqref{eq:minmax}.
	
	\item 
	Using the LME smoothing, we propose SSPG, a  stochastic smoothing proximal gradient method. At iteration~$k$, SSPG uses a stochastic estimator of $\nabla_{\vy}\EE_{\vx\sim\mathbb P}[\widetilde\Phi(\vy,\mu_k;\vx)]$ and performs one proximal gradient step with respect to the nonsmooth convex term~$\varphi$, while the smoothing parameter~$\mu_k$ is kept fixed or decreased according to a prescribed rule. We establish a nonasymptotic stationarity bound for a randomly selected iterate. In particular, SSPG obtains an $\epsilon$-scaled stationary point in expectation with iteration complexity $O(\epsilon^{-3})$ and sample complexity $O(\epsilon^{-5})$. After converting scaled stationarity to Goldstein stationarity, SSPG obtains an $(O(\epsilon),O(\epsilon))$-Goldstein stationary point with iteration complexity $\widetilde O(\epsilon^{-4})$ and sample complexity $\widetilde O(\epsilon^{-6})$. With a summable refinement schedule, every almost-sure cluster point of the resulting sequence of approximate stationary points is Clarke stationary, and hence directional stationary, for~\eqref{eq:minmax}.
	
	\item 
	For WDRO, we prove an explicit global upper bound on the optimal dual multiplier in the Wasserstein worst-case-risk reformulation. This allows the dual variable to be restricted to a compact interval without loss of optimality, making the compact-domain assumptions compatible with WDRO analysis and implementation. We then test SSPG on newsvendor, robust regression, $\infty$-Wasserstein/adversarially learning, and adversarially robust deep-learning tasks. The experiments show that SSPG is competitive with the tested baselines and, in several settings, improves the accuracy--runtime tradeoff relative to nonsmoothed or fixed-entropic alternatives.
\end{enumerate}

 \subsection{Definitions and Notations}\label{sec:defi}
Let $\|\cdot\|$ be the 2-norm of a vector.   
We use $[M]$ to represent $\{1,2,\ldots,M\}$ for an integer~$M$. Given a compact set $\mathcal{Z}$, denote $|\mathcal{Z}|$ as the cardinality of $\cal{Z}$ if it is a finite set and volume of $\cal{Z}$ if it is a continuous set.
The indicator function is denoted as $\mathbb{I}.$
We refer to $\mathbf{1}_m$ as the $m$-dimensional vector of all ones.
We write $\operatorname{Proj}$, $\operatorname{Prox}$, and $\operatorname{dist}$ for the projection, proximal, and distance operators, respectively.
The \textit{normal cone} of a convex set $\mathcal{Y}$ at $\vy^*$ is defined as 
\mbox{$\mathcal{N}_{\mathcal{Y}}(\vy^*)=\{\vgamma:\langle \vgamma, \vy-\vy^*\rangle \leq 0,  \forall \vy \in \mathcal{Y}\}.$}  

Let \(h:\mathbb{R}^d\to\mathbb{R}\) be locally Lipschitz continuous.
We denote by $h'(\bar{\vy};\vd)$ the directional derivative of $h$ at $\bar{\vy}$ in direction $\vd$:
$$
	h^{\prime}\left(\bar{\vy} ; \vd \right)=\lim _{ t \downarrow 0} \frac{1}{t}(h(\bar{\vy}+t \vd)-h(\bar{\vy})).
$$
If $h^{\prime}\left(\bar{\vy} ; \vd\right)$ is well-defined for any unit vector $\vd$, we say $h$ is directionally differentiable at $\bar{\vy} $. 
It is known that if $h$ is piecewise differentiable and Lipschitz continuous, then it is directionally differentiable \citep{mifflin,cui2021modern}. 
Denote $h^{\circ}(\bar{ \vy} ; \vd)$ as the \textit{generalized directional derivative} at $\bar{\vy}$ along the direction~$\vd$~\citep{clarke1990optimization}, by 
$$h^{\circ}(\bar{ \vy} ; \vd):=\limsup _{{\vy \rightarrow \bar{ \vy}}, {t \downarrow 0}} \frac{1}{t}(h(\vy+t \vd)-h(\vy)).$$ 
In general, it holds $h^{\circ}(\bar{ \vy} ; \vd)\geq h^{\prime}\left(\bar{\vy} ; \vd \right)$~\citep{clarke1990optimization}, and $h^{\prime}\left(\bar{\vy} ; \vd \right)=h^{\circ}(\bar{ \vy} ; \vd)=\vd\zz\nabla h(\bar{ \vy})$ 
if $h$ is differentiable. 
We use $\partial h$ to denote the (Clarke) subdifferential~\citep[Section 1.2]{clarke1990optimization} of a continuous function $h$, i.e.,
$$
\partial h(\bar{\vy})
:=
\left\{
\vxi\in\RR^d:
\langle \vxi,\vd\rangle\le h^\circ(\bar{\vy};\vd)
\ \text{for all }\vd\in\RR^d
\right\}.
$$
 For
\(\mu\ge 0\), we define the Goldstein \(\mu\)-subdifferential~\citep{kornowski2024algorithm} of \(h\)
at~\(\vy\) by
$
\partial^{\mu} h(\vy)
:= \operatorname{conv} \left(
\bigcup_{\vu\in \mathbb{B}(\vy,\mu)} \partial h(\vu)
 \right),
$
where \(\mathbb{B}(\vy,\mu)=\{\vu:\|\vu-\vy\|\le \mu\}\).

\begin{definition}[Clarke regular {\citep[Definition 2.3.4]{clarke1990optimization}}]
	\label{def:clarkeregular}
	A locally Lipschitz continuous function $h$ is Clarke regular at $\bar{\vy}$ if, for every direction $\vd$, the directional derivative $h'(\bar{\vy};\vd)$ exists and
	$
	h'(\bar{\vy};\vd)=h^\circ(\bar{\vy};\vd).
	$
\end{definition}
From Definition~\ref{def:clarkeregular}, 
if $h$ is convex or  differentiable, then it is Clarke regular~\citep{clarke1990optimization} and directionally differentiable.

\begin{definition}[Stationarity]\label{def:sta}
	Let $h:\RR^d\to\RR$ be directionally differentiable at $\vy^*$.
	The point $\vy^*$ is called directional stationary for $\min_{\vy}h(\vy)$ if
	$
	h'(\vy^*;\vd)\ge 0  \text{ for all }\vd\in\RR^d.
	$
	If $h$ is locally Lipschitz, $\vy^*$ is called Clarke stationary if
	$
	\vzero\in\partial h(\vy^*).
	$
	For $\mu,\epsilon\ge0$, $\vy^*$ is called $(\mu,\epsilon)$-Goldstein stationary, in expectation,  if
	$ 
	\mathbb{E}\left[\operatorname{dist}\bigl(\vzero,\partial^{\mu} h(\vy^*)\bigr)^2\right]\le \epsilon^2.
	$ 
\end{definition} 
For a locally Lipschitz function, directional stationarity implies Clarke stationarity. If $h$ is Clarke regular, the converse also holds. Thus the two notions coincide for Clarke regular objectives. The $(\mu,\epsilon)$-Goldstein condition is a finite-accuracy relaxation: it aggregates Clarke subgradients in a $\mu$-neighborhood of the candidate point and requires the resulting convex hull to contain a vector of norm at most $\epsilon$. It is particularly suitable for nonasymptotic
analysis of nonsmooth nonconvex stochastic optimization, when exact Clarke
stationarity is   difficult to certify. As \(\mu\downarrow 0\) and
\(\epsilon\downarrow 0\), Goldstein stationarity recovers Clarke stationarity under   outer-semicontinuity conditions.

\begin{definition}[Smoothing function \citep{chen2012smoothing}]\label{def:smooth}
	Let $h: \mathbb{R}^m \mapsto \mathbb{R}$ be continuous. We call $\widetilde{h}: \mathbb{R}^m \times \mathbb{R}_{+} \mapsto \mathbb{R}$ a smoothing function of $h$, if for any fixed {$\mu>0$}, $\widetilde{h}(\cdot, \mu)$ is continuously differentiable, and for any $\bar{\vy}$, it holds $\lim_{\vy\rightarrow\bar{\vy},\mu\downarrow 0} \widetilde{h}(\vy,\mu) = h(\bar{\vy})$.
\end{definition}

Approximate stationarity conditions based on Goldstein subdifferentials have recently emerged as   tools in the nonasymptotic analysis of nonsmooth   nonconvex optimization. For instance, \citet{lin2022gradient} established connections between ``uniform smoothing" and Goldstein stationarity, and developed both deterministic and stochastic gradient-free methods with corresponding nonasymptotic convergence guarantees. 
The LME smoothing approach employed in this paper differs from the ``uniform smoothing" technique introduced by \citet{lin2022gradient}. Specifically, our method smooths the function~$\Phi$, which incorporates the maximum operator within an expectation-over-maximization objective, a scenario not covered by the framework of ``uniform smoothing''. 
Nevertheless, Goldstein stationarity serves a convenient tool in our analysis.

\subsection{Organization}
The remainder of the paper is organized as follows.  Section~\ref{sec:liter} provides a review of the related literature. In Section~\ref{sec:algorithm}, 
we introduce the SSPG framework for solving the problem~\eqref{eq:minmax}. 
We first construct a smoothing function, then give our proposed algorithmic framework, and finally provide the convergence analysis, {with proofs given in Appendix \ref{sec:proof}}.  
Numerical results are presented in Section~\ref{sec:numerical}. 
We conclude the paper 
in Section~\ref{sec:conclu}.

					\section{Related Work}
\label{sec:liter}
	In this section, we briefly survey several optimization approaches for 
	WDRO problems. Additionally, we review smoothing methods and the  LME smoothing function, which serve as foundational tools for the proposed framework.

\subsection{Existing Methods for Solving WDRO}
{As introduced in~\citep{rahimian2019distributionally}, a widely adopted strategy in the distributionally robust optimization (DRO) literature is to  solve problem~\eqref{eq:model2}. Standard minimax solvers are   applicable when $\mathbb{P}$ is an empirical distribution. However, for general distributions, practical algorithms typically 
	require additional structural conditions. Examples include: (i) discretization of the support set $\mathcal{Z}$~\citep{xu2018distributionally, chen2021decomposition, liu2021discrete, pflug2007ambiguity}; (ii) representation of $\ell(\vtheta,\cdot)$ as a finite maximum of concave functions under an $\ell_1$ transport cost~\citep{mohajerin2018data,gao2023distributionally}; (iii) log-loss models, such as robust logistic regression~\citep{li2019first, selvi2022wasserstein, shafieezadeh2015distributionally}; or (iv) strong concavity of $\ell(\vtheta,\cdot)-\lambda d(\cdot,\vx)$~\citep{blanchet2022optimal,sinha2018certifiable}. 
	Without these structural assumptions, solving WDRO problems computationally becomes significantly challenging.}

Recognizing the inherent challenges in addressing the WDRO problem, \citet{le2025unregularized} investigate a regularized version wherein entropic smoothing yields a sampled approximation to the original objective. They demonstrate convergence of approximate gradients toward Clarke subgradients of the unregularized WDRO objective as the regularization parameter tends toward zero. This entropic-regularization perspective has also influenced robust learning software development, notably in \citep{florian2026skwdro, liu2025dro}. In their analyses, the authors explicitly distinguish between regularized WDRO  problem and the original WDRO problem, underscoring that their methods do not fundamentally solve the original WDRO problem. 

 \citet{wang2021sinkhorn} innovatively proposes addressing a dual problem of Sinkhorn DRO (SDRO),  
which can be written as 
\begin{equation}
	\label{eq:smoothg}
	\min_{\vtheta\in \Theta, \lb\in[0,B_{\lb}]}g_{\mathrm{s}}(\vtheta, \lb) :=  \lambda \delta^p+  \lb\eta \mathbb{E}_{\vx \sim \mathbb{P}}\left[\log \mathbb{E}_{\vz \sim \zeta}\left[e^{(\ell(\vtheta,\vz) -\lb d(\vx,\vz)) / \lb\eta}\right]\right]
\end{equation}
for some $\eta>0$, where $B_{\lb}$ is the same as that in \eqref{eq:model2}, and $\zeta$ is a distribution supported on $\cZ$. This perspective connects naturally to regularized WDRO~\citep{le2025unregularized}, since Sinkhorn distance is itself a Wasserstein-type discrepancy based on entropic regularization.
The objective function in~\eqref{eq:smoothg} can be viewed as a smoothing function of~\eqref{eq:model2} by Definition~\ref{def:smooth} with $\eta $ as the smoothing parameter. 
Assuming $\ell(\cdot,\vz)$ is convex for all feasible~$\vz$, 
\citet{wang2021sinkhorn} develop a convergent triple-loop algorithm by skillfully combining a bisection method, a stochastic mirror descent method, and multilevel Monte-Carlo simulation.
However, they fix $\eta>0$. In addition, in their complexity result \citep[Theorem 3]{wang2021sinkhorn}, they assume that any optimal solution $(\vtheta^*,\lb^*)$ of problem~\eqref{eq:smoothg} satisfies $\lb^*\ge \overline{\lb}$ for some positive scalar $\overline{\lb}$.
This assumption circumvents significant computational challenges by preventing both the Lipschitz constant and the gradient Lipschitz constant from exploding if 
$\lb$ approaches~0 as the algorithm progresses,  
in which case, an exponential increase in the sampling complexity of their multilevel Monte-Carlo simulations would occur with respect to $\frac{1}{\lb}$ (see the proof in \citep[Proposition EC.4]{wang2021sinkhorn}). 
Hence, their algorithm does not solve the original WDRO problem.

	In contrast, when applied to WDRO, our proposed method for~\eqref{eq:minmax} smooths the function in~\eqref{eq:model2} by the {LME} function and solves 
	\begin{equation}
		\label{eq:smoothg2}
		\min_{\vtheta \in \Theta, \lb\in[0,B_{\lb}]} \widetilde{g}(\vtheta, \lb, \mu) := \lambda \delta^p + \mu \mathbb{E}_{\vx \sim \mathbb{P}}\left[\log \mathbb{E}_{\vz \sim \zeta}\left[e^{(\ell(\vtheta, \vz) -\lb d(\vx, \vz)) / \mu}\right]\right].
	\end{equation}
	In contrast to \citep{wang2021sinkhorn}, our formulation allows the dual variable $\lambda$ to approach zero.  The formulation in~\eqref{eq:smoothg2} 
	replaces \(\lb \eta\) in~\eqref{eq:smoothg} with \(\mu\).  However, this subtle difference in the formulation leads to a fundamentally different method  
	from the algorithm in \citep{wang2021sinkhorn}.
	First, \citep{wang2021sinkhorn}  
	solves a dual problem of SDRO with a  fixed  \(\eta\). In contrast, we  
	solve WDRO directly, requiring \(\mu \downarrow 0\) in our algorithmic framework. 
	This distinction changes the limiting problem and requires controlling the smoothing bias as $\mu\downarrow0$, while the gradient Lipschitz constants of the smoothed objectives typically grow as $\mu$ decreases.
	Second, \(\lb\) is a primal variable in problem~\eqref{eq:smoothg} and appears in the denominator of the exponential term. 
	This results in both the Lipschitz constant and the gradient Lipschitz constant of~$g_s$  with respect to~$\lb$ becoming unbounded as 
	$\lambda$ approaches zero, which potentially yields slow convergence and high complexity. 
	In contrast, the Lipschitz constant of $\widetilde{g}$  with respect to \(\lb\) is bounded, and the gradient Lipschitz constant of $\widetilde{g}$   with respect to \(\lb\) can be bounded  by $\frac{C}{\mu}$ for some positive scalar $C$. 

	\subsection{Smoothing Methods and the LME Smoothing Function}
	{ Smoothing methods provide a powerful and widely-used framework for solving nonsmooth optimization problems; we refer readers to foundational discussions and examples in~\citep{nesterov2007smoothing, chen2012smoothing, burke2013epi, burke2017epi,liu2022linearly,liu2025single}. Typically, smoothing methods involve three essential steps:
		(i) constructing smooth approximations $\{\widetilde{g}(\cdot,\mu)\}_{\mu>0}$ to the nonsmooth function $g$, parameterized by a smoothing parameter $\mu>0$;
		(ii) designing algorithms to solve the smoothed problems $\min_\vy \widetilde{g}(\vy,\mu_k)$ approximately for a sequence $\{\mu_k\}$; and
		(iii) analyzing the convergence behavior of solutions as $\mu_k\downarrow 0$.
		A theoretical challenge occurs as $\mu\downarrow 0$:  
		the gradients $\nabla\widetilde g(\cdot,\mu)$ are not Lipschitz continuous with constants bounded uniformly over $\mu\in(0,1]$; in many smoothing schemes the Lipschitz constant scales as $\Theta(1/\mu)$. 
		
		For deterministic nonsmooth optimization, the theoretical foundations, including epi-convergence, iteration complexity analysis, and   gradient bounds, are well-established. In contrast, the theory for stochastic smoothing methods remains underdeveloped. Significant challenges in the stochastic scenario include the intricate coupling between smoothing-induced bias (controlled by $\mu$) and variance in stochastic gradient estimators, and the rigorous justification of interchangeability between expectation and differentiation, requiring careful considerations of measurability and integrability. Recently, \citet{wang2023stochastic} analyzed convergence properties of stochastic smoothing methods under  convexity of $\Psi$, and   bounded variance of the stochastic  {gradient} estimators for $g(\cdot,\mu)$ for every fixed smoothing parameter $\mu>0$. Their analysis, however, does not establish convergence to Clarke stationary points almost surely, a critical theoretical guarantee in smoothing methods. Without these   restrictive conditions, our paper directly addresses this gap by providing such convergence guarantees (see Theorem~\ref{lem:scaled}).
		
		We now discuss the LME smoothing function given in~\eqref{eq:smoothphi}. 
		When the reference measure $\zeta$ is uniform on a finite set
		$\mathcal{Z}=\{\vz_1,\ldots,\vz_q\}$, the LME smoothing becomes
		$
		\widetilde\Phi(\vy,\mu;\vx)
		=
		\mu\log\left(\frac1q\sum_{j=1}^q
		\exp(\Psi(\vy,\vz_j;\vx)/\mu)\right),
		$
		which is the log-sum-exp smoothing of the pointwise maximum up to the additive constant $-\mu\log q$.  This well-known LSE function, also termed the  {softmax maximum}~\citep{lecun2015deep} or the  {Neural Networks smoothing function}~\citep{chen2012smoothing, burke2013gradient}, has been extensively applied in finite minimax optimization contexts~\citep{polak2003algorithms, pee2011solving, blanchet2020semi, burke2020subdifferential,wang2023stochastic}, particularly when a differentiable approximation of the maximum operator is desired. Its gradient coincides exactly with the softmax function commonly employed in neural network models~\citep{lecun2015deep}, and extensive studies have documented its robust theoretical properties and practical effectiveness.
		
		While the LME function generalizes LSE to more general settings, including connected compact sets of positive Lebesgue measure, its theoretical properties and convergence analyses differ substantially. Unlike LSE, whose function value can be directly calculated, 
		the LME value generally requires numerical integration or sampling, except in special cases where the exponential integral is available in closed form.
		 Hence, convergence analyses of stochastic smoothing methods employing the   LME smoothing require separate and more intricate mathematical treatments, which have not been previously developed in the literature.
	}

		\section{SSPG: A Stochastic Smoothing Proximal Gradient Framework for Solving Problem~\eqref{eq:minmax}}\label{sec:algorithm}

We focus on solving the primal problem $\min_{\vy} g(\vy)$ in~\eqref{eq:minmax}. 
We 
 {first show a few important properties of the} smoothing function $\widetilde{\Phi}(\cdot,\mu;\vx)$ of $\Phi(\cdot;\vx)$ for each 
$\vx\in\cX$ in Section~\ref{sec:smoothingphi}.
With access to a stochastic gradient estimator, we then introduce the SSPG method in Section~\ref{sec:conver2} and establish its convergence results in Section~\ref{sec:conver3}. 

\subsection{The {LME} Smoothing Function}\label{sec:smoothingphi}

For every $\vx\in\mathcal{X}$, we define the LME smoothing function by~\eqref{eq:smoothphi}. The reference probability measure $\zeta$ is fixed throughout the algorithm: it is the uniform counting measure when $\mathcal{Z}$ is finite, and the normalized Lebesgue measure when $\mathcal{Z}$ has positive Lebesgue measure.
We  make the following assumption to ensure the smoothness of $\widetilde{\Phi}(\cdot,\mu;\vx)$. 
\begin{assumption}[{Gradient Lipschitz}]
	\label{ass:lk} 
There exists $L_\Psi>0$ such that 
	$\nabla_{\vy}{\Psi}(\cdot,\vz;\vx)$ is $L_{\Psi}$-Lipschitz continuous on $ \dom(\varphi)$  	for all $\vz\in\mathcal{Z}$ and all $\vx\in\cX$. 
\end{assumption}
The following lemma shows that $\widetilde{\Phi}(\cdot,\cdot;\vx)$ is a smoothing function of $\Phi(\cdot;\vx)$ and
establishes key 
properties of the smoothing function. 
In particular, it shows that $\widetilde{\Phi}(\vy,\mu;\vx)$ is nonincreasing with respect to $\mu$. The proof of this lemma is provided in Appendix~\ref{sec:smoothphi}.

\begin{lemma}
	\label{lem:smoothphi}
	Under Assumptions~\ref{ass:problemsetup}, \ref{ass:compact1} and~\ref{ass:lk}, the following statements hold for each $\vx\in\cX$.
	\begin{enumerate}
		\item[\textnormal{(a)}] $\widetilde{\Phi}(\cdot,\cdot;\vx)$ is a smoothing function of $\Phi(\cdot;\vx)$.
		\item[\textnormal{(b)}] 
		It holds that $\|\nabla_{\vy} \widetilde{\Phi}(\vy,\mu;\vx)\|\leq l_{\Psi}$ for any $\mu>0$, 
		$\widetilde{\Phi}(\vy,\mu_1;\vx)\leq \widetilde{\Phi}(\vy,\mu_2;\vx)$ for any $\mu_1\geq\mu_2>0$, and $\lim_{ \mu \downarrow 0} \mu \nabla_\mu \widetilde{\Phi}(\mathbf{y}, \mu;\vx) =0$.
		
		\item[\textnormal{(c)}] The $\vy$ partial gradient $\nabla_{\vy} \widetilde{\Phi}(\vy,\mu;\vx)$ is $(L_{\Psi}+ l^2_{\Psi}/\mu)$-Lipschitz continuous with respect to $\vy$ for any $\mu>0$.
	\end{enumerate}
\end{lemma}

For any $\mu_1,\mu_2>0$ and all $\vx\in\cX$, 
we establish upper bounds on   \mbox{$|\widetilde{\Phi}(\vy,\mu_1;\vx)-\widetilde{\Phi}(\vy,\mu_2;\vx)|$} with respect to $|\mu_1-\mu_2|$
in the following two lemmas. Their proofs are given in Appendices~\ref{sec:distancemufinite} and~\ref{sec:distancemu}, respectively.

\begin{lemma}
	\label{lem:distancemufinite}
	Suppose Assumptions~\ref{ass:problemsetup}, \ref{ass:compact1} and~\ref{ass:lk} hold, and $\mathcal{Z}$ is finite discrete. Then for any \mbox{$\kappa \geq 2 \log(|\cal Z|)$,} $\vy\in~\dom(\varphi)$, $1\ge \mu_1>\mu_2 > 0$, and all $\vx\in\cX$, 
    it holds  that {$$|\widetilde{\Phi}(\vy,\mu_1;\vx)-\widetilde{\Phi}(\vy,\mu_2;\vx)|\leq \kappa(\mu_1-\mu_2).$$}
\end{lemma}

\begin{lemma}
	\label{lem:distancemu}
	Suppose Assumptions~\ref{ass:problemsetup}, \ref{ass:compact1} and~\ref{ass:lk} hold, and $\cZ$ is a connected compact set in
	$\mathbb{R}^{m_2}$ with diameter $D_{\cZ}$. Then, for any 
	$\vy\in \operatorname{dom}(\varphi)$, $1\ge \mu_1>\mu_2>0$, and all $\vx\in\cX$, it holds that 
	\[
	\left|
	\widetilde{\Phi}(\vy,\mu_1;\vx)
	-
	\widetilde{\Phi}(\vy,\mu_2;\vx)
	\right|
	\le
	\left[
	2l_\Psi
	+
	2m_2\log\!\left(\frac{1+2 D_{\cZ}}{\mu_1-\mu_2}\right)
	\right](\mu_1-\mu_2).
	\] 
\end{lemma}

Based on the LME function in~\eqref{eq:smoothphi}, we introduce a (nonconvex) ``smoothing'' problem of~\eqref{eq:minmax} as follows:
\begin{equation}
	\label{eq:smootho}
	\min_{\vy} \left\{\widetilde{g}(\vy, \mu) := \varphi(\vy) + \mathbb{E}_{\vx\sim\mathbb{P}} [\widetilde{\Phi}(\vy, \mu;\vx)]\right\}.
\end{equation}
{It is important to note that problem~\eqref{eq:smootho} corresponds precisely to the subproblem solved at each iteration of our proposed  SSPG method in Section~\ref{sec:conver2}, with the smoothing parameter $\mu$ dynamically decreasing toward zero as the iterations proceed.}

\begin{definition}[$\epsilon$-scaled stationary point \citep{bian2013worst,bian2015linearly}]\label{def:scaled-stat}
	{Let $\epsilon\in(0,1]$ and let $\widetilde g(\cdot,\mu)$ denote the 
    objective in \eqref{eq:smootho}.}
	A {random} point $\vy^*$ is called an $\epsilon$-scaled stationary point of problem~\eqref{eq:minmax} \emph{in expectation} if
	it holds $\vy^*\in\dom(\varphi)$ and \mbox{$\mathbb{E}[(\dist(\vzero, \partial\widetilde{g}( \vy^*,\mu)))^2]\leq \epsilon^2$} for some $0< \mu \leq \epsilon$. 
\end{definition}

Definition~\ref{def:scaled-stat} is particularly useful in smoothing methods as it quantifies stationarity for the smoothing problem \(\min_{\vy}\widetilde g(\vy,\mu)\). However, it does not directly characterize stationarity for the original   problem
\(\min_{\vy}g(\vy)\). Furthermore, providing nonasymptotic certification of Clarke stationarity remains generally infeasible. We
therefore directly work with the  \((\mu,\epsilon)\)-Goldstein
stationarity notion introduced in Definition~\ref{def:sta}, and establish in Lemma~\ref{lem:smoothed-to-goldstein} a
conversion from  scaled stationarity to Goldstein stationarity. The proof to Lemma~\ref{lem:smoothed-to-goldstein} is given in Appendix~\ref{sec:smoothed}.

For \(\mu\in(0,1]\), define the LME approximation gap
\[
\Delta_\mu
:=
\sup_{\vy\in\operatorname{dom}(\varphi)}
	\mathbb{E}_{\vx\sim P}
\left[
\Phi(\vy;\vx)-\widetilde\Phi(\vy,\mu;\vx)
\right].
\]
By Lemma~5(a, b), \(\widetilde\Phi(\vy,\mu;\vx)\uparrow \Phi(\vy;\vx)\) as
\(\mu\downarrow0\), and hence \(\Delta_\mu\ge0\). Moreover, we have
\begin{equation}
	\label{eq:omega}\Delta_\mu
	\le
	\omega(\mu)
	:=
	\begin{cases}
		2\mu\log|\cZ| , 
		& \text{if } \cZ \text{ is finite discrete,} \\[1.5mm]
		2\mu\left[
		l_\Psi
		+
		m_2\log\left(\dfrac{1+2D_\cZ}{\mu}\right)
		\right],
		& \text{if } \cZ\text{ is connected compact with diameter }D_\cZ,
	\end{cases}
\end{equation}
where the finite case follows from Lemma~\ref{lem:distancemufinite} by taking \(\mu_1=\mu\) and \(\mu_2\downarrow 0\), and the connected compact case follows from Lemma~\ref{lem:distancemu} by taking \(\mu_1=\mu\) and letting \(\mu_2\downarrow0\).

\begin{lemma} 
	\label{lem:smoothed-to-goldstein}
	Suppose Assumptions~\ref{ass:problemsetup}, \ref{ass:compact1}, and~\ref{ass:lk} hold.
	Let \(\epsilon\in(0,1]\), and suppose that the random point \(\vy^*\in\operatorname{dom}(\varphi)\)
	is an \(\epsilon\)-scaled stationary point of problem~\((P)\) in expectation with smoothing
	parameter \(\mu\in(0,\epsilon]\).  Then
	\(\vy^*\) is a \((\mu_g, \epsilon_g)\)-Goldstein stationary point of problem~\((P)\) in
	expectation, namely,
	$
	\mathbb{E}\left[
	\operatorname{dist}\bigl(\vzero,\partial^{\mu_g} g(\vy^*)\bigr)^2
	\right]
	\le \epsilon_g^2,
	$
	where 
	\[
	{\mu_g}
	:=
	\sqrt{\frac{2\omega(\mu)}{L_\Psi}} = \widetilde{O}(\mu^{1/2}),
	\qquad
	\epsilon_g
	:=
	\sqrt{2}\epsilon+2\sqrt{2L_\Psi\omega(\mu)} = \widetilde{O}(\epsilon+\mu^{1/2}).
	\]
\end{lemma}

	We end this section with an almost-sure convergence result, whose proof   is given in Appendix~\ref{sec:scaled}.

\begin{theorem}[Almost surely convergence to a directional stationary point]
	\label{lem:scaled}
	Suppose Assumptions~\ref{ass:problemsetup}, \ref{ass:compact1}, and~\ref{ass:lk} hold.
	Let $\{\epsilon_k\}$  
	and $\{\vy^{(k)}\}\subset \dom(\varphi)$ be given, such that \mbox{$\sum_{k=0}^{\infty}\epsilon^2_k<+\infty$} and $\vy^{(k)}$ is an $\epsilon_k$-scaled stationary point of problem~\eqref{eq:minmax} in expectation. 
	Then there exists $\{\mu_k\}$ such that $0<\mu_k\le \epsilon_k$ for all $k\ge 0$, and
	\begin{equation}
		\label{eq:almost}
		\lim_{k\rightarrow \infty}\dist(0, \partial\widetilde{g}( \vy^{{(k)}},\mu_k)) =0, \text{ almost surely}.
	\end{equation}
	Moreover, if a subsequence $\{\vy^{(j_k)}\}_{k\ge1}$ converges almost surely to $\vy^\star$, then $\vy^\star$ is a directional stationary point of problem~\eqref{eq:minmax} almost surely.   
\end{theorem}

\subsection{The SSPG method}
\label{sec:conver2}
{In this subsection, we present a stochastic smoothing proximal gradient (SSPG) method for solving  problem~\eqref{eq:minmax}. At iteration~$k$, the SSPG method approximately solves the smoothed subproblem~\eqref{eq:smootho} with smoothing parameter $\mu=\mu_k$ via a single projected stochastic proximal gradient step. 
	The smoothing parameter is then updated  according to a predefined nonincreasing rule. Specifically, at iteration~$k$, we sample $M_k$ independent and identically distributed points $\{\vx_{k_j}\}_{j=1}^{M_k}$ from the distribution $\mathbb{P}$. Under standard bounded-variance assumptions, stochastic methods typically approximate the gradient $\mathbb{E}_{\vx\sim\mathbb{P}}[\nabla_{\vy}\widetilde{\Phi}(\vy^{(k)},\mu_k;\vx)]$ using the sample average
	$\frac{1}{M_k}\sum_{j=1}^{M_k}\nabla_{\vy}\widetilde{\Phi}(\vy^{(k)},\mu_k;\vx_{k_j})$.
	However, as  {the proof of} Lemma~\ref{lem:smoothphi}(b) reveals, exact evaluation of $\nabla_{\vy}\widetilde{\Phi}$ generally involves nested expectations and can thus be computationally prohibitive. To address this issue, 
	we instead assume access to gradient estimators $\mathcal{G}_{k_j}(\cdot,\cdot)$ approximating $\nabla_{\vy}\widetilde{\Phi}(\vy^{(k)},\mu_k;\vx_{k_j})$ for all $j=1,2,\ldots,M_k$. Specifically, we require 
}
\begin{equation}
	\label{eq:gradinetbias2}
	\mathbb{E} \left[\|\mathcal{G}_{k_j}(\vy^{(k)},\mu_{k}) - \nabla_{\vy} \widetilde{{\Phi}}( \vy^{(k)},\mu_{k};\vx_{k_j})\|^2\right]\leq \widehat{\epsilon}_k^2,
\end{equation}
for some scalar $\widehat{\epsilon}_k\geq 0$ and all $j=1,2,\ldots,M_k$.
Specific constructions of these estimators are provided in Remarks~\ref{rem:exact}--\ref{rem:exact2}. The numerical experiments presented in this paper adopt the techniques described in Remark~\ref{rem:exact}.
The $M_k$ gradient estimates are then aggregated into a stochastic gradient estimator
\begin{align}
	\label{eq:gradinetg}
	\mathcal{G}(\vy^{(k)},\mu_{k}) = \frac{1}{M_k} \sum_{j=1}^{M_k}\mathcal{G}_{k_j}(\vy^{(k)},\mu_{k})
\end{align}
of $\mathbb{E}_{\vx\sim\mathbb{P}}[\nabla_{\vy}\widetilde{\Phi}(\vy^{(k)},\mu_k;\vx)]$.
We proceed with a proximal gradient update step, followed by updating the value of \( \mu \) according to a predefined nonincreasing rule. 
The SSPG method is summarized in Algorithm~\ref{alg:spg2}.

\begin{algorithm}[t]
	\caption{A \textbf{s}tochastic \textbf{s}moothing \textbf{p}roximal \textbf{g}radient (SSPG) method for solving~\eqref{eq:minmax}
	}
	\label{alg:spg2}
	\begin{algorithmic}[1]
		\State{Choose $\vy^{(0)}\in\operatorname{dom}(\varphi)$, $\mu_0>0$, target estimator accuracies $\{\widehat\epsilon_k\}_{k=0}^{K-1}\subset(0,\infty)$, and an iteration budget $K$.}
		\For{{$k=0,1,\ldots, K-1$} }
		\State\label{line323}
		Choose a stepsize $\alpha_k>0$ and form $\mathcal{G}(\vy^{(k)},\mu_k)$ as in~\eqref{eq:gradinetg}. 
		Update $\vy$ by 
		\begin{equation}
			\label{eq:updatex2}
			\vy^{(k+1)} = \operatorname{Prox}_{\alpha_k\varphi} \left(\vy^{(k)}- {\alpha_k}{} \mathcal{G}(\vy^{(k)},\mu_{k}) \right).
		\end{equation}

		\State\label{line423}
		Choose $0<\mu_{k+1}\le \mu_k$.
		\EndFor
	\end{algorithmic}
\end{algorithm}

\begin{remark}[Derive the estimator $\cG$ via sampling methods]
	\label{rem:exact}
    We briefly discuss several ways to satisfy
	\eqref{eq:gradinetbias2}.
	First, if  \(\mathcal{Z}\) is finite, then
	\(\nabla_{\vy}\widetilde\Phi(\vy^{(k)},\mu_k;\vx)\) can be evaluated
	exactly, and one may take \(\widehat\epsilon_k=0\). 
	This situation is common in WDRO models where a discrete grid is used to approximate the entire sample space
    \citep{xu2018distributionally,chen2021decomposition,liu2021discrete,pflug2007ambiguity}. 
	Second,  {by the proof of Lemma~\ref{lem:smoothphi}(b)}, if for almost everywhere $\vx\sim\mathbb{P}$, the expectation \(\mathbb{E}_{\vz\sim \zeta}[e^{\Psi(\vy, \cdot; \vx)/\mu}]\) can be computed, 
	we can generate samples to approximate the expectation  \(\mathbb{E}_{\vz\sim \zeta}[e^{\Psi(\vy, \cdot; \vx)/\mu}\nabla_{\vy}\Psi(\vy, \cdot; \vx)]\). By doing so, we can generate an unbiased stochastic gradient estimator of \(\nabla_{\vy} \widetilde{\Phi}(\vy^{(k)}, \mu_{k};\vx)\). 
	Lastly, when  \(\mathcal{Z}\)  is a connected compact set, 
    under a flat maximum condition, we can efficiently construct a suitable stochastic gradient estimator. We detail this construction in Appendix~\ref{app:localized-small-ball-centered}. 
\end{remark}

\begin{remark}[Derive the estimator $\cG$ via Langevin Monte Carlo methods]  
	\label{rem:exact2}  
	Recent progress on non-log-concave sampling, especially the stationarity-based theory of~\citep{balasubramanian2022towards}, has shown that Langevin-type methods can be analyzed beyond the   strongly log-concave regime through optimization-inspired stationarity notions. This viewpoint is particularly relevant here, as shown below. Our main convergence theory does not rely on an LMC sampler, but this connection suggests a principled route for constructing inner gradient estimators in smooth nonconvex regimes.
	
	Fix an iteration $k$ and an index $j$, and abbreviate
	$\vx_j:=\vx_{k_j}$, $\vy:=\vy^{(k)}$, and $\mu:=\mu_k$. 
	Our goal is to construct a stochastic gradient estimator $\cG_j(\vy,\mu)$ such that 
	\begin{equation}
		\label{eq:gradinetbias3}
		\EE\left[\left\|\,\cG_j(\vy,\mu)
		-\nabla_{\vy}\left(\mu\log\EE_{\vz\sim\zeta}[e^{\Psi(\vy,\vz;\vx_j)/\mu}]\right)\right\|_2^2\right]\;\le\;\epsilon^2.
	\end{equation} 
	As an alternative approach, we rewrite $\nabla_{\vy} \widetilde\Phi(\vy,\mu;\vx_j)=\nabla_{\vy} [ \mu \log \mathbb{E}_{\vz\sim \zeta} [e^{{\Psi(\vy,\vz;\vx_j)}/{\mu} }] ] $ as
	\begin{align*}
		&\nabla_\vy\widetilde\Phi(\vy,\mu;\vx_j)
	=\EE_{\vz\sim\zeta^{(\Phi)}}\big[\nabla_\vy\Psi(\vy,\vz;\vx_j)\big], 
	\text{ where }\\
	&\frac{d\zeta^{(\Phi)}}{d\zeta}(\vz)
	=\frac{\exp(\Psi(\vy,\vz;\vx_j)/\mu)}{\EE_{\vu\sim\zeta}\exp(\Psi(\vy,\vu;\vx_j)/\mu)}\propto
	\exp\Big(\frac{\Psi(\vy,\vz;\vx_j)}{\mu}\Big).
	\end{align*}
	Thus, condition \eqref{eq:gradinetbias3} can be satisfied by approximately sampling from $\zeta^{(\Phi)}$; specifically, 
	we draw samples $\{\vz_1,\ldots,\vz_M\}\sim\widehat\pi_k$ (a sampling distribution that approximates  $\zeta^{(\Phi)}$) 
	and form the Monte Carlo gradient estimator 
	$
	\cG_j(\vy,\mu):=\frac{1}{M}\sum_{i=1}^M\nabla_\vy\Psi(\vy,\vz_i;\vx_j).
	$
	The left hand side of \eqref{eq:gradinetbias3} can then be bounded by 
	\begin{equation}
		\label{eq:gradinetbias5}
		\frac{\mathrm{Var}_{\vz\sim\zeta^{(\Phi)}}(\nabla_\vy \Psi(\vy,\vz;\vx_j))}{M}
		+\big\|\EE_{\vz\sim\widehat\pi_k}[\nabla_\vy \Psi(\vy,\vz;\vx_j)]-\EE_{\vz\sim\zeta^{(\Phi)}}[\nabla_\vy \Psi(\vy,\vz;\vx_j)]\big\|_2^2,
	\end{equation}
	which splits into Monte Carlo variance and a sampling bias. 
	Therefore, the central task reduces to approximate sampling from the distribution $\zeta^{(\Phi)}$. This problem aligns closely with the active research area known as Langevin diffusion~\citep{chewi2023optimization,wang2025iterative,ding2021random}.
	
	When $\Psi(\vy,\cdot;\vx_j)$ is differentiable in $\vz$ and $\cZ$ is a compact convex set, 
	one can employ standard Langevin Monte Carlo (LMC)~\citep{chewi2023optimization} to approximately generate samples $\{\vz_i\}_{i=1}^M$ from $\zeta^{(\Phi)}$ using the projected unadjusted Langevin algorithm:
	$$
	\vz_{t}=\mathrm{Proj}_{\cZ}\left(\vz_{t-1}+\alpha \nabla_{\vz}\Psi(\vy,\vz_{t-1};\vx_j)/\mu +\sqrt{2\alpha}\,\xi_t\right)
	$$
	for all $t\in[M]$, where $\vz_0\in\cZ$ is an initial point, $\alpha>0$ is a step size, and $\xi_t$ satisfies the normal distribution. 
	Classical theoretical guarantees for LMC typically assume $\cZ=\RR^{m_2}$, and strong log-concavity of the target measure, e.g., 
	$\Psi(\vy,\cdot;\vx_j)$ is strongly concave and smooth, so the measure $\exp(\Psi(\vy,\cdot;\vx_j)/\mu)$ is strongly log-concave~\citep{chewi2023optimization}.  Beyond this   scenario, two notable extensions have been established:  (i) In the  {nonsmooth convex} case,  \citet{durmus2019analysis} relax smoothness conditions but still maintain convexity assumptions. (ii) In the nonconvex smooth case, analyses by \citet{vempala2019rapid} and \citet{chewi2024analysis} rely on the Log-Sobolev Inequality (LSI)~\citep[Definition 2.2.7]{chewi2023optimization}, a condition serving as the sampling analogue to the PL condition commonly used in the optimization field. 
	While the nonsmooth convex extension (i) is already encompassed by Remark~\ref{rem:exact}, the LSI-based nonconvex smooth extension (ii) presents a fundamentally distinct scenario that Remark~\ref{rem:exact} does not cover.
	Indeed, LSI is strictly more general than strong log-concavity and does not imply convexity or PL condition satisfaction. For example, consider the distribution
	$
	\frac{d\zeta^{(\Phi)}}{d\zeta}(\vz)\propto \exp\left(\Psi(\vy,\vz;\vx)/\mu\right) 
	$
	with $\Psi(\vy,\vz;\vx) = -\frac{1}{2}\|\vz\|_2^2-\gamma\sin\left(\textbf{1}_{m_2}\zz\vz\right) $ and $\mu=1,$  
	which satisfies LSI for any $\gamma>1$ (as a bounded perturbation of a Gaussian) but whose potential  $\Psi/\mu= -\frac{1}{2}\|\vz\|_2^2-\gamma\sin(\textbf{1}_{m_2}\zz\vz)$ is neither concave nor satisfies the PL condition whenever $\gamma>1$ and $m_2>1$. Thus, LSI and PL conditions capture distinct structural properties and are not directly comparable.
	In summary, although employing LMC to construct stochastic gradient estimators currently lacks established theoretical results in the nonsmooth nonconvex setting and thus remains theoretically less understood compared to the optimization-based estimators discussed in Remark~\ref{rem:exact}, it has the advantage of supporting smooth nonconvex scenarios under broader conditions such as the LSI, which are not explicitly addressed by the optimization-based methods. 
\end{remark}

\subsection{Convergence Results}\label{sec:conver3}

The following lemma will be used for establishing the convergence results of Algorithm~\ref{alg:spg2}.  
Its proof is given in Appendix~\ref{sec:alg3}.  

\begin{lemma}\label{lem:alg3}
	Suppose Assumptions~\ref{ass:problemsetup}, \ref{ass:compact1}, and~\ref{ass:lk} hold, and~\eqref{eq:gradinetbias2} is satisfied for each iteration.
	Let $\{ \vy^{(k)}\}$ and $\{\mu_k\}$ be the sequences generated by Algorithm~\ref{alg:spg2} with $M_k= \lceil 4l_{\Psi}^2 \widehat{\epsilon}_k^{-2}\rceil$,  and $\alpha_k=  \frac{1}{L^{(k)}}$  for all $k\in[K]$, where $C_2=L_{\Psi}\mu_0+l^2_{\Psi}$, and $L^{(k)}=  \frac{C_2}{ \mu_k}$.
	Then, {for all \mbox{$k\in[K-1]$},} the following statements hold. 
	\begin{enumerate}
		\item[\textnormal{(a)}] 
	 	The sequence \(\left\{ \widetilde{g}( \vy^{(k)},\mu_k)\right\}\) satisfies 
		\begin{equation}
			\label{eq:funcgap2}
			\mathbb{E}\left[\widetilde{g}( \vy^{(k+1)},\mu_k)- \widetilde{g}( \vy^{(k)},\mu_k)\right]   \leq 
			-\frac{L^{(k)}}{4} \mathbb{E} \left[\| \vy^{(k+1)}- \vy^{(k)}\|^2\right] + \frac{4}{ L^{(k)}}\widehat{\epsilon}_k^2.
		\end{equation}
		\item[\textnormal{(b)}] It holds that 
		\begin{equation}\label{eq:kktregu2}		
			\begin{aligned}
				\mathbb{E}\left[\left(\dist\left(\vzero, \partial\widetilde{g}( \vy^{(k+1)},\mu_k)\right)\right)^2\right]  
				\leq 
				18 L^{(k)} \mathbb{E} \left[\widetilde{g}( \vy^{(k)},\mu_k)- \widetilde{g}( \vy^{(k+1)},\mu_k)\right] +  {112}{}\widehat{\epsilon}_k^2.
			\end{aligned}
		\end{equation}
	\end{enumerate}
\end{lemma}

Now we are ready to present the main convergence rate result.  
The proof  is given in Appendix~\ref{sec:scaled2}.

\begin{theorem}[Stationarity violation bound]
	\label{thm:scaled2}
	Suppose Assumptions~\ref{ass:problemsetup}, \ref{ass:compact1}, and~\ref{ass:lk} hold, and~\eqref{eq:gradinetbias2} is satisfied for each iteration. Let $0<\epsilon<1$,  and $K> k_1 \ge 0$ be given. Set $M_k= \lceil 4l_{\Psi}^2 \widehat{\epsilon}_k^{-2}\rceil$, and  $\alpha_k=  \frac{\mu_k}{C_2}$ for all $k\in[K]$, where $C_2=L_{\Psi}\mu_0+2l^2_{\Psi}$.
	Let~$\tau$ be randomly sampled from $\{k_1,k_1+1,\ldots,K-1\}$ with probability  $\operatorname{Prob}(\tau=k)=\frac{\mu_k}{\sum_{t=k_1}^{K-1}\mu_t}$. 
	Then, 
	\begin{align}
		\notag
		&
		\mathbb{E}\left[\left(\dist\left(0, \partial\widetilde{g}( \vy^{(\tau+1)},\mu_{\tau})\right)\right)^2\right]
		\\\leq &  
		\frac{18 C_2 \mathbb{E} \left[\widetilde{g}( \vy^{(k_1)},\mu_{k_1})- \widetilde{g}( \vy^{(K)},\mu_{K-1}) \right] }{\sum_{k=k_1}^{K-1} {\mu_k}} + \frac{18 C_2 \sum_{k=k_1}^{K-2} \omega\left(\mu_{k}- \mu_{k+1}\right),  }{\sum_{k=k_1}^{K-1} {\mu_k}} +  \frac{112\sum_{k=k_1}^{K-1} { \mu_k}{ } \widehat{\epsilon}_k^2}{\sum_{k=k_1}^{K-1} {\mu_k}},
		\label{eq:expectationsta}
	\end{align}
	where $\omega(\cdot)$ is given in~\eqref{eq:omega}.
\end{theorem}
  
There are multiple strategies for selecting \(\widehat{\epsilon}_k\) and \(\mu_k\). 
Typically, we set $\widehat{\epsilon}_k$ 
in the same order as $\epsilon$. For~$\mu_k$, it can 
be kept constant, decay 
over 
$k$, or be updated based on the difference between the objective function values at two consecutive iterations~\citep{liu2025stochastic}.
Below, we present a corollary detailing specific choices for \(\widehat{\epsilon}_k\) and \(\mu_k\), along with an analysis of the computational complexity required by 
Algorithm~\ref{alg:spg2}.
The proof is given in Appendix~\ref{sec:cor}.

\begin{corollary}
	\label{cor:main1}
	Under the same assumptions of Theorem~\ref{thm:scaled2}, the following
	two claims hold.
	
	\begin{enumerate}
		\item[\textnormal{(a)}]
		Set \(k_1=0\), \(\widehat{\epsilon}_k=\epsilon/16\), and
		\(\mu_k=\epsilon\) for all \(k\in[K]\), with
		\[
		K=
		\left\lceil
		36C_2\epsilon^{-3}
		\left(
		\widetilde g(\vy^{(0)},\epsilon)
		-
		\min_{\vy}\widetilde g(\vy,\epsilon)
		\right)
		\right\rceil .
		\]
		Then Algorithm~\ref{alg:spg2} outputs an $\epsilon$-scaled stationary point
		\(\vy^{(k+1)}\), for some \(0\le k<K\), in expectation.
		
		\item[\textnormal{(b)}]
		Let \(0<\mu_g, \epsilon_g \le  1\).
		Set \(k_1=0\),
		$
		\widehat{\epsilon}_k=\frac{\epsilon_g}{16},
		\mu_k=\mu_g^2
		\text{ for all } k\in[K],
		$
		with
		\[
		K=
		\left\lceil
		36C_2\epsilon_g^{-2}\mu_g^{-2}
		\left(
		\widetilde g(\vy^{(0)},\mu_g^2)
		-
		\min_{\vy}\widetilde g(\vy,\mu_g^2)
		\right)
		\right\rceil .
		\]
		Then Algorithm~\ref{alg:spg2} outputs a 
		\(( \sqrt{\frac{2\omega(\mu_g^2)}{L_\Psi}},\sqrt{2}\epsilon_g +  2\sqrt{2L_\Psi\omega(\mu_g^2)})\)-Goldstein stationary point 
		\(\vy^{(k+1)}\), for some \mbox{\(0\le k<K\),}  in expectation. Here, $\omega(\cdot)$ is given in~\eqref{eq:omega}.
	\end{enumerate}
\end{corollary}

 \begin{remark}[An explicit bound on $K$] 
 	Recall that $\widetilde g(\vy,\mu)=\varphi(\vy)+\mathbb{E}_{\vx\sim\mathbb{P}} [\widetilde\Phi(\vy,\mu; \vx)]$, where the smoothing terms $\widetilde\Phi(\cdot,\mu;\vx)$ are   $l_\Psi$-Lipschitz continuous in $\vy$ for all $\mu>0$ and $\vx\in\cX$, as established by Lemma~\ref{lem:smoothphi}(b). Denote 
		$
		\Delta_{\varphi} :=  \varphi(\vy^{(0)})-\min_{\vy\in\dom(\varphi)}\varphi(\vy),  $ and $
		\Delta :=  \Delta_\varphi + l_\Psi D,
		$
		where $D$ denotes the diameter of $\dom(\varphi)$. Then, for any smoothing parameter $\mu\in(0,1]$ and any initial point $\vy^{(0)}\in\dom(\varphi)$, the following bound holds:
		\begin{align*}
			\widetilde g(\vy^{(0)},\mu)-\min_{\vy} \widetilde g(\vy,\mu)
			&= \varphi(\vy^{(0)}) -  \varphi(\vy^*) + \mathbb{E}_{\vx\sim\mathbb{P}} [\widetilde\Phi(\vy^{(0)},\mu; \vx) - \widetilde\Phi(\vy^*,\mu; \vx)]
			\\ &\le \varphi(\vy^{(0)}) - \min_{\vy\in\dom(\varphi)}\varphi(\vy) + \mathbb{E}_{\vx\sim\mathbb{P}} [l_\Psi \|\vy^{(0)} - \vy^*\|]\;\le\; \Delta,
		\end{align*}
		where $\vy^*\in\arg\min_{\vy} \widetilde g(\vy,\mu)$. Additionally, we clarify that a lower bound of $\min_{\vy\in\dom(\varphi)}\varphi(\vy)$ is typically straightforward to obtain in many cases; for instance, if $\varphi$ is an indicator function of a closed convex set, we explicitly have $\min_{\vy\in\dom(\varphi)}\varphi(\vy)=0$.
		Consequently, we can directly verify that $\left\lceil 36\,C_2\,\epsilon^{-3}\,\Delta\right\rceil$ is a valid upper bound for the iteration number $K$ specified in Corollary~\ref{cor:main1}(a). Similarily, for Corollary~\ref{cor:main1}(b), a  valid upper bound for the iteration number $K$ can also be obtained.
\end{remark}

\begin{remark}[Asymptotic convergence]
	\label{rem:23}
	By choosing $\widehat\epsilon_k = \mu_k/16$, we iteratively refine our approximations to the stationary point of problem~\eqref{eq:minmax}. 
	For $t=1,2,\ldots$, set $K_1=0$ and
	$
	K_{t+1}=K_t+\left\lceil 36C_2 t^3\Delta\right\rceil .
	$
	For $K_t\le k<K_{t+1}$, choose $\mu_k=1/t$ and $\widehat\epsilon_k=\mu_k/16$. By Corollary~\ref{cor:main1}, each stage produces a point $\vy^{(t)}$ satisfying the $1/t$-scaled smoothed stationarity condition in expectation. Since $\sum_{t=1}^{\infty}t^{-2}<\infty$, Theorem~\ref{lem:scaled} implies that every almost-sure cluster point of $\{\vy^{(t)}\}$ is d-stationary for~\eqref{eq:minmax}. 
\end{remark}

\begin{remark}[Complexity results]\label{rem:sampling}
	By Corollary~\ref{cor:main1},    the iteration complexity for obtaining an $\epsilon$-scaled stationary point is ${O}(\epsilon^{-3})$, for obtaining an $(\epsilon,\epsilon)$-Goldstein stationary point is $\widetilde{O}(\epsilon^{-4})$. Recall from Theorem~\ref{thm:scaled2} that the sampling budget is $M_k= \lceil 4l_{\Psi}^2 \widehat{\epsilon}_k^{-2}\rceil$ per iteration. Thus, the sampling complexity for obtaining an $\epsilon$-scaled stationary point is ${O}(\epsilon^{-5})$, for obtaining an $(\epsilon,\epsilon)$-Goldstein stationary point is $\widetilde{O}(\epsilon^{-6})$.
\end{remark}

		\section{Numerical Experiments}\label{sec:numerical}

We evaluate Algorithm~\ref{alg:spg2} on four test problems related to WDRO and adversarially robustness: a newsvendor problem, a robust regression problem, an $\infty$-Wasserstein robust regression problem, and an adversarially robust image-classification problem. We compare SSPG with GDMax~\citep{jin2020local} and SDRO~\citep{wang2021sinkhorn}. For the $\infty$-Wasserstein experiments, we additionally compare with FGSM~\citep{goodfellow2014explaining} and PGD~\citep{madry2018towards}. All experiments are implemented in Python; hardware information is summarized in Table~\ref{tab:python-libraries-processor-info}.
Our implementation of SDRO follows the publicly available implementation accompanying~\citet{wang2021sinkhorn}.

\begin{table}[t]
	\centering
	\renewcommand{\arraystretch}{1.5}
	\resizebox{\textwidth}{!}{%
		\small
		\begin{tabular}{|>{\centering\arraybackslash}m{0.45\textwidth}|>{\centering\arraybackslash}m{0.45\textwidth}|}
			\hline
			\textbf{Problem Type} & \textbf{Processor Info} \\ \hline
			Newsvendor (Section~\ref{sec:news}),
			Regression Problem (Section~\ref{sec:regre}) and $\infty$-Wasserstein DRO (Section~\ref{sec:infinity_wasserstein_regression})
			& 12th Gen Intel(R) Core(TM) i5-1240P with 8GB RAM \\ \hline
			Adversarially Robust Deep Learning Problem (Section~\ref{sec:adversial})
			& NVIDIA Ampere A100 GPU with 80 GB RAM \\ \hline
		\end{tabular}%
	}
	\caption{Tested applications and processor information}
	\label{tab:python-libraries-processor-info}
\end{table}

\subsection{Newsvendor Problem}\label{sec:news}
The newsvendor problem, which models the expected {cost} of a retailer under uncertain demand, takes the form of problem~\eqref{eq:model}. In this subsection, we consider solving problems~\eqref{eq:model2}, ~\eqref{eq:smoothg} and~\eqref{eq:smoothg2} using
\begin{equation}\label{eq:newsvendor}
	\ell(\theta, x) = v\theta - u\min(\theta,x)\quad\text{and}\quad d(x,z) = \frac{1}{2} (x - z)^{2},
\end{equation}
where  $\theta \in \mathbb{R}_+$ represents the inventory level, $x \in \mathbb{R}$ denotes the demand, $v = 5$ is the underage cost,  and $u = 7$ is the overage cost. To ensure that the inner maximization problem has a finite solution, we impose the constraint $\lambda \geq 7$ on the target problems, as required in \citep{lee2021data}.


\textbf{Problem parameters:} We synthetically generate five different demand datasets, each consisting of \mbox{$n=100$} independent samples drawn from an exponential distribution with a rate parameter of 1.
We set $\delta = 1$ and $p = 2$.
For each dataset $X_{\mathrm{train}}$, the empirical distribution $\widehat{\mathbb{P}}_{n}$ is constructed from these~$n$ samples. The support set is defined as $\mathcal{Z}=
\{x:\min (X_{\mathrm{train}})\le x\le\max(X_{\mathrm{train}})\}$. 

\textbf{Algorithm parameters:}
For all methods, we use the same initialization $\theta^{(0)}\sim \mathcal{U}(0, 1)$ on all five datasets, where~$\cU$ denotes the uniform distribution.
In addition, for GDMax and SSPG, we use the same \mbox{$\lambda^{(0)}  \sim \mathcal{U}(7, 15)$} across all five datasets.
For SSPG, we set $\mu_0 = \lambda^{(0)}  \eta$, where $\eta \sim \mathcal{U}(0.1, 1)$.
Following \citep{wang2021sinkhorn}, we use a grid search for SDRO to fine-tune the hyperparameters $\lambda$ and~$\eta$ from the sets \mbox{\{7, 10, 15\}} and \{0.1, 0.5, 1\}, respectively.
Each method is terminated after 1000 iterations.

\textbf{Implementation of the compared methods at the $k$th iteration:}
For each method,   we solve several inner maximization problems to obtain $\{z_i^{(k+1)}\}_{i=1}^n\subset \RR $.
Specifically, for each $i\in [n]$, we solve problem $\max_{z \in \cZ} \ell(\theta^{(k)}, z) - \lambda^{(k)} d(x_i, z)$ using the projected gradient ascent with  fixed step size of~$10^{-2}$ for 20 iterations, starting from $\mathrm{Proj}_{{\mathcal{Z}}}(x_i + 10^{-3}r)$ where $r$ satisfies the normal distribution. 
After obtaining $\{z_i^{(k+1)}\}_{i=1}^n$, GDMax  performs  projected gradient descent   on  $(\theta,\lb)$ using   $\nabla_{(\theta,\lb)}\frac{1}{n}\sum_{i=1}^n\left[\lb^{(k)}\delta^p+\ell(\theta^{(k)},z_i^{(k+1)})-\lambda^{(k)} d(x_i,z_i^{(k+1)})\right]$ as a gradient and a fixed learning rate~$0.1$.

For SSPG and SDRO, for all $i\in [n]$, we generate sets $\Omega_i^k$ of size $M=32$, containing samples near~$z_i^{(k+1)}$. For any $\widehat{z}\in\Omega_i^k$, we set \( \widehat{z} = \mathrm{Proj}_{{\mathcal{Z}}}(z_i^{(k+1)} + r^{(k+1)}) \) with   $r^{(k+1)}$ satisfying the normal distribution.
SDRO then performs a projected gradient descent step on  $\theta$ with gradient $\nabla_{\theta}{g}_s^k(\theta^{(k)},\lb)$ and a fixed learning rate~$0.1$, where
\begin{align}
	\label{eq:gradientgsk}
	{g}_s^k( \theta,\lb) = \lambda \delta^{p} + \lb\eta \frac{1}{n}\sum_{i=1}^n\left[ \log \left(\frac{1}{M} \sum_{\widehat{z}\in\Omega_i^k}  \left[ e^{\frac{\ell(\theta, \widehat{z}) - \lambda d(x_i, \widehat{z})}{\lb\eta}} \right] \right)\right].
\end{align}
Our SSPG method conducts a projected gradient descent step on the primal variable $(\theta,\lb)$ with $\nabla_{(\theta,\lb)}\widetilde{g}^k( \theta^{(k)},\lb^{(k)} , \mu_k)$ and a fixed learning rate $0.1$, where
\begin{align}
	\label{eq:gradientgtildek}
	\widetilde{g}^k(\theta,\lb,\mu) = \lambda \delta^{p} + \mu \frac{1}{n}\sum_{i=1}^n\left[ \log \left(\frac{1}{M} \sum_{\widehat{z}\in\Omega_i^k}  \left[ e^{\frac{\ell(\theta, \widehat{z}) - \lambda d(x_i, \widehat{z})}{\mu}} \right] \right)\right].
\end{align}
We then update $\mu$ according to
\begin{equation}\label{eq:munews}
	\mu_{k+1} = \left\{
	\begin{aligned}
		& \mu_k, \hspace{2.7cm} \text{ if }
		\widetilde{g}^k( \theta^{(k+1)},\lb^{(k+1)}, \mu_k) - \widetilde{g}^k( \theta^{(k)}, \lb^{(k)},\mu_k)< -\mu_k^{2\sigma_2}, \\
		&\max\{\sigma_1 \mu_k, 10^{-4}\mu_{0}\}, \,\,\, \text{otherwise},
	\end{aligned}
	\right.
\end{equation}
where $\sigma_{1} = 0.99$, $\sigma_{2} = 0.5$.

\begin{figure}
	\centering
	\begin{subfigure}{0.4\textwidth}
		\centering
		\includegraphics[width=\linewidth]{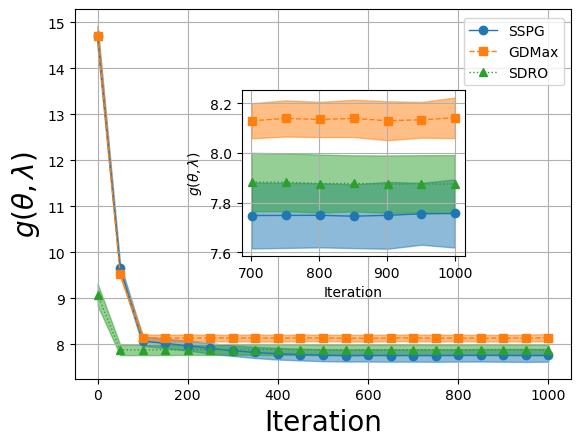}
		\caption{$g(\theta, \lambda)$}
		\label{fig:sub2}
	\end{subfigure}
	~~~~
	\begin{subfigure}{0.4\textwidth}
		\centering
		\includegraphics[width=\linewidth]{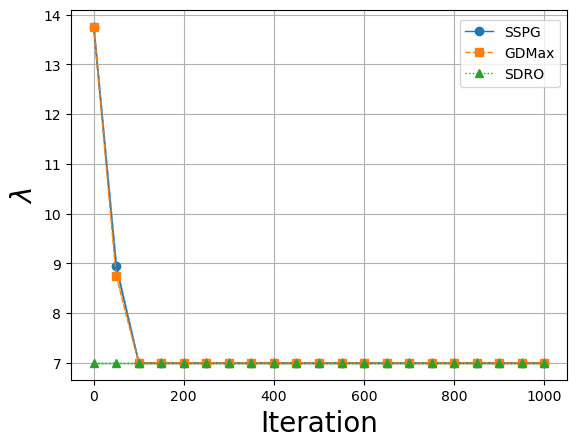}
		\caption{$\lambda$ iterates}
		\label{fig:sub1}
	\end{subfigure}
	\caption{Comparison of $g(\theta,\lambda)$ and $\lb$  among SSPG, GDMax, and SDRO for solving the newsvendor problem.}
	\label{fig:overall}
\end{figure}

\textbf{Performance comparisons:}

Figure~\ref{fig:overall}(a) reports the mean and standard deviation of $g(\theta,\lambda)$ over the five datasets. In this experiment, SSPG attains lower final objective values than GDMax and SDRO. Figure~\ref{fig:overall}(b) reports the corresponding values of $\lambda$. SSPG and GDMax move rapidly to the imposed lower bound $\lambda=7$, while SDRO uses the fixed value $\lambda=7$ selected by grid search.

Figure~\ref{fig:overall}(b) shows the mean and standard deviation of \( \lambda \).
It reveals that SSPG and GDMax quickly converge to values of 7, while SDRO maintains a fixed \(\lambda = 7\).
This choice results from a grid search indicating that the pair \(\lambda = 7, \eta=0.1\)  yields the best performance among the tested configurations for  SDRO.

\subsection{Regression Problem}\label{sec:regre}
The distributionally robust regression problem aims to find a robust solution to the standard regression problem by minimizing the worst-case risk.
In this subsection, we consider problems~\eqref{eq:model2},~\eqref{eq:smoothg}, and~\eqref{eq:smoothg2} with
\begin{equation}\label{eq:ell-d-func}
	\ell(\vtheta,\vx) = (h_{\vtheta}(\va) - b)^{2}, \text{ and }d(\vx,\vz) = d((\va, b), (\overline{\va}, \overline{b})) = \frac{1}{2} \|\va - \overline{\va}\|_{2}^{2} + \infty |b - \overline{b}|,
\end{equation}
where $h_{\vtheta}: \mathbb{R}^{m_2-1} \to \mathbb{R}$ is a {neural network} parameterized by $\vtheta$, $\vx = (\va, b)$, $\vz = (\overline{\va}, \overline{b})$ with $\va, \overline{\va} \in \mathbb{R}^{m_2-1}$ representing a vector of features and  $b, \overline{b} \in \mathbb{R}$ denoting a target.
Specifically,
we employ a neural network with a single hidden layer containing three neurons, using the ReLU activation function~\citep{goodfellow2016deep}. 
The $\infty$ before $|b - \overline{b}|$ means that there is no uncertainty in the target variable.

\textbf{Datasets:}
We consider three real-world datasets, \verb|Space GA|, \verb|BodyFat|, and \verb|MG|, from the \verb|LIBSVM| repository.
 Each dataset, containing
$n$ data points, is randomly partitioned into training (80\%) and testing (20\%) sets using \verb|train_test_split| from \verb|Scikit-learn|, using the same random seeds across all methods.  The resulting training and test sets are
$X_{\text{train}} \in \mathbb{R}^{\lfloor 0.8 \times n \rfloor \times m_2}$ and $ X_{\text{test}} \in \mathbb{R}^{\lceil 0.2 \times n \rceil \times m_2}$, respectively.
These two sets are further normalized using  \verb|StandardScaler|. To ensure fair comparisons, we employ five distinct random seeds for data preparation, including splitting, and initialization, ensuring identical dataset configurations across all three methods.

\textbf{Problem parameters:}
We set $\delta = 10$ and $p = 2$.
For each dataset $X_{\mathrm{train}}$, the empirical distribution~$\widehat{\mathbb{P}}_{n}$ is constructed from the training samples; the support set is given as $\mathcal{Z} = \widetilde{\mathcal{Z}} \times \mathbb{R}$ with~$\widetilde{\mathcal{Z}}~=~\prod_{j=1}^{m_2-1} [ \min([X_{\text{train}}]_{\cdot,j}), \max([X_{\text{train}}]_{\cdot,j}) ]$.  


\textbf{Algorithm parameters:}
For all methods, we initialize $\vtheta^{(0)}$ using PyTorch's default parameter initialization under a fixed random seed.
In addition, for GDMax and SSPG, we let $\lambda^{(0)}  \sim \mathcal{U}(1, 10)$.
For SSPG, we set $\mu_0 = \lambda^{(0)}  \eta$, where~$\eta \sim \mathcal{U}(0.1, 1)$.
Following \citep{wang2021sinkhorn}, we utilize a grid search for SDRO to fine-tune the hyperparameters $\lambda$ and~$\eta$ from the sets \{1, 5,10\} and \{0.1, 0.5, 1\}, respectively.
For each comparison method, the training is terminated after 500 epochs.

\textbf{Implementation of the compared methods at the $k$th iteration:}
For each method,  we perform an inner maximization step to obtain
$\{\vz_i^{(k+1)}= (\overline{\va}_i^{(k+1)}, \overline{b}_i^{(k+1)})\}_{i=1}^{\lfloor 0.8 \times n \rfloor} \subset  \mathbb{R}^{m_2}$.
Specifically, for each $i\in[\lfloor 0.8 \times n \rfloor]$, we solve problem $\max_{\vz \in \cZ} \ell(\vtheta^{(k)}, \vz) - \lambda^{(k)} d(\vx_i, \vz)$ using projected gradient ascent with a fixed step size of~$10^{-2}$ for five iterations, starting from $\vx_i + 10^{-3} \vr$ where $\vr$ satisfies the normal distribution.
Notice that by the definition of $d$ in~\eqref{eq:ell-d-func}, we actually fix the last component of $\vz_i$ as the label corresponding to $\vx_i$.

After obtaining $\{\vz_i^{(k+1)}\}_{i=1}^{\lfloor 0.8 \times n \rfloor}$, GDMax updates   $(\vtheta,\lb)$ via  a projected gradient descent step using the gradient $\nabla_{(\vtheta,\lb)}\mathbb{E}_{\vx \sim \widehat{\mathbb{P}}_n}\left[\lb^{(k)}\delta^p+\ell(\vtheta^{(k)},\vz^{(k+1)})-\lambda^{(k)} d(\vx,\vz^{(k+1)})\right]$ and a fixed learning rate $\alpha>0$.
For SSPG and SDRO, we generate, for each $i\in[\lfloor 0.8 \times n \rfloor]$, a set $\Omega_i^k$ of size $M=32$, containing samples near~$\vz_i^{(k+1)}$. Specifically, for any $\widehat{\vz}\in\Omega_i^k$, we let \( \widehat{\vz} = \mathrm{Proj}_{{\mathcal{Z}}}(\vz_i^{(k+1)} + \vr^{(k+1)}),\) where $\vr^{(k+1)})$ satisfies the normal distribution.
SDRO then performs a projected gradient descent step on  $\vtheta$ with gradient $\nabla_{\vtheta}{g}_s^k(\vtheta^{(k)},\lb)$ and a fixed learning rate $\alpha$, where ${g}_s^k$ is given in~\eqref{eq:gradientgsk} with $\widehat{z}$ replaced by $\widehat{\vz}$.
Our SSPG method similarly applies a projected gradient descent step on $(\vtheta,\lb)$ with gradient $\nabla_{(\vtheta,\lb)}\widetilde{g}^k( \vtheta^{(k)},\lb^{(k)} , \mu_k)$ and a fixed learning rate $\alpha$, where  $\widetilde{g}^k$ is given in~\eqref{eq:gradientgtildek} with $\widehat{z}$ replaced by~$\widehat{\vz}$. We then update~$\mu$ by 
\begin{equation}\label{eq:mu2}
	\begin{aligned}
		&\mu_{k+1} := \max\left\{10^{-4}\mu_{0},{\left(k+2\right)}^{-\frac{1}{3}}\mu_{0}\right\}.
	\end{aligned}
\end{equation}
For all compared methods, we { perform a grid search to} select the learning rate $\alpha$ from the set
$\{10^{-1}, 5 \times 10^{-1}, 10^{-2}, 5 \times 10^{-2}, 10^{-3}\}$.

\textbf{Performance comparisons:}
We measure model performance using the root mean square error (RMSE). 
Each model is evaluated on a modified version of the test set. Specifically, for each data point $\vx=(\va, b)$ in the test set,
the feature vector
$\va$ is perturbed according to $\va+\upsilon \boldsymbol{\omega}\|\va\|_2$,
where $\upsilon = 2$ and $\boldsymbol{\omega} \sim [\text{Laplace}(0,1)]^{q}$~\citep{goodfellow2016deep}, where Laplace represents the Laplace distribution.
{For each candidate learning rate, we run 5 random seeds and determine the best learning rate $\alpha$ for SSPG and GDMax based on the average RMSE. Additionally, for SDRO, we select the best-performing $\alpha$ using the same criterion, while the best $\lambda$ and $\eta$ are determined per seed.}
Finally, we report the mean value and standard deviation of RMSE and training time across
compared methods, as shown in Table~\ref{tab:algorithms2}.

\begin{table}[htbp]
	\centering
	\small
	\begin{tabular}{l|cc|cc|cc}
		\toprule
		\multirow{2}{*}{\textbf{Dataset}} & \multicolumn{2}{c|}{\textbf{SSPG}} & \multicolumn{2}{c|}{\textbf{GDMax}} & \multicolumn{2}{c}{\textbf{SDRO}} \\ \cmidrule(lr){2-3} \cmidrule(lr){4-5} \cmidrule(lr){6-7}
		& \textbf{RMSE} & \textbf{Time (s)} & \textbf{RMSE} & \textbf{Time (s)} & \textbf{RMSE} & \textbf{Time (s)} \\
		\midrule
		{\verb|Space GA|} & \textbf{0.41 $\pm$ 0.20} & 79.45  & 0.87 $\pm$ 0.84 & {51.31} & 0.50 $\pm$ 0.26 & 655.66 \\
		{\verb|MG|} & 
		0.38 $\pm$ 0.28 & 44.56 & 0.50 $\pm$ 0.53 & {21.14} & \textbf{0.24 $\pm$ 0.03} & 370.69  \\
		{\verb|BodyFat|} & \textbf{0.05 $\pm$ 0.03} & 20.02 & 0.10 $\pm$ 0.10 & {6.93} & 0.08 $\pm$ 0.06 & 161.25  \\
		\bottomrule
	\end{tabular}
	\caption{{RMSE (mean $\pm$ standard deviation) and time (mean) for the distributionally robust regression problem on perturbed test sets. 
	}}
	\label{tab:algorithms2}
\end{table}
From this table, we observe that SSPG achieves the lowest RMSE on two of the three datasets while also exhibiting smaller standard deviations, indicating superior error minimization and robust performance. SDRO, although it achieves the best performance on the \verb|MG| dataset, requires significantly longer runtime, limiting its practical advantage.
Overall, SSPG is an effective method for minimizing prediction errors and ensuring consistent performance, rendering it a good choice for robust regression tasks across diverse data environments.

\subsection{Adversarially Robust Learning}\label{sec:infinity_wasserstein_regression}
We next consider the $\infty$-Wasserstein robust regression formulation discussed in Section~\ref{sec:appli}. The problem  can be expressed as
\begin{align}
	\min_{\vtheta} \frac{1}{n} \sum_{i=1}^{n} \max_{\vz_i \in \mathcal{B}^{\infty}_{\hat{\epsilon}}(\va_i)} \ell(\vtheta, (\vz_i, b_{i}))~\text{ with }~ \ell(\vtheta, (\vz_i, b_{i})):= \left(h_{\vtheta}(\vz_i) - b_{i}\right)^{2}
	\label{eq:infinity_dro_objective}
\end{align}
where $h_{\vtheta}: \mathbb{R}^{m_2-1} \to \mathbb{R}$ is a neural network function parameterized by $\vtheta$, $ \va_{i} \in \mathbb{R}^{m_2-1}$ denotes a vector of features, $b_{i} \in \mathbb{R}$ denotes the corresponding target,  $ \mathcal{B}^{\infty}_{\hat{\epsilon}}(\va_i) $ denotes the \( \infty \)-ball centered at \( \va_i \) with radius~\(\hat{\epsilon}\). The neural network architecture that we employ is a single hidden layer containing five neurons, using the ReLU activation function~\citep{goodfellow2016deep}.

\textbf{Datasets:}
The datasets  are the same as those utilized in Section~\ref{sec:regre}.


\textbf{Algorithm parameters:}
For all methods, 
we initialize $\vtheta^{(0)}$ using PyTorch's default parameter initialization under a fixed random seed, and set \( \hat{\epsilon} = 0.25 \). For SSPG, we set \( \mu_0 = \lambda^{(0)} \eta \), with \( \eta \sim \mathcal{U}(0.1, 1) \). The step size for the inner problem is \( \hat{\epsilon} \) for FGSM, 
while for GDMax and PGD, we use \( \frac{\hat{\epsilon}}{2} \). Following \citep{wang2021sinkhorn}, we let $\mu = \lambda \eta$ and perform a grid search for SDRO to fine-tune the hyperparameters \( \lambda \) and \( \eta \) from the sets \{1, 5, 10\} and \mbox{\{0.1, 0.5, 1\}} respectively. Additionally, for each method, we perform a grid search to select the optimal learning rate $\alpha$, for the outer minimization problem, from the set $\{10^{-1}, 5 \times 10^{-2}, 10^{-2}\},$ and
the training is terminated after 250 epochs.

\textbf{Implementation of the Compared Methods at the \(k\)th Iteration:}
For the inner maximization problem in \eqref{eq:infinity_dro_objective}, each method first computes  
$\{\vz_i^{(k+1)}\}_{i=1}^{n} \subset \mathbb{R}^{m_2-1}$.
Specifically, for each method, we initialize from points obtained by adding random noise uniformly sampled from the interval $[-\hat{\epsilon}, \hat{\epsilon}]$ to $\va_i$.
For FGSM, the inner point is generated by one projected gradient step with stepsize $\hat\epsilon$. For PGD, GDMax, SSPG, and SDRO, the inner maximization is approximated by 20 projected gradient-ascent steps with stepsize $\hat\epsilon/2$, initialized from a uniformly perturbed point in $\mathcal B^\infty_{\hat\epsilon}(\va_i)$.


After obtaining $\{\vz_i^{(k+1)}\}_{i=1}^{n}$, GDMax updates  $\vtheta$ via  a projected gradient descent step using the gradient $\nabla_{\vtheta}\mathbb{E}_{\vx \sim \widehat{\mathbb{P}}_n}\left[\ell(\vtheta, (\vz^{(k+1)}_i, b_{i}))\right]$ and a fixed learning rate $\alpha>0$.
For SSPG and SDRO, we generate, for each $i\in[n]$, a set $\Omega_i^k$ of size $M=8$, containing samples near~$\vz_i^{(k+1)}$. Specifically, for any $\widehat{\vz}\in\Omega_i^k$, we let \( \widehat{\vz} = \mathrm{Proj}_{\mathcal{B}^{\infty}_{\hat{\epsilon}}(\va_i)}(\vz_i^{(k+1)} + \vr^{(k+1)}),\) where $\vr^{(k+1)})$ satisfies the normal distribution.
SDRO then performs a projected gradient descent step on  $\vtheta$ with gradient $\nabla_{\vtheta}{g}_s^k(\vtheta^{(k)},\lb)$ and a fixed learning rate~$\alpha$, where ${g}_s^k$ is given in~\eqref{eq:gradientgsk} with   $\delta=0$, $d\equiv 0$, and $\ell(\theta,\hat z)$ replaced by $\ell(\vtheta, (\widehat{\vz}, b_{i}))$.
Our SSPG method similarly applies a projected gradient descent step on $\vtheta$ with gradient $\nabla_{\vtheta}\widetilde{g}^k( \vtheta^{(k)},0, \mu_k)$ and a fixed learning rate~$\alpha$, where  $\widetilde{g}^k$ is given in~\eqref{eq:gradientgtildek} with $\ell(\theta,\hat z)$ replaced by $\ell(\vtheta, (\widehat{\vz}, b_{i}))$. We then update~$\mu$ by \eqref{eq:mu2}.

\textbf{Performance comparisons:}
We conducted experiments on the $\infty$-Wasserstein DRO problem following the same experimental settings as those outlined for the numerical tests in Table~\ref{tab:algorithms2}. Then, we report the mean values and standard deviations of the MSE and training time for all compared methods in Table~\ref{tab:algorithms_infinity_dro}.
\begin{table}[htbp]
	\centering
	\resizebox{\textwidth}{!}{\begin{tabular}{l|cc|cc|cc|cc|cc}
			\toprule
			\multirow{2}{*}{\textbf{Dataset}} & \multicolumn{2}{c|}{\textbf{SSPG}} & \multicolumn{2}{c|}{\textbf{GDMax}} &
			\multicolumn{2}{c}{\textbf{SDRO}} &
			\multicolumn{2}{c}{\textbf{FGSM}} &
			\multicolumn{2}{c}{\textbf{PGD}} \\ \cmidrule(lr){2-3} \cmidrule(lr){4-5} \cmidrule(lr){6-7} \cmidrule(lr){8-9}
			\cmidrule(lr){10-11}
			& \textbf{MSE} & \textbf{Time (s)} & \textbf{MSE} & \textbf{Time (s)} &
			\textbf{MSE} & \textbf{Time (s)} &
			\textbf{MSE} & \textbf{Time (s)} &
			\textbf{MSE} & \textbf{Time (s)} \\
			\midrule
			{\texttt{Space GA}} & \textit{0.214 $\pm$ 0.103} & 37.089 & 1.173 $\pm$ 0.574 &24.207 & 0.214 $\pm$ 0.103 & 288.771 & 0.231 $\pm$ 0.122 & 10.784 & \textbf{0.212 $\pm$ 0.103} & 14.492 \\
			{\texttt{MG}} & \textit{0.351 $\pm$ 0.094} & 19.829 & 1.036 $\pm$ 0.293 &12.555 & 0.351 $\pm$ 0.095 & 162.904 & 0.362 $\pm$ 0.104 & 6.267 & \textbf{0.344 $\pm$ 0.095} & 19.893 \\
			{\texttt{BodyFat}} & \textit{0.043 $\pm$ 0.067} & 10.551 & 0.785 $\pm$ 0.879 & 6.505 & 0.043 $\pm$ 0.066 & 89.754 & 0.063 $\pm$ 0.091 & 2.004 & \textbf{0.041 $\pm$ 0.064} & 11.624 \\
			\bottomrule
	\end{tabular}}
	\caption{{MSE (mean $\pm$ standard deviation) and time (mean) for the $\infty$-Wasserstein DRO   problem on perturbed test sets.
	}}
	\label{tab:algorithms_infinity_dro}
\end{table}
Table~\ref{tab:algorithms_infinity_dro} shows that PGD attains the lowest mean MSE on all three datasets, while SSPG is close to PGD and substantially faster than SDRO. This is expected because PGD is tailored to the $\ell_\infty$ perturbation model in~\eqref{eq:infinity_dro_objective}, whereas SSPG is designed for the broader expectation-over-maximization formulation.

\subsection{Adversarially Robust Deep Learning Problem}\label{sec:adversial}
We finally consider adversarially robust image classification, where the goal is to train a classifier that is robust to perturbations of the input images \citep{madry2018towards}. Given an image,  we use a neural network prediction function $h_{\vtheta}: \mathbb{R}^{l \times w \times h} \to \mathbb{R}^{m_3}$, parameterized by $\vtheta$, to predict the target class. 
The corresponding loss function and distance function in~\eqref{eq:model2},~\eqref{eq:smoothg}, and~\eqref{eq:smoothg2} are given by
\begin{align}
	\ell(\vtheta,\vx) = -\sum_{i=1}^{m_3} b_i \log \left( \frac{e^{\left[h_{\vtheta}\left(\textbf{a}\right)\right]_i}}{\sum_{j=1}^{m_3}e^{[h_{\vtheta}(\textbf{a})]_j}}\right), \text{ and } d(\mathbf{x}, \vz)  = \frac{1}{2} \|\mathbf{a} - \overline{\va}\|_{F}^{2} + \infty \|\vb - \overline{\vb}\|_1,
\end{align}
where $\vx = (\va, \vb)$, $\vz = (\overline{\va}, \overline{\vb})$ with $\va, \overline{\va} \in {\mathbb{R}}^{l \times w \times h}$ representing an image and  $\vb, \overline{\vb} \in \{0,1\}^{m_3}$ denoting their corresponding one-hot encoded label vectors. We set \(\delta = 10\), \(p = 2\), and \(\mathcal{Z} =[\min X_{\mathrm{train}},\max X_{\mathrm{train}}]^{l \times w \times h} \times \{0,1\}^{m_3}\).

\textbf{Datasets and neural network architectures:}
To evaluate the efficacy of the compared methods, we use two benchmark datasets: \verb|Fashion-MNIST| and \verb|CIFAR-10|. These datasets are loaded using \verb|torchvision| with the standard training/test split applied and the images are normalized using the usual mean-subtraction and standard-deviation scaling.

For \verb|Fashion-MNIST|, we utilize a convolutional neural network (CNN) architecture 
consisting of three convolutional layers with 32, 64, and 128 filters, each using a $3\times 3$ kernel, followed by ReLU activation and max pooling. The middle convolutional layer also employs dropout \citep{goodfellow2016deep} with a probability of 0.3 and batch normalization \citep{ioffe2015batch}. The output of the convolutional layers is then passed through a fully connected layer with 512 hidden units, followed by ReLU activation, dropout (probability 0.25), and batch normalization.

For \verb|CIFAR-10|, we adopt the All-CNN architecture \citep{springenberg2014striving}, incorporating batch normalization after each ReLU activation in every convolutional layer.

\textbf{Algorithm parameters:}
For all   methods, we initialize $\vtheta^{(0)}$ using PyTorch's default parameter initialization under a fixed random seed. 
In addition, for GDMax and SSPG, we let $\lambda^{(0)}  \sim \mathcal{U}(1, 10)$.
For SSPG, we set $\mu_0 = \lambda^{(0)}  \eta$, where $\eta \sim \mathcal{U}(0.1, 1)$.
We utilize a grid search for SDRO to fine-tune the hyperparameters $\lambda$ and $\eta$ from the sets \{1, 10\} and \{0.1, 1\}, respectively.
For each comparison method, the training is terminated after 100 epochs.

\textbf{Implementation of the compared methods at the $k$th iteration:}
For all methods, we first sample a mini-batch of size
$B=100$ from the training set, denoted as $\{\vx_{j_1^k},\vx_{j_2^k},\ldots,\vx_{j_B^k}\}$.
Next, for each~$i~\in~[B]$, we
perform an inner maximization step to obtain
$\{\vz_{j_i}^{(k+1)}\}_{i=1}^B$.
Specifically, we solve $B$ problems of the form $\max_{\vz \in \cZ} \ell(\vtheta^{(k)}, \vz) - \lambda^{(k)} d(\vx_{j_i^k}, \vz)$ using   projected gradient ascent with a fixed step size of~$10^{-2}$ for 15 iterations, starting from $\vx_{j_i^k}+10^{-3}\vr$ where $\vr$ follows the standard normal distribution.

After obtaining $\{\vz_{j_i}^{(k+1)}\}_{i=1}^B$, GDMax updates the   primal variable $(\vtheta,\lb)$ via  a projected gradient descent step with     $\nabla_{(\vtheta,\lb)} \frac{1}{B}\sum_{i=1}^B\left[\lb^{(k)}\delta^p+\ell(\vtheta^{(k)},\vz_{j_i^k}^{(k+1)})-\lambda^{(k)} d(\vx_{j_i^k},\vz_{j_i^k}^{(k+1)})\right]$ and a learning rate $\alpha_k>0$.

For SSPG and SDRO, we generate, for each $i\in\{1,2,\ldots,B\}$, a set~$\Omega_{j_i}^k$ of size $M=8$, containing samples near~$\vz_{j_i}^{(k+1)}$. Specifically, for any $\widehat{\vz}\in\Omega_{j_i}^k$, we form the immediate point \( \widehat{\vz} = \mathrm{Proj}_{{\mathcal{Z}}}(\vz_{j_i}^{(k+1)} + \vr^{(k+1)}),\) where $\vr^{(k+1)})$ satisfies the normal distribution.
For each $j_i$, we retain only the samples that improve upon $\vz_{j_i}^{(k+1)}$, defining the refined set as
\begin{align*}
	\overline{\Omega}_{j_i}^k:= \{\vz_{j_i}^{(k+1)}\} \cup \left\{\widehat{\vz}\in {\Omega}_{j_i}^k: \ell(\vtheta^{(k)}, \widehat{\vz}) - \lambda^{(k)} d(\vx_{j_i^k}, \widehat{\vz}) > \ell(\vtheta^{(k)}, \vz_{j_i}^{(k+1)}) - \lambda^{(k)} d(\vx_{j_i^k}, \vz_{j_i}^{(k+1)})\right\}.
\end{align*}
SDRO then updates  $\vtheta$ via a projected gradient descent step with the gradient $\nabla_{\vtheta}{g}_s^k(\vtheta^{(k)},\lb)$ and the learning rate $\alpha_k$,
where
\begin{align}
	\label{eq:gradientgsk2}
	{g}_s^k( \vtheta,\lb) = \lambda \delta^{p} + \lb\eta \frac{1}{B}\sum_{i=1}^B\left[ \log \left(\frac{1}{|\overline{\Omega}_{j_i}^k|} \sum_{\widehat{\vz}\in\overline{\Omega}_{j_i}^k}  \left[ e^{\frac{\ell(\vtheta, \widehat{\vz}) - \lambda d(\vx_{j_i^k}, \widehat{\vz})}{\lb\eta}} \right] \right)\right].
\end{align}
Our SSPG method conducts a projected gradient descent step on the primal variable $(\vtheta,\lb)$ using  $\nabla_{(\vtheta,\lb)}\widetilde{g}^k( \vtheta^{(k)},\lb^{(k)} , \mu_k)$ and the learning rate $\alpha_k$, where
\begin{align}
	\label{eq:gradientgtildek2}
	\widetilde{g}^k(\vtheta,\lb,\mu) = \lambda \delta^{p} + \mu \frac{1}{B}\sum_{i=1}^B\left[ \log \left(\frac{1}{|\overline{\Omega}_{j_i}^k|} \sum_{\widehat{\vz}\in\overline{\Omega}_{j_i}^k}  \left[ e^{\frac{\ell(\vtheta, \widehat{\vz}) - \lambda d(\vx_{j_i}^k, \widehat{\vz})}{\mu}} \right] \right)\right].
\end{align}
We then update~$\mu$ by~\eqref{eq:mu2}.
For all compared methods, we set $\alpha_k= \alpha \gamma^{\lfloor k/20 \rfloor}$, where we perform a grid search to choose $\alpha$ from the set
$\{1 \times 10^{-1}, 5 \times 10^{-1}, 5 \times 10^{-2}\}$, and $\gamma$ is  selected from
$\{0.5,0.9\}$.

\textbf{Performance comparisons:}
We evaluate model performance using accuracy and overall training time. Each model is tested on a modified version of the test set, where each feature vector is perturbed in a manner similar to the distributionally robust regression setting. Specifically, for each data point $\vx=(\va, \vb)$ in the test set,
we apply the perturbation $\va+\upsilon \boldsymbol{\omega}\|\va\|_2$,
where $\boldsymbol{\omega} \sim [\text{Laplace}(0,1)]^{q}$~\citep{goodfellow2016deep}, $\upsilon = 2 \times 10^{-3}$ for  \verb|Fashion-MNIST|  and $\upsilon= 2 \times 10^{-4} $ for \verb|CIFAR-10|.

To ensure a fair comparison,
we use five distinct random seeds for initialization across all three methods.
Given the large dataset sizes and high computational cost,
we first sample 20\% of the training set, ensuring that every class has the same number of samples, for hyperparameter tuning, optimizing $\alpha$, $\gamma$, $\lambda$, and $\eta$ from the specified choices in a manner similar to the one used for the regression problems. The best learning rate for SSPG and GDMax, as well as the hyperparameters for SDRO, are selected based on the highest accuracy.
After selecting the parameters, we apply each method to the full training set. Finally, we aggregate the results across the five seeds and report in Table~\ref{tab:algorithm_comparison_deep_learning} the mean accuracy and training time, along with their standard deviations, for all methods.

Table~\ref{tab:algorithm_comparison_deep_learning} shows that SSPG attains the highest mean accuracy on Fashion-MNIST among the tested methods, while SDRO attains a slightly higher mean accuracy on CIFAR-10. SSPG is substantially faster than SDRO and has accuracy comparable to the best-performing method in both datasets. Overall, SSPG achieves an effective balance between accuracy and computational efficiency.


\begin{table}[htbp]
	\centering
	\scriptsize
	\begin{tabular}{l|cc|cc|cc}
		\toprule
		\multirow{2}{*}{\textbf{Dataset}} & \multicolumn{2}{c|}{\textbf{SSPG}} & \multicolumn{2}{c|}{\textbf{GDMax}} & \multicolumn{2}{c}{\textbf{SDRO}} \\ \cmidrule(lr){2-3} \cmidrule(lr){4-5} \cmidrule(lr){6-7}
		& \textbf{Accuracy (\%)} & \textbf{Time (hrs)} & \textbf{Accuracy} & \textbf{Time (hrs)} & \textbf{Accuracy} & \textbf{Time (hrs)} \\
		\midrule
		{\verb|Fashion-MNIST|} & \textbf{93.09 $\pm$ 0.13} & 2.52 & 92.64 $\pm$ 0.07 & \textbf{2.04} & 92.85 $\pm$ 0.04 & 10.04  \\
		{\verb|CIFAR-10|} & 86.64 $\pm$ 0.11 & 4.20 & 86.41 $\pm$ 0.36 & \textbf{2.93} & \textbf{86.68 $\pm$ 0.08} & 13.41  \\
		\bottomrule
	\end{tabular}
	\caption{Comparison of accuracy (mean $\pm$ standard deviation) and time (mean) for the adversarially deep learning problem on the perturbed test set.}\label{tab:algorithm_comparison_deep_learning}
\end{table} 

\section{Conclusion}\label{sec:conclu}
We study nonconvex minEmax problems in which the objective is an expectation of pointwise maxima. We introduce an LME smoothing of the random value function and developed SSPG, a stochastic smoothing proximal-gradient method for the resulting sequence of smoothed problems.

Our analysis provides two main guarantees. First, stationarity of the smoothed problems can be converted into Goldstein stationarity of the original nonsmooth problem through explicit smoothing-gap bounds. Second, SSPG attains an $\epsilon$-scaled stationary point in expectation with iteration complexity $O(\epsilon^{-3})$ and sample complexity $O(\epsilon^{-5})$, and it attains an $(O(\epsilon),O(\epsilon))$-Goldstein stationary point with complexities $\widetilde O(\epsilon^{-4})$ and $\widetilde O(\epsilon^{-6})$, respectively. Under a refinement schedule with summable stationarity tolerances, any almost-sure cluster point of the generated approximate-stationary sequence is Clarke stationary, and hence directional stationary, for the original problem.

For WDRO, the derived dual-multiplier bound permits the dualized worst-case-risk problem to be treated on a compact feasible set. The numerical results on newsvendor, robust regression, $\infty$-Wasserstein/adversarially learning, and adversarially robust image classification demonstrate that SSPG is a practical first-order approach and is competitive with the tested baselines. 

	\vskip 0.2in
	\bibliography{dro}

\newpage

\appendix
\section{An Example Illustrating the Limitation of Value-Based Inner Accuracy}\label{sec:appenE}
 
 
 When solving $\min_{\vy}g(\vy)$, subgradient methods may exhibit zigzagging behavior and may fail to approach stationary points, especially when the selected subgradients are unstable,~\citet{li2025subgradient}. In the expectation-over-maximization setting, a natural heuristic is to approximate a subgradient of $g$ by $\nabla_{\vy}\Psi(\vy,\vz^\epsilon)$, where $\vz^\epsilon$ is a near-optimal solution of the inner maximization problem. The following example shows that inner value accuracy alone does not control the distance from this vector to the Clarke subdifferential of the outer objective.

\begin{example}[Inner value-accuracy does not imply outer subgradient-accuracy]\label{exam:2}
	
	Fix compact sets \(\cY=[-1,1]^2\) and \(\cZ=[-1,1]\). Let \(\eta\in(0,\epsilon]\).
	Choose a \(C^\infty\) bump function
	\[
	\psi(z)=
	\begin{cases}
		\exp\bigl(-\tfrac{1}{1-(2z)^2}\bigr)\big/\exp(-1), & |z|<\tfrac12,\\[2mm]
		0, & |z|\ge \tfrac12,
	\end{cases}
	\]
	so that \(0\le \psi\le 1\), \(\psi(0)=1\), and \(\psi(z)=0\) for \(|z|\ge\tfrac12\).
	Define \(u(z):=1-\eta\,\psi(z)\).
	
	Next define \(\vv:\cZ\to\mathbb{R}^2\) as a continuous map by
	\[
	\vv(z):=
	\begin{cases}
		(0,1)\zz, & |z|\le \tfrac14,\\
		(2-4|z|)(0,1)\zz+(4|z|-1)\,(\operatorname{sign}z,0)\zz, & \tfrac14<|z|<\tfrac12,\\
		(\operatorname{sign}z,0)\zz, & |z|\ge \tfrac12.
	\end{cases} 
	\]
	Set
	\[
	f(\vy,z):=u(z)+\langle \vy, \vv(z)\rangle,\qquad g(\vy):=\max_{z\in \cZ} f(\vy,z).
	\]
	At \(\vy=\mathbf{0}\), we claim that though the inner maximization can be solved to arbitrarily high accuracy in terms of function value, the corresponding gradient  \(\nabla_\vy f(\mathbf{0},z_\epsilon)\)—evaluated at such a high-accuracy solution~$z_\epsilon$—may still remain at unit distance from the subdifferential set \(\partial g(\mathbf{0})\). 
	Since \(u(z)\le 1\) with equality for all \(|z|\ge\tfrac12\), we have
	$
	g(\mathbf{0})=\max_{z\in \cZ} u(z)=1,
	$
	and the set of maximizers at \(\vy=\mathbf{0}\) includes all \(z\) with \(|z|\ge\tfrac12\).
	Because \mbox{\(\nabla_\vy f(\mathbf{0},z)=\vv(z)\)}, Danskin’s theorem gives
	\[
	\partial g(\mathbf{0})=\operatorname{conv}\{\vv(z):|z|\ge\tfrac12\}
	=\operatorname{conv}\{(1,0)\zz,(-1,0)\zz\}=\{(t,0)\zz:t\in[-1,1]\}.
	\]
	For any \(\eta\le \epsilon\), the point \(z_\epsilon:=0\) is \(\epsilon\)-\textbf{optimal} for the inner problem at \(\vy=\mathbf{0}\) because
	$
	g(\mathbf{0})-f(\mathbf{0},0)=\eta\le\epsilon.
	$
	However, it holds that 
	$
	\nabla_\vy f(\mathbf{0},0)=\vv(0)=(0,1)\zz,$ and $
	\operatorname{dist}\bigl((0,1)\zz,\partial g(\mathbf{0})\bigr)=1.
	$
	Thus, an arbitrarily small inner value gap \(g(\vy)-f(\vy,z_\epsilon)\) does {not} control the distance from \(\nabla_\vy f(\vy,z_\epsilon)\) to \(\partial g(\vy)\); there is no modulus \(\phi(\epsilon)\to0\) ensuring
	\(\operatorname{dist}\bigl(\nabla_\vy f(\vy,z_\epsilon),\partial g(\vy)\bigr)\le \phi(\epsilon)\)
	from value accuracy alone.
	
	We now illustrate the smoothing-gradient approach described in Algorithm~\ref{alg:spg2} by considering a numerical experiment starting at the point $\vy=\mathbf{0}$. Specifically, at each iteration, 
	we draw $M$ independent samples $\{z_i\}_{i=1}^M$ from the uniform reference measure $\zeta$ on $\mathcal Z=[-1,1]$, and then compute the corresponding smoothing gradient defined as
	\[
	\cG(\vy,\mu)=\nabla_\vy \tilde g(\vy,\mu)=
	\frac{\mathbb{E}_{z\sim\mathbb{P}}[\,e^{f(\vy,z)/\mu}\,\nabla_\vy f(\vy,z)\,]}{\mathbb{E}_{z\sim\mathbb{P}}[\,e^{f(\vy,z)/\mu}\,]}
	\,\approx\,
	\frac{\sum_{i=1}^M e^{f(\vy,z_i)/\mu}\,\nabla_\vy f(\vy,z_i)}{\sum_{i=1}^M e^{f(\vy,z_i)/\mu}}.
	\]
	To examine its behavior as the smoothing parameter $\mu$ decreases and $M$ increases, we set $\eta=\varepsilon=0.01$ and compute numerical approximations of $\cG(\vy,\mu)$ for various pairs $(\mu,M)$. The results obtained are:
	\[
	\begin{aligned}
		(\mu,M)&=(0.001,1000), && \cG(\vy,\mu)\approx(-0.0605,\,0.0057);\\
		(\mu,M)&=(0.002,800),  && \cG(\vy,\mu)\approx(-0.0141,\,0.0214);\\
		(\mu,M)&=(0.005,500),  && \cG(\vy,\mu)\approx(0.0036,\,0.1208);\\
		(\mu,M)&=(0.010,100),  && \cG(\vy,\mu)\approx(-0.0065,\,0.3291).
	\end{aligned}
	\]
	These numerical values are consistent with the smoothing-gradient behavior predicted by the construction: as $\mu$ decreases and $M$ increases, the computed gradient becomes small near the origin. This contrasts with the value-based inner maximizer $z_\epsilon=0$, whose associated outer gradient remains away from $\partial g(\mathbf 0)$. 
\end{example}

		\section{Boundedness of the Dual Multiplier of the Worst-case Risk Problem in WDRO}\label{appen:A}
		{
			In this section, we present a lemma establishing an explicit bound $B_{\lb}$, ensuring that every optimal solution~\(\lb^*\)	of problem \eqref{eq:worstriskdual} satisfies \(\lb^*\in[0,B_{\lb}]\).
			
			\begin{lemma}\label{lem:lambda-bound}
				Consider problem \eqref{eq:worstriskdual} at a given parameter $\vtheta$, with $d(\vz_1,\vz_2)=\|\vz_1-\vz_2\|_p^p$, $p\ge 1$ and $\delta>0$. Assume the loss function $\ell(\vtheta,\cdot)$ is $L$-Lipschitz continuous for all $\vtheta\in\Theta$. Then, any optimal solution $\lambda^\star$ of problem \eqref{eq:worstriskdual} has an upper bound 
				\[
				\lambda^\star \le L\,C_{p,m_2}\,\delta^{-(p-1)},\quad \text{where}\quad
				C_{p,m_2}=\sup_{\vv \in \RR^{m_2}, \vv\neq \vzero}\frac{\|\vv\|_2}{\|\vv\|_p}=\begin{cases}
					1, & 1\le p\le 2,\\[0.6ex]
					m_2^{\frac12-\frac1p}, & p>2.
				\end{cases}
				\]
			\end{lemma}
			\proof
			Define the objective
			\[
			J(\lambda):=\lambda\delta^p+\mathbb{E}_{\vx\sim \mathbb{P}} [\phi(\lambda;\vx)],\quad\text{with}\quad \phi(\lambda;\vx):=\max_{\vz\in\mathcal{Z}}\{\ell(\vtheta,\vz)-\lambda\|\vz-\vx\|_p^p\}.
			\] 
			Since each $\phi(\lambda;\vx)$ is a pointwise supremum of affine functions in $\lambda$, it is convex and non-increasing, hence $J(\lambda)$ is convex.  
			
			\paragraph{Case of $p>1$.} By Danskin's theorem~\citep{bertsekas1997nonlinear},
			\begin{align}
				\partial\phi(\lambda;\vx)
				&= -\,\operatorname{co}\Big\{\,\|\vz-\vx\|_p^p \ :\ \vz\in\argmax_{\vu\in\cZ}\big(\ell(\vtheta,\vu)-\lambda\|\vu-\vx\|_p^p\big)\Big\}.
				\label{eq:partiallb}
			\end{align}
			Let $\vz(\lambda;\vx)\in\argmax_{\vz\in\cZ}\{\ell(\vtheta,\vz)-\lambda\|\vz-\vx\|_p^p\}$, and define $t(\lambda;\vx)=\|\vz(\lambda;\vx)-\vx\|_p$. 
			
			Fix $\vx\in\cX$. We now prove 
			\begin{align}
				\label{eq:partialphi}
				\partial\phi(\lambda;\vx)\ \subseteq\ \Big[-\Big(\tfrac{L\,C_{p,m_2}}{\lambda}\Big)^{\frac{p}{p-1}},\ 0\Big], \text{ for all }\lb>0.
			\end{align} 
			If every active maximizer satisfies $\vz=\vx$, then all active slopes in~\eqref{eq:partiallb} are zero, and hence $\partial q(\lambda;\vx)=\{0\}$. Thus, we conclude that $\partial\phi(\lambda;\vx) = \{0\}$, 
			which  then implies~\eqref{eq:partialphi}.
			Otherwise, there exists $z(\lambda;\vx)\in\cZ$ such that $\vz(\lambda;\vx)\neq\vx$. We then consider the case when
			$t(\lambda;\vx)=\|\vz(\lambda;\vx)-\vx\|_p>0$. Due to the fact that $\vx\in\cZ$, 
			it follows directly from optimality of  $\vz(\lambda;\vx)$ that
			$
			\ell(\vtheta,\vz(\lambda;\vx))-\lambda t(\lambda;\vx)^p \;\ge\; \ell(\vtheta,\vx).
			$
			Using $L$-Lipschitz continuity of $\ell(\vtheta,\cdot)$  and $\|\vv\|_2\le C_{p,m_2}\|\vv\|_p$ for all $\vv\in\RR^{m_2}$, we arrive at
			\begin{equation}\label{eq:basic-ti}
				\lambda\,t(\lambda;\vx)^p \;\le\; \ell(\vtheta,\vz(\lambda;\vx))-\ell(\vtheta,\vx)
				\;\le\; L\,\|\vz(\lambda;\vx)-\vx\|_2
				\;\le\; L\,C_{p,m_2}\,t(\lambda;\vx).
			\end{equation} 
			For $\lambda>0$, dividing \eqref{eq:basic-ti} by $t(\lambda;\vx)$ yields
			$
			t(\lambda;\vx)^{p-1}\ \le\ \frac{L\,C_{p,m_2}}{\lambda}
			$, and hence $
			-t(\lambda;\vx)^p \ge- \Big(\frac{L\,C_{p,m_2}}{\lambda}\Big)^{\frac{p}{p-1}}.
			$
			Since this bound holds for every maximizer, we have \eqref{eq:partialphi}.
			
			By the optimality condition of problem \eqref{eq:worstriskdual}, 
			any minimizer $\lambda^*$ satisfies
			\[
			0\in \partial J(\lambda^*)+\mathcal{N}_{[0,\infty)}(\lambda^*).
			\]
			If $\lambda^*>0$, then $0\in\partial J(\lambda^*)=\delta^p+\partial \mathbb{E}_{\vx\sim \mathbb{P}}[\phi(\lambda^*;\vx)]\subseteq \delta^p+\mathbb{E}_{\vx\sim \mathbb{P}}[\partial\phi(\lambda^*;\vx)]$, where the second inclusion comes from~\citep[Thm.~1]{burke2020subdifferential}.  Therefore, there exists a measurable selection
			$s(\lambda^*;\vx)\in\partial\phi(\lambda^*;\vx)$ with
			\[
			0=\delta^p+\mathbb{E}_{\vx\sim\mathbb{P}}[s(\lambda^*;\vx)]
			\ \ge\ \delta^p-\Big(\tfrac{L\,C_{p,m_2}}{\lambda^*}\Big)^{\frac{p}{p-1}}.
			\]
			Hence $\lambda^*\le L\,C_{p,m_2}\,\delta^{-(p-1)}$. If $\lambda^*=0$, the bound is trivial. 
			
			\paragraph{Case of $p=1$.}
			For any $\vz\in\cZ$,
			\[
			\ell(\vtheta,\vz)-\lambda\|\vz-\vx\|_1
			\le \ell(\vtheta,\vx)+L\|\vz-\vx\|_2-\lambda\|\vz-\vx\|_1
			\le \ell(\vtheta,\vx)+(L-\lambda)\|\vz-\vx\|_1.
			\]
			Thus $\phi(\lambda;\vx)\le \ell(\vtheta,\vx)$ for all $\lambda\ge L$. Since $\vz=\vx\in\cZ$, 
			it holds
			$\phi(\lambda;\vx)= \ell(\vtheta,\vx)$ for $\lambda\ge L$. Therefore
			$
			J(\lambda)=\lambda\,\delta+\mathbb{E}_{\vx\sim\mathbb{P}}[\ell(\vtheta,\vx)]
			$
			is strictly increasing on $[L,\infty)$ (because $\delta>0$), so any minimizer satisfies $\lambda^*\le L$ (equivalently, $\lambda^*\le L\,C_{1,m_2}\,\delta^{0}$ with $C_{1,m_2}=1$).
			
			The proof is then completed by combining the above two cases.
			\endproof
		}

		\section{Methods to Generate a Stochastic Gradient Estimator}\label{sec:C}
In this section, we detail two situations under which the gradient estimation condition~\eqref{eq:gradinetbias2} can be satisfied. As noted in Remark~\ref{rem:exact},
if, for all $\vx\in\cX$, 
the function \(\Psi(\vy, \cdot; \vx)\) exhibits certain structures, we can efficiently generate a stochastic gradient estimator.
For simplicity of notations, we fix $\vx\in\cX$ 
in this section, and
define $H(\vz):=\Psi(\vy, \vz; \vx)$.

\subsection{Computing a Stochastic Gradient Estimator with a Specific Loss Function and Support Set}\label{appen:piece} 
As noted in Remark~\ref{rem:exact}, 
if  we can compute 
\(\mathbb{E}_{\vz\sim \zeta}[e^{H(\vz)/\mu}]\)
for all \(\mu > 0\), we can generate a desired stochastic gradient estimator.
We now describe several cases, in which this expectation can be computed. When \(H(\cdot)\) is a linear function and $\zeta$ is the uniform distribution over its support set, 
we can calculate  \(\mathbb{E}_{\vz\sim \zeta}[e^{H(\vz)/\mu}]\). 
This computation often reduces to evaluating the integrals of \(e^{\va^\top \vz + c}\) over \(\mathcal{Z}\).
 Notably, \citep{Barvinok1992ComputingTV} demonstrates that for several structured sets $\mathcal Z$, such exponential integrals admit efficient finite-dimensional formulas or algorithms. Specifically, the authors provide efficient methods for the following four cases:
\begin{enumerate}
	\item[(i)] \(\mathcal{Z}\) is a regular simplex, i.e., \(\mathcal{Z}=\left\{\vz\in \mathbb{R}_{+}^{m_2}: \mathbf{1}_{m_2}\zz\vz=1\right\}\);
	\item[(ii)] \(\mathcal{Z} = [0,r]^{m_2}\) is a \(m_2\)-dimensional cube;
	\item[(iii)] \(\mathcal{Z}=\operatorname{conic}\left\{\vu_1, \ldots, \vu_s\right\} \subset \mathbb{R}^{m_2}\) is a   simple convex cone given as the conic hull of its extreme rays \(\vu_1, \ldots, \vu_s \in \RR^{m_2}\);
	\item[(iv)] $\mathcal{Z}$ is an intersection of several half-spaces.
\end{enumerate}

\subsection{A Sampling Method to Generate a Stochastic Gradient Estimator with a Connected Compact set $\cZ$}
\label{app:localized-small-ball-centered}
In this section, we consider the case where $\cZ$ is a connected compact set, and introduce an algorithmic framework that outputs an approximate gradient $\mathcal{G}_{k_j}(\vy^{(k)},\mu_k)$ satisfying condition~\eqref{eq:gradinetbias2}, i.e.,
\begin{equation}\label{eq:gradinetbias2-restate}
	\mathbb{E}\!\big[\|\mathcal{G}_{k_j}(\vy^{(k)},\mu_k)
	-\nabla_{\vy}\widetilde\Phi(\vy^{(k)},\mu_k;\vx_{k_j})\|^2\big]
	\le \widehat\epsilon_k^2.
\end{equation}
The construction is used at a fixed outer iteration $k$ and a fixed sample $\vx_{k_j}$. For simplicity, in the remaining part of this section, unless with further specification, we denote and fix 
\begin{equation}\label{eq:fix-x-y-mu}
(\vx,\vy,\mu)=(\vx_{k_j},\vy^{(k)},\mu_k).    
\end{equation} We will build a localized estimator for
\(\nabla_{\vy}\widetilde\Phi(\vy,\mu;\vx)\) and choose its sample size and localization radius so that the mean-square error is at most \(\widehat\epsilon_k^2\).

The estimator analyzed here is a localized version of the LME Gibbs estimator.
We define 
\begin{equation}
	\label{eq:app-gap-definition}
\begin{aligned}    
	&F(\vz) = F_{\vx,\vy}(\vz)
	:=
	\Phi(\vy;\vx)-\Psi(\vy,\vz;\vx)
	=
	\max_{\vu\in\cZ}\Psi(\vy,\vu;\vx)-\Psi(\vy,\vz;\vx)\ge0, \\
    & w_\mu(\vz):=\exp\!\left(-\mu^{-1}{F(\vz)} \right).
\end{aligned}    
\end{equation}
Then we have
$
	\exp\!\left(\mu^{-1}{\Psi(\vy,\vz;\vx)} \right)
	=
	\exp\!\left(\mu^{-1}{\Phi(\vy;\vx)} \right)
	w_\mu(\vz).
$
For a localization radius \(r>0\), define the near-optimal level set
\begin{equation}
	\label{eq:app-local-set}
	    \cU_r=\cU_r(\vx,\vy):=
    \big\{\vz\in\cZ:F_{\vx,\vy}(\vz)\le r\big\},
\end{equation}
and let \(\nu_{r}\) be the conditional uniform law on \(\cU_r\):
\begin{equation}
	\label{eq:app-conditional-law}
	\nu_{r}(A):=
	\frac{\zeta(A\cap\cU_r)}{\zeta(\cU_r)},
	\quad \text{for any measurable set }A\subseteq\cZ.
\end{equation}
Moreover, we define
	$$
		D_r=\int_{\cU_r}e^{-F(\vz)/\mu}\,\zeta(d\vz),
		\text{ and }
		D_r^{\rm out}=\int_{\cZ\setminus\cU_r}e^{-F(\vz)/\mu}\,\zeta(d\vz).
$$
Notice that \(D_r\) is the contribution of the local region \(\cU_r\)
to the Gibbs partition function, and \(D_r^{\rm out}\) is the
contribution from the complement \(\cZ\setminus \cU_r\).

\begin{algorithm}[t]
	\caption{A method to generate a gradient estimator}
	\label{alg:app-local-small-ball}
	\begin{algorithmic}[1]
		\Require Fixed sample \(\vx\), outer iterate \(\vy\), 
        {smoothing parameter} \(\mu\), radius \(r\), sample size \(M\), and access to exact i.i.d. samples from \(\nu_{r}\) defined in \eqref{eq:app-conditional-law}.
		\State Draw \(\vz_1,\ldots,\vz_M\stackrel{\rm i.i.d.}{\sim}\nu_{r}\).
		\State Return
		$
		G_{\vx,r}(\vy,\mu)
        = \frac{\sum_{j=1}^M \exp(\mu^{-1}{\Psi(\vy,\vz_j;\vx)} )\nabla_{\vy}\Psi(\vy,\vz_j;\vx)}
		{\sum_{j=1}^M \exp(\mu^{-1}{\Psi(\vy,\vz_j;\vx)} )}
        .$ 
	\end{algorithmic}
\end{algorithm}

With the above notations, we present the method to generate a stochastic gradient estimator in Algorithm~\ref{alg:app-local-small-ball}. 
Equivalently, the output of Algorithm~\ref{alg:app-local-small-ball} equals
\begin{equation}
	\label{eq:app-local-estimator}
	G_{\vx,r}(\vy,\mu)
	=
	\frac{\sum_{j=1}^M
		\exp\!\left(-\mu^{-1}F(\vz_j)\right)
		\nabla_{\vy}\Psi(\vy,\vz_j;\vx)}
	{\sum_{j=1}^M
		\exp\!\left(-\mu^{-1}F(\vz_j)\right)} = \frac{\sum_{j=1}^M w_\mu(\vz_j)\nabla_{\vy}\Psi(\vy,\vz_j;\vx)}
		{\sum_{j=1}^M w_\mu(\vz_j)}.
\end{equation}

To guarantee that \(\mathcal{G}_{k_j}(\vy^{(k)},\mu_k):=G_{\vx,r}(\vy,\mu)\) with notation given in \eqref{eq:fix-x-y-mu} 
satisfies the condition in~\eqref{eq:gradinetbias2-restate}, we need a  
flat-maximum condition assumed below.  It is a positive-measure near-optimal-level-set condition. Related positive-measure optimality regions appear in random search through the essential optimum \citep{solis1981minimization}; near-optimality dimension in optimistic optimization and \(X\)-armed bandits quantifies the size of near-optimal sets~\citep{munos2011optimistic,bubeck2011xarmed}; in continuous evolutionary optimization it is explicitly described as a level set of positive Lebesgue measure~\citep{glasmachers2020global}. 
Further discussions are provided in
Remarks~\ref{rem:app-eb-vs-lower-growth}--%
\ref{rem:app-flat-plateau-and-isolated-maximizers}.

\begin{assumption}[Flat maximum]
    \label{ass:app-localized-eb}
Suppose that $\cZ$ is a connected compact set. There exists a constant $c_{\rm flat}\in(0,1]$ such that, for every $\vx\in\cX$, $\vy\in\dom(\varphi)$, $\mu\in(0,1]$, and every $r\ge\mu$,
\begin{equation}
    \label{eq:app-flat-maximum-assumption}
    \zeta\!\left(\cU_r\right) = \zeta\!\left(\cU_r(\vx,\vy)\right)
    \ge c_{\rm flat}\mu.
\end{equation}
\end{assumption}



The next theorem gives the bound on the difference between \(G_{\vx,r}(\vy,\mu)\) and \(\nabla_{\vy}\widetilde\Phi(\vy,\mu;\vx)\) in expectation.

\begin{theorem}
    \label{lem:app-localized-centered-s1}
Suppose Assumptions~\ref{ass:problemsetup},~\ref{ass:compact1}, and~\ref{ass:app-localized-eb} hold. Fix $\vy\in\dom(\varphi)$ and $\mu\in(0,1]$. Choose $r\ge\mu$ and $M\in\mathbb N$, and let $G_{\vx,r}(\vy,\mu)$ be the output of Algorithm~\ref{alg:app-local-small-ball}. 
Then
\begin{align}
    \label{eq:app-total-error-bound}
    &\mathbb{E}\!\left[\left\|G_{\vx,r}(\vy,\mu)-\nabla_{\vy}\widetilde\Phi(\vy,\mu;\vx)\right\|^2\right] \notag\\
    &\quad\le
    \frac{32l_{\Psi}^2}{M\kappa_{\rm flat}}
    +8l_{\Psi}^2\exp\!\left(-\frac{M\kappa_{\rm flat}}{8}\right)
    +8l_{\Psi}^2\left(\frac{D_r^{\rm out}}{D_r+D_r^{\rm out}}\right)^2 \notag\\
    &\quad\le
    \frac{32l_{\Psi}^2}{M\kappa_{\rm flat}}
    +8l_{\Psi}^2\exp\!\left(-\frac{M\kappa_{\rm flat}}{8}\right)
    +\frac{8l_{\Psi}^2}{c_{\rm flat}^2\mu^2}
      \exp\!\left(-2\left(\frac{r}{\mu}-1\right)\right),
\end{align}
where
$
    \kappa_{\rm flat}:=e^{-1}c_{\rm flat}\mu.
$
Consequently, for any target accuracy $\widehat\epsilon>0$, if
\begin{equation}
    \label{eq:app-radius-under-eb}
    r\ge
    \mu\left(1+
    \left[\log\!\left(\frac{4l_{\Psi}}{c_{\rm flat}\mu\widehat\epsilon}\right)\right]_+
    \right)
\end{equation}
and
\begin{equation}
    \label{eq:app-fixed-mu-sample-size}
    M\ge
    \left\lceil
    \max\left\{
    \frac{128e\,l_{\Psi}^2}{c_{\rm flat}\mu\widehat\epsilon^2},
    \frac{8e}{c_{\rm flat}\mu}
    \left[\log\!\left(\frac{32l_{\Psi}^2}{\widehat\epsilon^2}\right)\right]_+
    \right\}
    \right\rceil,
\end{equation}
then $\mathbb{E}\!\left[\left\|G_{\vx,r}(\vy,\mu)-
    \nabla_{\vy}\widetilde\Phi(\vy,\mu;\vx)\right\|^2\right]
    \le\widehat\epsilon^2.$
\end{theorem}

\begin{proof}
We define the localized  counterpart of
\eqref{eq:app-local-estimator} by   
\begin{equation}
	\label{eq:app-local-population}
	g_{\vx,r}(\vy,\mu)
	:=
	\frac{\int_{\cU_r(\vy)}
		\exp\!\left(-\mu^{-1}F(\vz)\right)
		\nabla_{\vy}\Psi(\vy,\vz;\vx)\,\zeta(d\vz)}
	{\int_{\cU_r(\vy)}
		\exp\!\left(-\mu^{-1}F(\vz)\right)\,\zeta(d\vz)}.
\end{equation} 
Also, let the full Gibbs distribution and its localization to \(\cU_r\) be
$$
		\pi(d\vz):=
		\frac{e^{-F(\vz)/\mu}}{D_r+D_r^{\rm out}}\,\zeta(d\vz)
		\text{ and }
		\pi_r(d\vz):=
		\frac{e^{-F(\vz)/\mu}\mathbf{1}_{\cU_r}(\vz)}{D_r}\,\zeta(d\vz),
$$
respectively. By the definition of \(\widetilde\Phi\) in \eqref{eq:smoothphi}, its gradient
with respect to \(\vy\) is obtained by differentiating the corresponding
LME function. Under the interchange of differentiation
and integration, this gives $\nabla_{\vy}\widetilde\Phi(\vy,\mu;\vx)
		=
		\int_{\cZ}\nabla_{\vy}\Psi(\vy,\vz;\vx)\,\pi(d\vz)$.  {In addition, by the definition of $\pi_r$, it is straightforward to have}  
	\begin{equation}
		\label{eq:app-gradient-population-as-gibbs} 
		g_{\vx,r}(\vy,\mu)=
		\int_{\cZ}\nabla_{\vy}\Psi(\vy,\vz;\vx)\,\pi_r(d\vz).
	\end{equation}
%
Using \(\|\nabla_{\vy}\Psi(\vy,\vz;\vx)\|\le l_{\Psi}\) and the
definitions of \(\pi\) and \(\pi_r\), we obtain
\begin{align}
	&\big\|g_{\vx,r}(\vy,\mu)
	-\nabla_{\vy}\widetilde{\Phi}(\vy,\mu;\vx)\big\| \notag\\
	= &
	\left\|
	\frac{D_r^{\rm out}}{D_r+D_r^{\rm out}}
	\int_{\cU_r}
	\nabla_{\vy}\Psi(\vy,\vz;\vx)\,\pi_r(d\vz)
	-
	\frac{1}{D_r+D_r^{\rm out}}
	\int_{\cZ\setminus\cU_r}
	e^{-F(\vz)/\mu}
	\nabla_{\vy}\Psi(\vy,\vz;\vx)\,\zeta(d\vz)
	\right\| \notag\\
	\le &
	l_{\Psi}\frac{D_r^{\rm out}}{D_r+D_r^{\rm out}}
	+
	\frac{l_{\Psi}}{D_r+D_r^{\rm out}}
	\int_{\cZ\setminus\cU_r}e^{-F(\vz)/\mu}\,\zeta(d\vz) \notag\\
	= &
	\frac{2l_{\Psi}D_r^{\rm out}}{D_r+D_r^{\rm out}}.
	\label{eq:app-local-sampling-bias-flat}
\end{align}
For every $\vz\notin\cU_r$, $F(\vz)>r$, and hence
\begin{equation}
    \label{eq:app-outside-denominator}
    D_r^{\rm out}
    \le e^{-r/\mu}\zeta(\cZ\setminus\cU_r)
    \le e^{-r/\mu}.
\end{equation}
Because $r\ge\mu$, one has $\cU_\mu\subseteq\cU_r$. On $\cU_\mu$, it holds that $e^{-F(\vz)/\mu}\ge e^{-1}$, and Assumption~\ref{ass:app-localized-eb} gives $\zeta(\cU_\mu)\ge c_{\rm flat}\mu$. Thus
\begin{equation}
    \label{eq:app-inside-denominator-flat}
    D_r\ge\int_{\cU_\mu}e^{-F(\vz)/\mu}\,\zeta(d\vz)
    \ge e^{-1}c_{\rm flat}\mu
    =\kappa_{\rm flat}.
\end{equation}
Combining the above two inequalities gives
\begin{equation}
    \label{eq:app-explicit-tail-under-eb}
    \big\|g_{\vx,r}(\vy,\mu)
	-\nabla_{\vy}\widetilde{\Phi}(\vy,\mu;\vx)\big\| \le \frac{2l_{\Psi} D_r^{\rm out}}{D_r+D_r^{\rm out}}\le\frac{2l_{\Psi} D_r^{\rm out}}{D_r}
    \le\frac{2l_{\Psi}}{c_{\rm flat}\mu}
       \exp\!\left(1-\frac{r}{\mu}\right).
\end{equation}

	It remains to bound the sampling error $\|G_{\vx,r}(\vy,\mu)-g_{\vx,r}(\vy,\mu)\|^2$.      
Let
$m_r:=\mathbb E_{\nu_r}[w_\mu(\vz)]
    =\frac{D_r}{\zeta(\cU_r)}.$
Since $\zeta(\cU_r)\le1$,~\eqref{eq:app-inside-denominator-flat} implies $m_r\ge\kappa_{\rm flat}$. Define
\[
    \overline W:=\frac1M\sum_{j=1}^M w_\mu(\vz_j),
    \qquad
    \overline\xi:=\frac1M\sum_{j=1}^M w_\mu(\vz_j)
    \big\{\nabla_{\vy}\Psi(\vy,\vz_j;\vx)-g_{\vx,r}(\vy,\mu)\big\}.
\]
Then \(G_{\vx,r}(\vy,\mu)-g_{\vx,r}(\vy,\mu)=\overline\xi/\overline W\), \(\mathbb{E}_{\nu_{r}}[\overline\xi]=0\), and \(\|\nabla_{\vy}\Psi(\vy,\vz;\vx)-g_{\vx,r}(\vy,\mu)\|\le2l_{\Psi}\).  Moreover, since \(0\le w_\mu\le1\) {and $\{\vz_j\}$ are i.i.d.,}
\begin{equation}
    \label{eq:app-xi-second-moment-flat}
    \mathbb E_{\nu_r}\!\left[\|\overline\xi\|^2\right]=
		\frac1M\mathbb{E}_{\nu_{r}}\!\big[w_\mu(\vz)^2\|\nabla_{\vy}\Psi(\vy,\vz;\vx)-g_{\vx,r}(\vy,\mu)\|^2\big]
    \le\frac{4l_{\Psi}^2m_r}{M}.
\end{equation}
Let $\mathcal E$ be the event $\{\overline W\ge m_r/2\}$. On $\mathcal E$,
$\|\overline\xi/\overline W\|^2\le4m_r^{-2}\|\overline\xi\|^2$, and therefore
\begin{equation}
    \label{eq:app-sampling-good-event-flat}
    \mathbb E_{\nu_r}\!\left[
    \|G_{\vx,r}(\vy,\mu)-g_{\vx,r}(\vy,\mu)\|^2\mathbf 1_{\mathcal E}
    \right]\le
	\frac{4}{m_r^2}
	\mathbb{E}_{\nu_{r}}\|\overline\xi\|^2
    \le\frac{16l_{\Psi}^2}{Mm_r}
    \le\frac{16l_{\Psi}^2}{M\kappa_{\rm flat}}.
\end{equation}
	By the multiplicative Chernoff bound for independent random variables in \([0,1]\)~\citep[Theorem~1.1]{dubhashi2009concentration}, we have
	\begin{equation}
		\label{eq:app-denominator-bad-event-flat}
		\mathbb{P}(\mathcal E^c)
		=\mathbb{P}(\overline W<m_r/2)
		\le
		\exp\!\left(-\frac{Mm_r}{8}\right)
		\le
		\exp\!\left(-\frac{M\kappa_{\rm flat}}{8}\right).
	\end{equation}
	Since both \(G_{\vx,r}(\vy,\mu)\) and \(g_{\vx,r}(\vy,\mu)\) are linear combinations of vectors with norm at most \(l_{\Psi}\), their distance is at most \(2l_{\Psi}\).  Combining this fact with \eqref{eq:app-sampling-good-event-flat} and \eqref{eq:app-denominator-bad-event-flat} gives
	\begin{equation}
		\label{eq:app-local-sampling-error-flat}
		\mathbb{E}_{\nu_{r}}\left[\big\|G_{\vx,r}(\vy,\mu)-g_{\vx,r}(\vy,\mu)\big\|^2\right]
		\le
		\frac{16l_{\Psi}^2}{M\kappa_{\rm flat}}
		+
		4l_{\Psi}^2\exp\!\left(-\frac{M\kappa_{\rm flat}}{8}\right).
	\end{equation}
	Combining \eqref{eq:app-explicit-tail-under-eb} and \eqref{eq:app-local-sampling-error-flat} with \(\|a+b\|^2\le2\|a\|^2+2\|b\|^2\) proves \eqref{eq:app-total-error-bound}.
\end{proof}

 \begin{remark}
\label{rem:app-eb-vs-lower-growth}
Theorem~\ref{lem:app-localized-centered-s1} separates the sampling error from the localization bias. Under Assumption~\ref{ass:app-localized-eb}, the required sample size is
$M=\widetilde O(\mu^{-1} \widehat\epsilon^{-2})$.
\end{remark}

\begin{remark}
\label{rem:app-flat-plateau-and-isolated-maximizers}
A simple sufficient condition for
	\eqref{eq:app-flat-maximum-assumption} is a uniform plateau at the
	maximum, namely there exists \(c_{\rm flat}>0\) such that
	$
		\zeta(\cS_{\vx,\vy})\ge c_{\rm flat}
		\text{ for all } \vx \text{ and } \vy ,
	 $
where $ \cS_{\vx,\vy}:=\Argmax_{\vz\in\cZ}\Psi(\vy,\vz;\vx)$.     
	Indeed, \(F_{\vx,\vy}=0\) on \(\cS_{\vx,\vy}\), and
	\(\cS_{\vx,\vy}\subseteq \cU_r(\vx,\vy)\) for every \(r>0\).  Hence this
	plateau condition implies \eqref{eq:app-flat-maximum-assumption}.
The plateau condition, however, is not suitable when the maximizer set
	has zero reference measure.  This case arises, for example, when
	\(\Psi(\vy,\cdot;\vx)\) is strongly concave as its   maximizer is
	unique.  In such settings, we consider an alternative condition.
    
Suppose that, for some $p>0$ and $c_p>0$,
\begin{equation}
    \label{eq:app-polynomial-volume-growth}
    \zeta(\cU_t(\vx,\vy))\ge c_p t^p,
    \quad \forall\, 0<t\le1.
\end{equation}
Then the proof remains valid with
$\kappa_{\rm flat}$ replaced by $e^{-1}c_p\mu^p$. Consequently,
$M=\widetilde O(\mu^{-p}\widehat\epsilon^{-2})$, and it is sufficient to take
$
    r\ge\mu\left(1+
    \left[\log\!\left(\frac{4l_{\Psi}}{c_p\mu^p\widehat\epsilon}\right)\right]_+
    \right)
$
{in Algorithm~\ref{alg:app-local-small-ball}.} 
This formulation covers isolated maximizers. For example, suppose that a maximizer $\vz^*$ satisfies the local growth bound $F(\vz)\le L\|\vz-\vz^*\|^s$ and the reference measure satisfies
$\zeta(B(\vz^*,\rho)\cap\cZ)\ge\kappa\rho^q$. Then, for all sufficiently small $t$, it holds $\zeta(\cU_t)\ge\kappa L^{-q/s}t^{q/s}.$
Thus the exponent is $p=q/s$; the case $p=1$ in Assumption~\ref{ass:app-localized-eb} is the linear-growth specialization.
\end{remark}

	\section{Proofs of the Main Results}\label{sec:proof}

In this section we provide proofs of our results presented in Section~\ref{sec:algorithm}.

	\subsection{Proof of Lemma~\ref{lem:smoothphi}}\label{sec:smoothphi}
	\proof
	(a) 
	We take $\mu\downarrow 0$ in $\mu \log \mathbb{E}_{\vz\sim \zeta} [e^{\Psi(\vy,\vz;\vx)/\mu}]$ to obtain
	\begin{align}
		\notag
		& \lim _{\mu \downarrow 0} \mu \log \mathbb{E}_{\vz\sim \zeta} [e^{\Psi(\vy,\vz;\vx)/\mu} ]= \lim _{\beta \rightarrow \infty} \frac{1}{\beta} \log \mathbb{E}_{\vz\sim \zeta}[e^{\beta\Psi(\vy,\vz;\vx) }]  \stackrel{(i)}{=}\lim_{\beta \rightarrow \infty}  \nabla_\beta \log \mathbb{E}_{\vz\sim \zeta}[e^{\beta\Psi(\vy,\vz;\vx)  }]   \\
		= & {\lim _{\beta \rightarrow \infty} \frac{ \nabla_\beta \mathbb{E}_{\vz\sim \zeta}[e^{\beta\Psi(\vy,\vz;\vx)  } ]}{\mathbb{E}_{\vz\sim \zeta}[e^{\beta\Psi(\vy,\vz;\vx)  }] }   \stackrel{(ii)}{=}
			\lim _{\beta \rightarrow \infty} \frac{  \mathbb{E}_{\vz\sim \zeta}[\nabla_\beta e^{\beta\Psi(\vy,\vz;\vx)  } ]}{\mathbb{E}_{\vz\sim \zeta}[e^{\beta\Psi(\vy,\vz;\vx)  }] }   }
		=
		\lim _{\beta \rightarrow \infty} \frac{ \mathbb{E}_{\vz\sim \zeta}[e^{\beta\Psi(\vy,\vz;\vx)  } \Psi(\vy,\vz;\vx)]}{\mathbb{E}_{\vz\sim \zeta}[e^{\beta\Psi(\vy,\vz;\vx)  }] }  
		\label{eq:limitu}
		\\\notag 
		=
		&\lim _{\beta \rightarrow \infty} \frac{ \mathbb{E}_{\vz\sim \zeta}[e^{\beta\Psi(\vy,\vz;\vx) -\beta \max_{\vz\in \mathrm{supp}(\zeta)}\Psi(\vy,\vz;\vx)  } \Psi(\vy,\vz;\vx)]}{\mathbb{E}_{\vz\sim \zeta}[e^{\beta\Psi(\vy,\vz;\vx) -\beta \max_{\vz\in \mathrm{supp}(\zeta)}\Psi(\vy,\vz;\vx) }] }  
		\stackrel{(iii)}{=}  \max_{\vz \in \mathrm{supp}(\zeta)}\Psi(\vy,\vz;\vx) \stackrel{(iv)}{=}\max _{\vz \in \mathcal{Z}}\Psi(\vy,\vz;\vx).
	\end{align}
	Here (i) comes from L'Hôpital's rule~\citep{taylor1952hospital}.
	When $\mathcal{Z}$ is a finite set, (ii) holds directly; and when $\mathcal{Z}$ is a connected compact  set, (ii) holds by Leibniz integral rule~\citep{flanders1973differentiation} and the continuity of $e^{\beta\Psi(\vy,\vz;\vx)  } \Psi(\vy,\vz;\vx)$ with respect to~$\beta$.
	(iii) results from the fact that 
	$\vz$ {is a continuous random variable},   
	\begin{equation*}
		\begin{aligned}
			&\lim _{\beta \rightarrow \infty} e^{\beta\Psi(\vy,\vz;\vx) -\beta \max_{\vz\in \mathrm{supp}(\zeta)}\Psi(\vy,\vz;\vx)} = \left\{
			\begin{aligned}
				1, &\text{ if } \vz \in\argmax_{\vz\in \mathrm{supp}(\zeta)}\Psi(\vy,\vz;\vx) \\ 
				0, & \text{ otherwise}.
			\end{aligned}
			\right., \text{ and}\\
			&\lim _{\beta \rightarrow \infty} e^{\beta\Psi(\vy,\vz;\vx) -\beta \max_{\vz\in \mathrm{supp}(\zeta)}\Psi(\vy,\vz;\vx)} \Psi(\vy,\vz;\vx) = \left\{
			\begin{aligned}
				\Psi(\vy,\vz;\vx), &\text{ if } \vz \in\argmax_{\vz\in \mathrm{supp}(\zeta)}\Psi(\vy,\vz;\vx) \\ 
				0, \quad\quad& \text{ otherwise}.
			\end{aligned}
			\right.,
		\end{aligned}
	\end{equation*}
	In addition, (iv) follows from the definition of $\zeta$.
	
	Notice that $\mu \log \mathbb{E}_{\vz\sim \zeta} [e^{\Psi(\vy,\vz;\vx)/\mu} ]$ is continuously differentiable with respect to~$\vy$ by Assumption~\ref{ass:problemsetup}. We then obtain that  $\widetilde{\Phi}(\cdot,\cdot;\vx)$ is a smoothing function of $\Phi(\cdot;\vx)$ from Definition~\ref{def:smooth}.
	
	(b) The $\vy$ and $\mu$ partial gradients of $\widetilde{\Phi}(\cdot,\cdot;\vx)$ are derived by direct calculation. 
	For the $\vy$-partial gradient, we have
	\begin{align*}
		\left\|\nabla_{\vy} 
		\widetilde{\Phi}(\vy,\mu;\vx)\right\| &=  \left\|\frac{\mathbb{E}_{\vz\sim \zeta}[e^{\Psi(\vy,\vz;\vx) / \mu}  \nabla_{\vy}\Psi(\vy,\vz;\vx)] }{\mathbb{E}_{\vz\sim \zeta}[e^{\Psi(\vy,\vz;\vx) / \mu}] }\right\| 
		\leq \frac{\mathbb{E}_{\vz\sim \zeta}\left[\left\|e^{\Psi(\vy,\vz;\vx) / \mu}  \nabla_{\vy}\Psi(\vy,\vz;\vx)\right\| \right]  }{\left\|\mathbb{E}_{\vz\sim \zeta}[e^{\Psi(\vy,\vz;\vx) / \mu}]\right\|  }\\
		&\leq \frac{\mathbb{E}_{\vz\sim \zeta}\left[l_{\Psi}\left\|e^{\Psi(\vy,\vz;\vx) / \mu}   \right\| \right]  }{\left\|\mathbb{E}_{\vz\sim \zeta}[e^{\Psi(\vy,\vz;\vx) / \mu}]\right\|  }= l_{\Psi}.
	\end{align*}
	For the $\mu$-partial gradient, 
	it holds
	$\lim_{ \mu \downarrow 0} \mu \nabla_\mu \widetilde{\Phi}(\mathbf{y}, \mu;\vx) =0$ by~\eqref{eq:limitu}.
	
	Let $\sigma = \frac{\mu_2}{\mu_1}\leq 1$. We then have 
	\begin{equation*}
		\begin{aligned}
			&\mu_1 \log \left(\mathbb{E}_{\vz\sim \zeta} [e^{\Psi(\vy,\vz;\vx)/\mu_1} ] \right)  - \mu_2 \log \left(\mathbb{E}_{\vz\sim \zeta} [e^{\Psi(\vy,\vz;\vx)/\mu_2} ]\right)  \\ =& \mu_1 \log \left(\mathbb{E}_{\vz\sim \zeta} [e^{\Psi(\vy,\vz;\vx)/\mu_1} ]\right)  - \sigma\mu_1 \log \left(\mathbb{E}_{\vz\sim \zeta} [e^{\Psi(\vy,\vz;\vx)/(\sigma\mu_1)} ]\right) 
			 \\= & \mu_1 \log \frac{\mathbb{E}_{\vz\sim \zeta} [e^{\Psi(\vy,\vz;\vx)/\mu_1} ]}{\left(\mathbb{E}_{\vz\sim \zeta} [e^{\Psi(\vy,\vz;\vx)/(\sigma\mu_1)} ]\right)^\sigma} \leq 0,		\end{aligned}
	\end{equation*} 
	where the inequality follows from Jensen's inequality using $\mathbb{E}\left[X^{1/\sigma}\right] \geq (\mathbb{E}[X])^{1/\sigma}$
	with $X=e^{\Psi(\vy,\vz;\vx)/\mu_1}$. Thus $\widetilde{\Phi}(\vy,\mu_1;\vx)\leq \widetilde{\Phi}(\vy,\mu_2;\vx)$ for any $\mu_1\geq\mu_2>0$.
	
	(c)  Fix \(\vx\in\cX\), \(\mu>0\), and \(\vy_1,\vy_2\in\operatorname{dom}(\varphi)\). 
	For each \(\vy\in\operatorname{dom}(\varphi)\), define the probability measure \(Q_\vy\) on \(\mathcal Z\) by
	\[
	\mathrm{d}Q_\vy(\vz)
	=
	\frac{\exp(\Psi(\vy,\vz;\vx)/\mu)}
	{\mathbb{E}_{\vz\sim\zeta}[\exp(\Psi(\vy,\vz;\vx)/\mu)]}\,\zeta(d\vz).
	\]
	Then
	$
	\nabla_\vy\widetilde\Phi(\vy,\mu;\vx)
	=
	\mathbb{E}_{\vz\sim Q_\vy}[\nabla_\vy\Psi(\vy,\vz;\vx)].
	$
	Let \(\vd:=\vy_1-\vy_2\), \(g_i(\vz):=\nabla_\vy\Psi(\vy_i,\vz;\vx)\) for \(i=1,2\), and \(Q_i:=Q_{\vy_i}\). We decompose
	\[
	\begin{aligned}
		&\|\nabla_\vy\widetilde\Phi(\vy_1,\mu;\vx)-\nabla_\vy\widetilde\Phi(\vy_2,\mu;\vx)\|  \le 
		\|\mathbb{E}_{Q_1}[g_1(\vz)-g_2(\vz)]\|
		+
		\|\mathbb{E}_{Q_1}[g_2(\vz)]-\mathbb{E}_{Q_2}[g_2(\vz)]\|.
	\end{aligned}
	\]
	By Assumption~\ref{ass:lk}, the first term is bounded by
	\begin{align}
		\label{eq:bound1}
			\|\mathbb{E}_{Q_1}[g_1(\vz)-g_2(\vz)]\|
		\le L_\Psi\|\vy_1-\vy_2\|.
	\end{align}
	It remains to bound the second term. Let \(\vy_t:=\vy_2+t\vd\) for \(t\in[0,1]\), and let \(Q_t:=Q_{\vy_t}\). Define
	$
	H(t):=\mathbb{E}_{\vz\sim Q_t}[g_2(\vz)].
	$
	Since \(\|\nabla_\vy\Psi(\vy,\vz;\vx)\|\le l_\Psi\) by Assumption~\ref{ass:lk}, differentiation under the expectation is justified by dominated convergence. 
    By the standard covariance identity for normalized exponential tilts
    \citep[Proposition~3.1]{wainwright2008graphical}, for all \(t\in[0,1]\), we have
	\[
	H'(t)
	=
	\frac{1}{\mu}
	\left(
	\mathbb{E}_{Q_t}\left[g_2(\vz)\left\langle \nabla_\vy\Psi(\vy_t,\vz;\vx),\vd\right\rangle\right]
	-
	\mathbb{E}_{Q_t}[g_2(\vz)]\,
	\mathbb{E}_{Q_t}\left[\left\langle \nabla_\vy\Psi(\vy_t,\vz;\vx),\vd\right\rangle\right]
	\right).
	\]
	For any unit vector \(\vv\), the Cauchy--Schwarz inequality together with Assumption~\ref{ass:compact1} gives
	\[
	\begin{aligned}
		|\vv^\top H'(t)|
		&=
		\frac{1}{\mu}
		\left|
		\operatorname{Cov}_{Q_t}
		\left(
		\vv^\top g_2(\vz),
		\left\langle \nabla_\vy\Psi(\vy_t,\vz;\vx),\vd\right\rangle
		\right)
		\right|  \\
		&\le
		\frac{1}{\mu}
		\sqrt{
			\operatorname{Var}_{Q_t}(\vv^\top g_2(\vz))
			\operatorname{Var}_{Q_t}
			\left(
			\left\langle \nabla_\vy\Psi(\vy_t,\vz;\vx),\vd\right\rangle
			\right)
		}  \\
		&\le 	\frac{1}{\mu} {  \|\vv\|\|g_2(\vz)\| \| \nabla_\vy\Psi(\vy_t,\vz;\vx)\| \|\vd\|}
		\le
		\frac{l_\Psi^2}{\mu}\|\vd\|.
	\end{aligned}
	\]
	Taking the supremum over all unit vectors \(\vv\), we obtain
	$
	\|H'(t)\|
	\le
	\frac{l_\Psi^2}{\mu}\|\vy_1-\vy_2\|.
	$
	Therefore,
	\begin{align}
		\label{eq:bound2}
		\|\mathbb{E}_{Q_1}[g_2(\vz)]-\mathbb{E}_{Q_2}[g_2(\vz)]\|
		=
		\|H(1)-H(0)\|
		\le
		\int_0^1 \|H'(t)\|\,dt
		\le
		\frac{l_\Psi^2}{\mu}\|\vy_1-\vy_2\|.
	\end{align}
	Combining~\eqref{eq:bound1} and~\eqref{eq:bound2} yields
	\[
	\|\nabla_\vy\widetilde\Phi(\vy_1,\mu;\vx)-\nabla_\vy\widetilde\Phi(\vy_2,\mu;\vx)\|
	\le
	\left(L_\Psi+\frac{l_\Psi^2}{\mu}\right)\|\vy_1-\vy_2\|.
	\]
	The proof is then completed.
	\endproof
	
	\subsection{Proof of Lemma~\ref{lem:distancemufinite}}\label{sec:distancemufinite}
	\proof 
	Define $\omega_t:\RR^t \mapsto\RR$  and $\widetilde{b}_t:\RR^t\times \RR \mapsto\RR$ by $$\omega_t(\vx):=\log  \left( \sum_{i=1}^t e^{x_i} \right), \text{ and }\widetilde{b}_t(\vy,\mu):= \mu \log \left(\sum_{i=1}^t e^{y_i / \mu}\right).$$ 
	We have
	$
	\widetilde{b}_t(\vy,\mu) 
	= \mu \omega_t\left(\frac{\vy}{\mu}\right)= \max _{\vx\in D_t}\left\{\langle \vx, \vy\rangle -\mu \omega_t^*(\vx)\right\},
	$
	where $D_t:=\{\mathbf{x} \in \mathbb{R}^t: \mathbf{x} \geq \mathbf{0}, \textbf{1}_t\zz \mathbf{x}=1\}$, and the second equality follows from that   the conjugate function of $\omega$ over   $D_t$ is $\omega_t^*(\vy)=\sum_{i=1}^t y_i \log y_i$
	(cf. \citep[Theorem 4.2]{beck2012smoothing}). 
	Then		
	\begin{align}
\notag
			& \left|\widetilde{b}_t(\vy,\mu_1)-\widetilde{b}_t(\vy,\mu_2)\right|  = \left| \max _{\vx\in D_t }\left\{\langle \vx, \vy\rangle -\mu_2 \omega_t^*(\vx)\right\}-\max _{\vx\in D_t }\left\{\langle \vx, \vy\rangle -\mu_1 \omega_t^*(\vx)\right\}\right|  \\ \notag
			= & \max\bigg\{ \max _{\vx\in D_t }\left\{\langle \vx, \vy\rangle -\mu_2 \omega_t^*(\vx)\right\}-\max _{\vx\in D_t }\left\{\langle \vx, \vy\rangle -\mu_1 \omega_t^*(\vx)\right\}, \\ &\quad\quad\quad\quad\quad \max _{\vx\in D_t }\left\{\langle \vx, \vy\rangle -\mu_1 \omega_t^*(\vx)\right\} - \max _{\vx\in D_t }\left\{\langle \vx, \vy\rangle -\mu_2 \omega_t^*(\vx)\right\}\bigg\}\notag
			\\\notag
			\leq & \max\left\{ \max _{\vx\in{D_t}} (\mu_1-\mu_2)\omega_t^*(\vx), \max _{\vx\in{D_t}} (\mu_2-\mu_1)\omega_t^*(\vx)\right\}\\ \notag
			= &\left(\mu_1-\mu_2\right) \max _{\vx\in{D_t}}-\omega_t^*(\vx)=
			\left(\mu_1-\mu_2\right) \left|\max _{\vx\in{D_t}}[\langle\vzero, \vx\rangle-\omega_t^*(\vx)]\right|
			 \\ = & \omega_t(\vzero)\left(\mu_1-\mu_2\right) =\log(t)\left(\mu_1-\mu_2\right),
			 		\label{eq:boundbtilde}
	\end{align}
	where the fourth equality uses the fact that the conjugate function of $\omega_t^*$ is $\omega_t$ itself, and the inequality holds because for any continuous functions $f_1, f_2: \RR^t \rightarrow \mathbb{R}$,
	$$
	\max _{\vu}\left\{f_1(\vu)-f_2(\vu)\right\}+\max _{\vu}\left\{f_2(\vu)\right\} \geq \max _{\vu }\left\{f_1(\vu)-f_2(\vu)+f_2(\vu)\right\}=\max _{\vu}\left\{f_1(\vu)\right\}.
	$$ 
	Since $\mathcal{Z}$ is a finite discrete set, we let $t=|\mathcal Z|$ and $r_j =\Psi(\vy,\vz_j;\vx)$  for all $j\in[t]$. We have $$\widetilde{\Phi}(\vy,\mu;\vx)=\mu \log \mathbb{E}_{\vz\sim \zeta} [e^{\Psi(\vy,\vz;\vx)/\mu}]  = \mu\log \frac{1}{t} + \mu \log \sum_{j=1}^t e^{\Psi(\vy,\vz_j;\vx)/\mu} =\widetilde{b}_t(\vr,\mu)+\mu\log\frac{1}{t}.$$  Thus it follows from~\eqref{eq:boundbtilde} that
	\begin{align*}
		|\widetilde{\Phi}(\vy,\mu_1;\vx) - \widetilde{\Phi}(\vy,\mu_2;\vx)| 
		\leq & |\widetilde{b}_t(\vr,\mu_1) - \widetilde{b}_t(\vr,\mu_2)| +  \left|\log\frac{1}{t} \right| \left(\mu_1-\mu_2\right)
		\\
		\leq &(\log(t)+\log({t}))(\mu_1-\mu_2)= 2\log(t)(\mu_1-\mu_2),
	\end{align*}
	which indicates the desired result.  
	\endproof

\subsection{Proof of Lemma~\ref{lem:distancemu}}\label{sec:distancemu}
Since $\cZ$ is a connected compact  set endowed with normalized Lebesgue
measure, 
it admits diameter-bounded equal-measure partitions. Hence, there exists a
constant $C_{\cZ}>0$, depending only on $\cZ$ and $m_2$, such that for every
$\rho\in(0,1]$, one can find a measurable partition
\[
\cZ=\bigcup_{i=1}^{M_\rho}\mathcal A_i,
\qquad
\zeta(\mathcal A_i)=\frac{1}{M_\rho},
\qquad
\operatorname{diam}(\mathcal A_i)\le \rho,
\]
with
\[
M_\rho\le \left(\frac{C_{\cZ}}{\rho}\right)^{m_2}.
\]
Here $C_{\cZ}$ is a geometry-dependent constant; for instance, one may take
$C_{\cZ}=1+2 D_{\cZ}$ with $D_{\cZ}$ being the diameter of $\cZ$. See, e.g.,~\citep{rogers1997covering} for covering bounds of connected compact sets and~\citep{gigante2017diameter} for diameter-bounded equal-measure
partitions. In the subsequent lemma, we provide a more detailed and precise formulation of Lemma~\ref{lem:distancemu}.

\begin{lemma}
	Suppose Assumptions~\ref{ass:problemsetup}, \ref{ass:compact1} and~\ref{ass:lk} hold, and $\cZ$ is a connected compact set in
	$\mathbb{R}^{m_2}$. Then, for any $\rho\in(0,1]$,
	$\vy\in \operatorname{dom}(\varphi)$, $1\ge \mu_1>\mu_2>0$, and all $\vx\in\cX$, 
    it holds that
	\[
	\left|
	\widetilde{\Phi}(\vy,\mu_1;\vx)
	-
	\widetilde{\Phi}(\vy,\mu_2;\vx)
	\right|
	\le
	2l_\Psi\rho
	+
	2m_2\log\!\left(\frac{1+2 D_{\cZ}}{\rho}\right)(\mu_1-\mu_2),
	\]
	where $D_{\cZ}$ denotes the diameter of $\cZ$.
	In particular, by setting $\rho=\mu_1-\mu_2$, we obtain
	\[
	\left|
	\widetilde{\Phi}(\vy,\mu_1;\vx)
	-
	\widetilde{\Phi}(\vy,\mu_2;\vx)
	\right|
	\le
	\left[
	2l_\Psi
	+
	2m_2\log\!\left(\frac{1+2 D_{\cZ}}{\mu_1-\mu_2}\right)
	\right](\mu_1-\mu_2).
	\]
\end{lemma}

\begin{proof}
	Fix $\rho\in(0,1]$, $\vy\in\operatorname{dom}(\varphi)$,
	$1\ge \mu_1>\mu_2>0$, and $\vx\in\cX$. 
    Let
	$\{\mathcal A_i\}_{i=1}^{M_\rho}$ be an equal-$\zeta$-{measure}  
    partition of $\cZ$ such that
	\[
	\zeta(\mathcal A_i)=\frac{1}{M_\rho},
	\qquad
	\operatorname{diam}(\mathcal A_i)\le \rho,
	\text{ and }
	M_\rho\le \left(\frac{1+2 D_{\cZ}}{\rho}\right)^{m_2}.
	\]
	Choose an arbitrary representative $\vz_i\in\mathcal A_i$ for each $i\in[M_\rho]$, and define
	the finite discrete LME approximation
	\[
	\widetilde{\Phi}_\rho(\vy,\mu;\vx)
	:=
	\mu\log\!\left(
	\frac{1}{M_\rho}
	\sum_{i=1}^{M_\rho}
	\exp\!\left(\frac{\Psi(\vy,\vz_i;\vx)}{\mu}\right)
	\right),
	\qquad \mu>0.
	\]
	Since $\Psi(\vy,\cdot;\vx)$ is $l_\Psi$-Lipschitz continuous on $\cZ$, for every
	$\vz\in\mathcal A_i$, it holds
	$$
	|\Psi(\vy,\vz;\vx)-\Psi(\vy,\vz_i;\vx)|\le l_\Psi\rho .
	$$
	Using $\zeta(\mathcal A_i)=1/M_\rho$, we get, for every $\mu>0$,
	\[
	e^{-l_\Psi\rho/\mu}
	\frac{1}{M_\rho}
	\sum_{i=1}^{M_\rho}
	e^{\Psi(\vy,\vz_i;\vx)/\mu}
	\le
	\mathbb E_{\vz\sim\zeta}
	\left[e^{\Psi(\vy,\vz;\vx)/\mu}\right]
	\le
	e^{l_\Psi\rho/\mu}
	\frac{1}{M_\rho}
	\sum_{i=1}^{M_\rho}
	e^{\Psi(\vy,\vz_i;\vx)/\mu}.
	\]
	Taking $\mu\log(\cdot)$ {to all sides of the above inequalities} gives
	$
	\left|
	\widetilde{\Phi}(\vy,\mu;\vx)
	-
	\widetilde{\Phi}_\rho(\vy,\mu;\vx)
	\right|
	\le
	l_\Psi\rho,
	$ for all $ \mu>0.
	$
	
	The function $\widetilde{\Phi}_\rho$ is exactly the LME smoothing over the finite
	discrete set $\{\vz_i\}_{i=1}^{M_\rho}$ with the uniform distribution. Therefore, applying
	Lemma~\ref{lem:distancemufinite} with $|\cZ_\rho|=M_\rho$ yields
	\[
	\left|
	\widetilde{\Phi}_\rho(\vy,\mu_1;\vx)
	-
	\widetilde{\Phi}_\rho(\vy,\mu_2;\vx)
	\right|
	\le
	2\log(M_\rho)(\mu_1-\mu_2).
	\]
	Combining the last two displays by the triangle inequality, we obtain
	\[
	\begin{aligned}
		\left|
		\widetilde{\Phi}(\vy,\mu_1;\vx)
		-
		\widetilde{\Phi}(\vy,\mu_2;\vx)
		\right|
		&\le
		2l_\Psi\rho
		+
		2\log(M_\rho)(\mu_1-\mu_2) \\
		&\le
		2l_\Psi\rho
		+
		2m_2\log\!\left(\frac{1+2 D_{\cZ}}{\rho}\right)(\mu_1-\mu_2),
	\end{aligned}
	\]
	where the last inequality follows from
	$M_\rho\le (1+2 D_{\cZ}/\rho)^{m_2}$. Setting
	$\rho=\mu_1-\mu_2$ gives the desired bound in Lemma~\ref{lem:distancemu}.
\end{proof}

\subsection{Proof of Lemma~\ref{lem:smoothed-to-goldstein}}\label{sec:smoothed}

\begin{proof}
	By Assumption~\ref{ass:lk}, for every \(\vz\in\cZ\) and every \(\vx\in\cX\),
    \(\nabla_{\vy}\Psi(\cdot,\vz;\vx)\) is \(L_\Psi\)-Lipschitz continuous.
	Hence
	$
	\Psi(\cdot,\vz;\vx)+\frac{L_\Psi}{2}\|\cdot\|^2
	$
	is convex, or equivalently, \(\Psi(\cdot,\vz;\vx)\) is \(L_\Psi\)-weakly convex.
	Since pointwise maxima preserve weak convexity, it follows that
	$
	\Phi(\cdot;\vx)
	=
	\max_{\vz\in\cZ}\Psi(\cdot,\vz;\vx)
	$
	is \(L_\Psi\)-weakly convex.
	The LME smoothing also preserves weak convexity. Indeed,
	\[
	\widetilde\Phi(\vy,\mu;\vx)
	+
	\frac{L_\Psi}{2}\|\vy\|^2
	=
	\mu
	\log
	\mathbb E_{\vz\sim\zeta}
	\left[
	\exp
	\left(
	\frac{
		\Psi(\vy,\vz;\vx)+\frac{L_\Psi}{2}\|\vy\|^2
	}{\mu}
	\right)
	\right],
	\]
	which is a log-mean-exp form of convex functions. Taking expectation with respect to
	\(\vx\sim P\) and adding the convex function \(\varphi\), 
    we obtain that both \(g\) and
	\(\widetilde g(\cdot,\mu)\) {defined in \eqref{eq:smootho}} are \(L_\Psi\)-weakly convex.
	
	By Lemma~5(b) and the definition of \(\Delta_\mu\) in \eqref{eq:omega}, for every
	\(\vu\in\operatorname{dom}(\varphi)\), it holds that 
	$
	0
	\le
	g(\vu)-\widetilde g(\vu,\mu)
	\le
	\Delta_\mu .
	$
	Fix any \(\vxi_\mu\in\partial\widetilde g(\vy^*,\mu)\). Since
	\(\widetilde g(\cdot,\mu)\) is \(L_\Psi\)-weakly convex, 
    it holds that for all \(\vu\in\operatorname{dom}(\varphi)\),
	\begin{align}
    \label{eq:ggelpsi}
	    	\widetilde g(\vu,\mu)
	\ge
	\widetilde g(\vy^*,\mu)
	+
	\langle \vxi_\mu,\vu-\vy^*\rangle
	-
	\frac{L_\Psi}{2}\|\vu-\vy^*\|^2 .
	\end{align} 
	Combining~\eqref{eq:ggelpsi} with $
	0
	\le
	g(\vu)-\widetilde g(\vu,\mu)
	\le
	\Delta_\mu
	$, we arrive at
	\[
	g(\vu)
	\ge
	g(\vy^*)
	+
	\langle \vxi_\mu,\vu-\vy^*\rangle
	-
	\frac{L_\Psi}{2}\|\vu-\vy^*\|^2
	-
	\Delta_\mu ,
	\qquad
	\forall \vu\in\operatorname{dom}(\varphi).
	\]
For \(\rho>0\), define
\[
	H(\vu)
	:=
	g(\vu)
	-
	\langle \vxi_\mu,\vu-\vy^*\rangle
	+
	\frac{L_\Psi}{2}\|\vu-\vy^*\|^2 .
\]
Then \eqref{eq:ggelpsi} implies
$
	H(\vu)\ge H(\vy^*)-\Delta_\mu,
	\forall\,\vu\in\operatorname{dom}(\varphi).
$
Thus \(\vy^*\) is a \(\Delta_\mu\)-minimizer of \(\min_{\vy}H(\vy)\).  
Without loss of generality, we assume \(\Delta_\mu>0\).

Since \(g\) is \(L_\Psi\)-weakly convex, the function \(H\) is proper, lower
semicontinuous, and convex.  By Ekeland's
variational principle \citep[Proposition~1.43]{roc1998var}, for
every \(r>0\) there exists \(\vy_r\in\dom(\varphi)\) such that  
\begin{equation}
	\label{eq:ekeland-distance}
	\|\vy_r-\vy^*\|\le r, 	\ H(\vy_r)\le H(\vy^*), \text{ and }\operatorname{Argmin}_\vy\left\{H(\vy)+\frac{\Delta_\mu}{r}\|\vy-\vy_r\|\right\}=\{\vy_r\} .
\end{equation}
Therefore,
\[
	\dist\bigl(\vzero,\partial H(\vy_r)\bigr)
	\le \dist \left( \vzero,  \frac{\Delta_\mu}{r}\partial_{\vy} \|\vy-\vy_r\|\mid_{\vy=\vy_{r}} \right) = \dist \left( \vzero, 
	\frac{\Delta_\mu}{r}\left\{\vb\in\mathbb R^{m_1}:\|\vb\|\le1\right\} \right) \le 
	\frac{\Delta_\mu}{r}.
\]
Since we have
$
	\partial H(\vy_r)
	=
	\partial g(\vy_r)
	-
	\vxi_\mu
	+
	L_\Psi(\vy_r-\vy^*),
$ there exists \(\vxi_r\in\partial g(\vy_r)\) such that
\[
	\big\|
	\vxi_r
	-
	\vxi_\mu
	+
	L_\Psi(\vy_r-\vy^*)
	\big\|
	\le
	\frac{\Delta_\mu}{r}.
\]
Since \(\|\vy_r-\vy^*\|\le r\), we have
\(\vxi_r\in\partial^r g(\vy^*)\).  
Consequently,
\[
	dist\bigl(\vxi_\mu,\partial^r g(\vy^*)\bigr) \le \|\vxi_\mu-\vxi_r\|
	\le
	L_\Psi\|\vy_r-\vy^*\|
	+
	\frac{\Delta_\mu}{r}
	\le
	L_\Psi r
	+
	\frac{\Delta_\mu}{r}.
\]
Choose  
	$
	r
	=
	\sqrt{\frac{\Delta_\mu}{L_\Psi}} \le\mu_g.
	$
	Because \(r\le\mu_g\), we have
	\(\partial^r g(\vy^*)\subseteq\partial^{\mu_g} g(\vy^*)\). Hence 
	\[
	\operatorname{dist}
	\left(
	\vxi_\mu,\partial^{\mu_g} g(\vy^*)
	\right)
	\le
	{L_\Psi r}{}
	+
	\frac{\Delta_\mu}{r}
	=
	2\sqrt{L_\Psi\Delta_\mu}.
	\]
	Since the above bound holds for every
	\(\vxi_\mu\in\partial\widetilde g(\vy^*,\mu)\), we obtain 
	\[
	\operatorname{dist}
	\left(
	\vzero,\partial^{\mu_g} g(\vy^*)
	\right) {\le \min_{\vxi_\mu\in\partial\widetilde g(\vy^*,\mu)} \left(\|\vxi_\mu\| + \operatorname{dist}
	\left(
	\vxi_\mu,\partial^{\mu_g} g(\vy^*)
	\right)\right)}
	\le
	\operatorname{dist}
	\left(
	\vzero,\partial\widetilde g(\vy^*,\mu)
	\right)
	+
	2\sqrt{L_\Psi\Delta_\mu}.
	\]
	Taking expectations on both sides and using Jensen's inequality, together with the
	definition of an \(\epsilon\)-scaled stationary point in expectation, gives  
	\[
	\begin{aligned}
			\left(
		\mathbb E
		\left[
		\operatorname{dist}
		\left(
		0,\partial^{\mu_g}g(\vy^{(\tau+1)})
		\right)^2
		\right]
		\right)^{1/2} 
		&\le
		\sqrt{
			2\mathbb{E}
			\left[
			\operatorname{dist}
			\left(
			\vzero,\partial\widetilde g(\vy^*,\mu)
			\right)^2
			\right]
		}
		+
		2\sqrt{2L_\Psi\omega(\mu)}  \\
		&\le
		\sqrt{2}\epsilon
		+
		2\sqrt{2L_\Psi\omega(\mu)} .
	\end{aligned}
	\] 
	This proves that \(\vy^*\) is a
	\((\mu_g,\epsilon_g)\)-Goldstein stationary point in expectation.
\end{proof}

 	\subsection{Proof of Theorem~\ref{lem:scaled}}\label{sec:scaled}
	Before we give the proof of Theorem~\ref{lem:scaled}, we present a necessary lemma.
	\begin{lemma}[Theorem 5.5.2 of ~\citep{cui2021modern}]
		\label{lem:regularcondi}
		Under Assumptions~\ref{ass:problemsetup}--\ref{ass:compact1}, the function $ g$ in~\eqref{eq:minmax} is Clarke regular. 
	\end{lemma}
	By Lemma~\ref{lem:regularcondi}, we know that any Clarke stationary point of problem~\eqref{eq:minmax} is also a directional stationary point.
	
	\proof (of Theorem~\ref{lem:scaled})
	From the definition of $\vy^{(k)}$, 
	there exist $\vgamma_{\vy}^{(k)} \in \partial \varphi (\vy^{(k)}) \subset\RR^{m_1}$ and $0<\mu_k\leq \epsilon_k$
	{such that}
	\begin{equation}
		\label{eq:kkgamma}
		\mathbb{E}\left[\left(\dist\left(\vzero, \partial\widetilde{g}( \vy^{(k)},\mu_k)\right)\right)^2\right]= \mathbb{E}\left[\left\| \mathbb{E}_{\vx\sim\mathbb{P}}[\nabla\widetilde{{\Phi}}( \vy^{(k)},\mu_k;\vx)]+  \vgamma_{\vy}^{{(k)}}\right\|^2\right] \leq \epsilon^2_k, \text{ for all }k\ge 0.
	\end{equation}
	Define 
	\begin{align*}
		&A_k(\delta): = \left\{\left\|\mathbb{E}_{\vx\sim\mathbb{P}}[\nabla\widetilde{{\Phi}}( \vy^{(k)},\mu_k;\vx)]+  \vgamma_{\vy}^{{(k)}}\right\|^2 \ge \delta \right\}, \text{ and }\\ &A^c_k(\delta): = \left\{\left\|\mathbb{E}_{\vx\sim\mathbb{P}}[\nabla\widetilde{{\Phi}}( \vy^{(k)},\mu_k;\vx)]+  \vgamma_{\vy}^{{(k)}}\right\|^2 < \delta \right\}.
	\end{align*} 
	By Markov's inequality and~\eqref{eq:kkgamma},  we have for any $\delta>0$,
	\begin{align*}
		\sum_{k=0}^{\infty}\mathrm{Prob}\left(
		A_k(\delta)
		\right) \leq \sum_{k=0}^{\infty}\frac{1}{\delta}
		\mathbb{E}\left[\left\|\mathbb{E}_{\vx\sim\mathbb{P}}[\nabla\widetilde{{\Phi}}( \vy^{(k)},\mu_k;\vx)]+  \vgamma_{\vy}^{{(k)}}\right\|^2\right] 
		\leq \frac{1}{\delta}\sum_{k=0}^{\infty}\epsilon^2_k<+\infty.
	\end{align*}
	Setting $\delta=\frac{1}{t}$ for $t\in\mathbb{N}_{++}$, we apply the Borel-Cantelli Lemma~\citep{shiryaev2016probability} to the sequence of events $\left(A_k(\delta): k\ge 0\right)$, and obtain
	$ 
	\mathrm{Prob}(\limsup_{k\rightarrow\infty} A_k(\frac{1}{t}))=0.
	$
	By the definition of the limit operator, $\omega$ belongs to the event $$\Omega:=\left\{\omega: \lim_{k\rightarrow\infty}\left\|\mathbb{E}_{\vx\sim\mathbb{P}}[\nabla\widetilde{{\Phi}}( \vy^{(k)}(\omega),\mu_k;\vx)]+  \vgamma_{\vy}^{{(k)}}(\omega)\right\|=0\right\},$$
	if and only if
	$$\omega\in \bigcap_{t=1}^{\infty}\bigcup_{k=1}^{\infty}\bigcap_{i=k}^{\infty} A_i^c(\frac{1}{t})=\left(\bigcup_{t=1}^{\infty}\limsup_{k\rightarrow\infty}A_k(\frac{1}{t})\right)^c.$$
	It then follows $$\mathrm{Prob}\left( \Omega\right)=1-\mathrm{Prob}\left(\bigcup_{t=1}^{\infty}\limsup_{k\rightarrow\infty}A_k(\frac{1}{t})\right)\geq 1-\sum_{t=1}^{\infty}\mathrm{Prob}\left(\limsup_{k\rightarrow\infty}A_k(\frac{1}{t})\right)=1.$$
	Combining this with $\mathrm{Prob}\left(  
	\Omega\right)\leq 1$, we derive $\mathrm{Prob}(\Omega)=1$. This implies~\eqref{eq:almost} by the equation in~\eqref{eq:kkgamma}.
	
	Define event $\overline\Omega:= \{\omega: \lim_{k\rightarrow \infty}\vy^{(j_k)}(\omega) = \vy^{*}(\omega)\}$. According to our setting in this theorem, it holds $\mathrm{Prob}(\overline\Omega)=1$, and thus $\mathrm{Prob}(\Omega \cap \overline\Omega)=1$.
	Let $\omega$ belong to the probability-one event on which
	$
	\left\|
	\mathbb E_{\vx\sim\mathbb P}
	[\nabla_\vy\widetilde\Phi(\vy^{(k)}(\omega),\mu_k;\vx)]
	+\vgamma_\vy^{(k)}(\omega)
	\right\|\to0.
	$
	Suppose $\vy^{(j_k)}(\omega)\to \vy^*(\omega)$. Define
	$
	\vv_k(\omega):=
	\mathbb E_{\vx\sim\mathbb P}
	[\nabla_\vy\widetilde\Phi(\vy^{(j_k)}(\omega),\mu_{j_k};\vx)].
	$
	By Lemma~\ref{lem:smoothphi}, $\|\vv_k(\omega)\|\le l_\Psi$, hence
	$\{\vv_k(\omega)\}$ is bounded. Moreover,
	$\vgamma_\vy^{(j_k)}(\omega)=-\vv_k(\omega)+o(1)$ is also bounded.
	Passing to a further subsequence if necessary, we may assume
	$\vv_k(\omega)\to \vv^*(\omega)$ and
	$\vgamma_\vy^{(j_k)}(\omega)\to\vgamma_\vy^*(\omega)$.
	Next, 
	the gradient-consistency result for the smoothing approximation, together with $\mu_{j_k}\downarrow0$ and $\vy^{(j_k)}\to\vy^\star$, yields
	\[
	\operatorname{dist}\left(
	\vv_k,\partial(g-\varphi)(\vy^\star)
	\right)\to0
	\]
	along the selected subsequence. Hence $\vv^\star\in\partial(g-\varphi)(\vy^\star)$.
	Since $g - \varphi$ is Clarke regular at $\vy^*(\omega)$ by Lemma~\ref{lem:regularcondi}, and given that $\lim_{k\to\infty} \mu_{j_k} = 0$, we invoke \citep[Theorem 1]{burke2020subdifferential}. This allows us to conclude that
	$
	\vv^*(\omega)\in \partial (g-\varphi)(\vy^*(\omega)).
	$
	Since
	$\vgamma_\vy^{(j_k)}(\omega)\in\partial\varphi(\vy^{(j_k)}(\omega))$
	and $\partial\varphi$ is outer semicontinuous~\citep[Theorem 8.6]{roc1998var},
	$
	\vgamma_\vy^*(\omega)\in\partial\varphi(\vy^*(\omega)).
	$
	Taking limits in
	$\vv_k(\omega)+\vgamma_\vy^{(j_k)}(\omega)\to0$ gives
	$0=\vv^*(\omega)+\vgamma_\vy^*(\omega)\in\partial g(\vy^*(\omega))$.
	Therefore $\vy^*(\omega)$ is a Clarke stationary point of problem~\eqref{eq:minmax}.
	By Lemma~\ref{lem:regularcondi}, it is also a directional stationary point for all $\omega\in \Omega\cap \overline\Omega$. Since $\mathrm{Prob}(\Omega \cap \overline{\Omega}) = 1$, $\vy^*$ is a directional stationary point of problem~\eqref{eq:minmax} almost surely.
	\endproof

	\subsection{Proof of Lemma~\ref{lem:alg3}}\label{sec:alg3}
	
	\proof{}
	$($a$)$
	From the updating rule~\eqref{eq:updatex2} with $\alpha_k=\frac{1}{L^{(k)}}$, it follows that
	\begin{equation}
		\label{eq:KKT2}
		L^{(k)}\left(\vy^{(k+1)}-\vy^{(k)}\right)+ \mathcal{G}(\vy^{(k)},\mu_{k}) + \vgamma_{\vy}^{(k+1)} =\vzero, 
	\end{equation}
	for some $\vgamma_{\vy}^{(k+1)}\in \partial \varphi(\vy^{(k+1)}).$
	We derive
	\begin{align}
		\notag
		&\mathbb{E}\left[\widetilde{g}( \vy^{(k+1)},\mu_k)- \widetilde{g}( \vy^{(k)},\mu_k) \right]\\
		\notag	\leq& \mathbb{E}\left[\left\langle\mathbb{E}_{\vx\sim\mathbb{P}}[\nabla_{\vy} \widetilde{{\Phi}}( \vy^{(k)},\mu_k;\vx)], \vy^{(k+1)}- \vy^{(k)}\right\rangle\right]+\mathbb{E}\left[\frac{L^{(k)}}{2}\| \vy^{(k+1)}- \vy^{(k)}\|^2\right] \\ & +\mathbb{E}\left[ \varphi(\vy^{(k+1)}) - \varphi(\vy^{(k)})\right] 
		\notag
		\stackrel{\eqref{eq:KKT2}}{=} \mathbb{E}\left[\left\langle -\vgamma_{\vy}^{(k+1)}-L^{(k)}( \vy^{(k+1)}- \vy^{(k)}), \vy^{(k+1)}- \vy^{(k)}\right\rangle\right] \\ &+ \mathbb{E}\left[\frac{L^{(k)}}{2}\| \vy^{(k+1)}- \vy^{(k)}\|^2\right]   \notag +\mathbb{E}\left[\left\langle\mathbb{E}_{\vx\sim\mathbb{P}}[\nabla_{\vy} \widetilde{{\Phi}}( \vy^{(k)},\mu_k;\vx)]-\mathcal{G}(\vy^{(k)},\mu_{k}), \vy^{(k+1)}- \vy^{(k)}\right\rangle\right]  \\ \notag& + \mathbb{E}\left[\varphi(\vy^{(k+1)}) - \varphi(\vy^{(k)})\right]\\
		\notag\leq & -\frac{L^{(k)}}{2}\mathbb{E} \left[\| \vy^{(k+1)}- \vy^{(k)}\|^2\right]  + \mathbb{E}\left[-\left\langle  \vgamma_{\vy}^{{(k+1)}}, \vy^{(k+1)}- \vy^{(k)}\right\rangle   + \varphi(\vy^{(k+1)}) - \varphi(\vy^{(k)})\right]
		\\\notag&\quad\quad\quad\quad\,\,+ \mathbb{E}\left[\frac{L^{(k)}}{4}\|\vy^{(k+1)}- \vy^{(k)}\|^2 + \frac{1}{L^{(k)}}\left\|\mathbb{E}_{\vx\sim\mathbb{P}} [\nabla_{\vy}\widetilde{{\Phi}}_( \vy^{(k)},\mu_k;\vx)]-\mathcal{G}(\vy^{(k)},\mu_{k})\right\|^2\right] \\
		\leq &-\frac{L^{(k)}}{4}\mathbb{E}\left[\| \vy^{(k+1)}- \vy^{(k)}\|^2\right]+ \frac{1}{L^{(k)}}\mathbb{E}\left[\left\| \mathbb{E}_{\vx\sim\mathbb{P}}[\nabla_{\vy}\widetilde{{\Phi}}( \vy^{(k)},\mu_k;\vx)]-\mathcal{G}(\vy^{(k)},\mu_{k})\right\|^2\right],
		\label{eq:b2}
	\end{align}
	where the first inequality comes from the  $L^{(k)}$-smoothness of $\widetilde\Phi(\cdot,\mu;\vx)$ for each $\vx\in\cX$, the second inequality holds by Young inequality, and the last inequality results from 
	the convexity of $\varphi$ and $\vgamma_{\vy}^{(k+1)}\in \partial \varphi(\vy^{k+1})$.
	
	Next, we bound the second term in the right hand side of~\eqref{eq:b2}, 
	\begin{align}
		\notag
		& \mathbb{E}\left[\left\| \mathbb{E}_{\vx\sim\mathbb{P}}[\nabla_{\vy}\widetilde{{\Phi}}( \vy^{(k)},\mu_k;\vx)]-\mathcal{G}(\vy^{(k)},\mu_{k})\right\|^2\right] 
		\\= \notag&  \mathbb{E}\left[\left\| \mathbb{E}_{\vx\sim\mathbb{P}}[\nabla_{\vy}\widetilde{{\Phi}}( \vy^{(k)},\mu_k;\vx)]-\frac{1}{M_k}\sum^{M_k}_{j=1}
		\mathcal{G}_{k_j}(\vy^{(k)},\mu_{k})\right\|^2\right]
		\\
		\notag\leq & 
		2\mathbb{E}\left[\left\| \frac{1}{M_k} \sum_{j=1}^{M_k}\nabla_{\vy}\widetilde{{\Phi}}( \vy^{(k)},\mu_k;\vx_{k_j})-\frac{1}{M_k}\sum^{M_k}_{j=1}
		\mathcal{G}_{k_j}(\vy^{(k)},\mu_{k})
		\right\|^2\right]
		\\ & \notag\quad\quad\quad\quad\quad + 
		2\mathbb{E}\left[\left\| \mathbb{E}_{\vx\sim\mathbb{P}}[\nabla_{\vy}\widetilde{{\Phi}}( \vy^{(k)},\mu_k;\vx)]-\frac{1}{M_k} \sum_{j=1}^{M_k}\nabla_{\vy}\widetilde{{\Phi}} ( \vy^{(k)},\mu_k;\vx_{k_j})
		\right\|^2\right] 
		\\\notag
		\leq & 2\widehat{\epsilon}_k^2 + 2\mathbb{E}\left[\left\| \mathbb{E}_{\vx\sim\mathbb{P}}[\nabla_{\vy}\widetilde{{\Phi}}( \vy^{(k)},\mu_k;\vx)]-\frac{1}{M_k} \sum_{j=1}^{M_k}\nabla_{\vy}\widetilde{{\Phi}} ( \vy^{(k)},\mu_k;\vx_{k_j})
		\right\|^2\right]  \\
		= &  {2\widehat{\epsilon}_k^2 + \frac{2}{M_k}\mathbb{E}\left[\left\| \mathbb{E}_{\vx\sim\mathbb{P}}[\nabla_{\vy}\widetilde{{\Phi}}( \vy^{(k)},\mu_k;\vx)]-\nabla_{\vy}\widetilde{{\Phi}} ( \vy^{(k)},\mu_k;\vx_{k_1})
			\right\|^2\right]} \leq    2\widehat{\epsilon}_k^2 + \frac{8}{M_k}l_{\Psi}^2\leq 4\widehat{\epsilon}_k^2,
		\label{eq:boundg2}
	\end{align}
	where the first inequality uses triangle inequality, the second one yields from~\eqref{eq:gradinetbias2},  the  second equality holds because of 
	the i.i.d. sampling, the third inequality uses 
	$\|\nabla_{\vy} \widetilde{\Phi}(\vy,\mu;\vx)\|\leq l_{\Psi}$ for all $\vx\in\cX$ from Lemma~\ref{lem:smoothphi}(b), and the last one follows by $M_k\ge \lceil 4l_{\Psi}^2 \widehat{\epsilon}_k^{-2}\rceil$.
	
	(b)  
	We deduce that
	\begin{align*}
		&\mathbb{E}\left[\left\|\mathbb{E}_{\vx\sim\mathbb{P}}[\nabla_{\vy}\widetilde{{\Phi}}( \vy^{(k+1)},\mu_k;\vx)]+  \vgamma_{\vy}^{{(k+1)}}\right\|^2\right] \\
		\leq & 2\mathbb{E}\left[\left\|\mathbb{E}_{\vx\sim\mathbb{P}}[\nabla_{\vy}\widetilde{{\Phi}}( \vy^{(k)},\mu_k;\vx)]+  \vgamma_{\vy}^{{(k+1)}}\right\|^2\right]  \\ & +2 \mathbb{E}\left[\left\|\mathbb{E}_{\vx\sim\mathbb{P}}[\nabla_{\vy}\widetilde{{\Phi}}( \vy^{(k)},\mu_k;\vx)]- \mathbb{E}_{\vx\sim\mathbb{P}}[\nabla_{\vy}\widetilde{{\Phi}}( \vy^{(k+1)},\mu_k;\vx)]\right\|^2\right] 
		\\\leq & 2\mathbb{E}\left[\left\|\mathbb{E}_{\vx\sim\mathbb{P}}[\nabla_{\vy}\widetilde{{\Phi}}( \vy^{(k)},\mu_k;\vx)]+  \vgamma_{\vy}^{{(k+1)}}\right\|^2\right]  +2\mathbb{E}\left[(L^{(k)})^2\| \vy^{(k+1)}- \vy^{(k)}\|^2\right]
		\\ &
		\stackrel{\eqref{eq:KKT2}}{=}
		2\mathbb{E} \left[\left\|-L^{(k)}( \vy^{(k+1)}- \vy^{(k)}) + \left(\mathbb{E}_{\vx\sim\mathbb{P}} [\nabla_{\vy}\widetilde{{\Phi}}( \vy^{(k)},\mu_k;\vx)]- \mathcal{G}(\vy^{(k)},\mu_{k})\right)\right\|^2\right]\\ &+2\mathbb{E}\left[(L^{(k)})^2\| \vy^{(k+1)}- \vy^{(k)}\|^2\right]\notag
		\\ \leq & 
		\frac{5(L^{(k)})^2}{2} \mathbb{E}\left[ \| \vy^{(k+1)}- \vy^{(k)}\|^2\right]
		+10 \mathbb{E}\left[\left\|\mathbb{E}_{\vx\sim\mathbb{P}} [\nabla_{\vy}\widetilde{{\Phi}}( \vy^{(k)},\mu_k;\vx)]- \mathcal{G}(\vy^{(k)},\mu_{k})\right\|^2\right] 
		\\ & \quad\quad\quad\quad\quad +2\mathbb{E}\left[(L^{(k)})^2\| \vy^{(k+1)}- \vy^{(k)}\|^2\right]\notag\\
		=&  
		\frac{9 (L^{(k)})^2}{2} \mathbb{E}\left[ \| \vy^{(k+1)}- \vy^{(k)}\|^2\right] 
		+10 \mathbb{E}\left[\left\|\mathbb{E}_{\vx\sim\mathbb{P}}[\nabla_{\vy}\widetilde{{\Phi}}( \vy^{(k)},\mu_k;\vx)]- \mathcal{G}(\vy^{(k)},\mu_{k})\right\|^2\right]
		\\ \le& 18 L^{(k)} \left(\mathbb{E} \left[\widetilde{g}( \vy^{(k)},\mu_k)- \widetilde{g}( \vy^{(k+1)},\mu_k)\right] +  \frac{4}{L^{(k)}}\widehat{\epsilon}_k^2 \right) +40 \widehat{\epsilon}_k^2
		\\ =& 18 L^{(k)} \mathbb{E} \left[\widetilde{g}( \vy^{(k)},\mu_k)- \widetilde{g}( \vy^{(k+1)},\mu_k)\right] +  {112}{}\widehat{\epsilon}_k^2,
	\end{align*}
	where we use triangle inequality in the first inequality, Lemma~\ref{lem:smoothphi}(c) in the second one, and  Young's inequality in the third one. To obtain the last inequality, we have used $$\mathbb{E}[\|\mathbb{E}_{\vx\sim\mathbb{P}}[\nabla_{\vy}\widetilde{{\Phi}}( \vy^{(k)},\mu_k;\vx)]-\mathcal{G}(\vy^{(k)},\mu_{k})\|^2] \le 4 \widehat{\epsilon}_k^2$$ from~\eqref{eq:boundg2}, and
	$\frac{L^{(k)}}{4} \mathbb{E} \left[\| \vy^{(k+1)}- \vy^{(k)}\|^2\right]\leq -\mathbb{E}\left[\widetilde{g}( \vy^{(k+1)},\mu_k)- \widetilde{g}( \vy^{(k)},\mu_k)\right]   + \frac{4}{ L^{(k)}}\widehat{\epsilon}_k^2$
	from
	\eqref{eq:funcgap2}.   
	The proof is then completed.
	\endproof

	\subsection{Proof of Theorem~\ref{thm:scaled2}}\label{sec:scaled2}
	\proof 
	Summing up~\eqref{eq:kktregu2} for all $k=k_1, k_1+1,\ldots,K-1$, and noticing $L^{(k)}=C_2/\mu_k$, we have 
	\begin{align}
		\notag
		&\sum_{k=k_1}^{K-1} \frac{\mu_k}{C_2}\mathbb{E}\left[\left(\dist\left(\vzero, \partial\widetilde{g}( \vy^{(\tau+1)},\mu_\tau)\right)\right)^2\right]	=		\sum_{k=k_1}^{K-1} \frac{\mu_k}{C_2}\mathbb{E}\left[\left(\dist\left(\vzero, \partial\widetilde{g}( \vy^{(k+1)},\mu_k)\right)\right)^2\right]  \\ \leq & 18 \sum_{k=k_1}^{K-1}  \mathbb{E} \left[\widetilde{g}( \vy^{(k)},\mu_k)- \widetilde{g}( \vy^{(k+1)},\mu_k)\right] +  \sum_{k=k_1}^{K-1} \frac{112 \mu_k}{ C_2} \widehat{\epsilon}_k^2.
		\label{eq:sumgradient}
	\end{align}
	For the first term in the right hand side of~\eqref{eq:sumgradient}, we have 
	\begin{align}
		\notag
		& \sum_{k=k_1}^{K-1} \mathbb{E} \left[\widetilde{g}( \vy^{(k)},\mu_k)- \widetilde{g}( \vy^{(k+1)},\mu_k)\right]  \\  
		\notag= &   \mathbb{E} \left[\widetilde{g}( \vy^{(k_1)},\mu_{k_1})- \widetilde{g}( \vy^{(K)},\mu_{K-1}) \right] +\sum_{k=k_1}^{K-2} \mathbb{E} \left[\widetilde{g}( \vy^{(k+1)},\mu_{k+1})- \widetilde{g}( \vy^{(k+1)},\mu_{k})\right]   \\
		\leq &   \mathbb{E} \left[\widetilde{g}( \vy^{(k_1)},\mu_{k_1})- \widetilde{g}( \vy^{(K)},\mu_{K-1})\right] +\sum_{k=k_1}^{K-2} \omega\left(\mu_{k}- \mu_{k+1}\right), 
		\label{eq:difffunc}
	\end{align}
	where {the first inequality follows from Lemmas~\ref{lem:distancemufinite}--\ref{lem:distancemu}}.
	Substituting~\eqref{eq:difffunc} into~\eqref{eq:sumgradient} and multiplying each side by $\frac{C_2}{\sum_{k=k_1}^{K-1}{\mu_k}}$ yields 
	\eqref{eq:expectationsta}. 
	\endproof
	
	\subsection{Proof  of Corollary~\ref{cor:main1}} \label{sec:cor}

	\proof
 \textnormal{(a)}  We look at the three terms on the
	right-hand side of~\eqref{eq:expectationsta}. Since \(\mu_k=\epsilon\) and
	\(k_1=0\), we have
	$
	\sum_{k=k_1}^{K-1}\mu_k=K\epsilon .
	$
	The first term satisfies
	\[
	\frac{
		18C_2
		\mathbb E
		\left[
		\widetilde g(\vy^{(0)},\epsilon)
		-
		\widetilde g(\vy^{(K)},\epsilon)
		\right]
	}{
		K\epsilon
	}
	\le
	\frac{
		18C_2
		\left(
		\widetilde g(\vy^{(0)},\epsilon)
		-
		\min_{\vy}\widetilde g(\vy,\epsilon)
		\right)
	}{
		K\epsilon
	}
	\le
	\frac{\epsilon^2}{2},
	\]
	where the last inequality follows from the definition of \(K\). The second term
	is zero because \(\mu_k=\epsilon\) for all \(k\in[K]\). The third term is
	$
	\frac{
		112\sum_{k=0}^{K-1}\epsilon\widehat\epsilon_k^2
	}{
		K\epsilon
	}
	=
	112\left(\frac{\epsilon}{16}\right)^2
	\le
	\frac{\epsilon^2}{2}.
	$
	Therefore, the right-hand side of~\eqref{eq:expectationsta} is at most
	\(\epsilon^2\). Hence
	$
	\mathbb E
	\left[
	\left(
	\operatorname{dist}
	\left(
	0,\partial\widetilde g(\vy^{(\tau+1)},\mu_\tau)
	\right)
	\right)^2
	\right]
	\le
	\epsilon^2 .
	$
	Since \(\tau\) is sampled from \(\{0,1,\ldots,K-1\}\), there exists some
	\(0\le k<K\) satisfying the same bound.  
	
	  \textnormal{(b)}  
	Since \(\mu_k=\mu_g^2\) and \(k_1=0\), we have
	$
	\sum_{k=0}^{K-1}\mu_k=K\mu_g^2 .
	$
	The first term on the right-hand side of~\eqref{eq:expectationsta} satisfies
	\[
	\frac{
		18C_2
		\mathbb E
		\left[
		\widetilde g(\vy^{(0)},\mu_g^2)
		-
		\widetilde g(\vy^{(K)},\mu_g^2)
		\right]
	}{
		K\mu_g^2
	}
	\le
	\frac{
		18C_2
		\left(
		\widetilde g(\vy^{(0)},\mu_g^2)
		-
		\min_{\vy}\widetilde g(\vy,\mu_g^2)
		\right)
	}{
		K\mu_g^2
	}
	\le
	\frac{\epsilon_g^2}{2}.
	\]
	The second term is again zero. The
	third term is
	$
	\frac{
		112\sum_{k=0}^{K-1}\mu_g^2\widehat\epsilon_k^2
	}{
		K\mu_g^2
	} 
	\le
	\frac{\epsilon_g^2}{2}.
	$
	Therefore, Theorem~\ref{thm:scaled2} yields
	$
	\mathbb E
	\left[
	\left(
	\operatorname{dist}
	\left(
	0,\partial\widetilde g(\vy^{(\tau+1)},\mu_g^2)
	\right)
	\right)^2
	\right]
	\le\epsilon_g^2 .
	$
	
	By the proof of Lemma~\ref{lem:smoothed-to-goldstein}, and \(\Delta_{\mu_g^2}\le\omega(\mu_g^2) \),
	we obtain 
	\[
	\begin{aligned}
		\left(
		\mathbb E
		\left[
		\operatorname{dist}
		\left(
		0,\partial^{\mu_g}g(\vy^{(\tau+1)})
		\right)^2
		\right]
		\right)^{1/2}
		 & \le
		\left(
		2\mathbb E
		\left[
		\left(
		\operatorname{dist}
		\left(
		0,\partial\widetilde g(\vy^{(\tau+1)},\mu_g^2)
		\right)
		\right)^2
		\right]
		\right)^{1/2}
		+
		2\sqrt{2L_\Psi\Delta_{\mu_g^2}}
		\\
		& \le
		\sqrt{2}\epsilon_g +  2\sqrt{2L_\Psi\omega(\mu_g^2) }.
	\end{aligned}
	\]
	Thus the randomly sampled output \(\vy^{(\tau+1)}\) is a
	\(( \sqrt{\frac{2\omega(\mu_g^2)}{L_\Psi}},\sqrt{2}\epsilon_g +  2\sqrt{2L_\Psi\omega(\mu_g^2)})\)-Goldstein stationary point in expectation. Hence at
	least one iterate \(\vy^{(k+1)}\), \(0\le k<K\), satisfies the same bound. This
	proves part~\textnormal{(b)}.
	\endproof

\end{document}